\documentclass[12pt]{article}

\usepackage[utf8]{inputenc}
\usepackage[T1]{fontenc}
\usepackage{lmodern}
\usepackage[english]{babel}
\usepackage{amsmath,mathtools,amsthm,amssymb,amsfonts}
\usepackage{color}
\usepackage{comment}

\usepackage{hyperref}
\usepackage{mathrsfs}
\usepackage{graphicx}
\usepackage{float}
\usepackage{enumitem}
\usepackage{tikz} 
\usepackage{esint}
\usepackage{pgfplots}
\usepackage{esvect}
\usepackage{pdfpages}
\usepackage{nicefrac}
\usepackage{leftidx, tensor}
\usepackage{cleveref}
\usepackage{accents}
\usetikzlibrary{shapes,arrows}
\usetikzlibrary{calc}
\usetikzlibrary{shapes.geometric}
\usetikzlibrary{shapes.arrows}
\usetikzlibrary{arrows.meta}
\usetikzlibrary{decorations.markings}
\usetikzlibrary{fit}
\usetikzlibrary{patterns}
\usetikzlibrary{hobby}
\usepgfplotslibrary{patchplots}
\pgfplotsset{compat=1.11}

\newcommand{\definedas}{\mathrel{\raise.095ex\hbox{\rm :}\mkern-5.2mu=}}

\newcommand{\R}{\mathbb{R}}

\newcommand{\Z}{\mathbb{Z}}

\newcommand{\Sbb}{\mathbb{S}}

\renewcommand{\d}{\,\mathrm{d}}

\newcommand{\ul}[1]{\underline{#1}}

\newcommand{\btr}[1]{\left\vert#1\right\vert}
\newcommand{\newbtr}[1]{\vert#1\vert}
\newcommand{\norm}[1]{\btr{\btr{#1}}}

\newcommand{\spann}[1]{\left\langle#1\right\rangle}

\newcommand{\Ric}{\mathrm{Ric}}
\newcommand{\scal}{\mathrm{R}}
\newcommand{\Riem}{\operatorname{Rm}}

\newcommand{\two}{\operatorname{II}}
\newcommand{\tr}{\text{tr}}
\newcommand{\dive}{\operatorname{div}}

\newcommand{\newnorm}[1]{\vert\vert#1\vert\vert}
\newcommand{\hatgamma}{\widehat{\gamma}}
\newcommand{\tildegamma}{\widetilde{\gamma}}

\theoremstyle{plain}
\newtheorem{thm}{Theorem}[section]
\newtheorem{prop}[thm]{Proposition}

\newtheorem{lem}[thm]{Lemma}

\theoremstyle{definition}
\newtheorem{defi}[thm]{Definition}

\newtheorem{bem}[thm]{Remark}
\newtheorem{kor}[thm]{Corollary}

\begin{document}
		\begin{center}\LARGE On the uniqueness of surfaces of constant spacetime mean curvature in asymptotically Schwarzschildean lightcones\end{center}
		\vspace{0.5cm}
		\begin{center}
			{\large Klaus Kr\"oncke and Markus Wolff}
		\end{center}
		\vspace{0.4cm}
		\begin{abstract}
			In this paper, we address the uniqueness of surfaces of constant spacetime mean curvature in an asymptotically Schwarzschildean lightcone of mass $m>0$. We prove that there exists a unique asymptotically flat foliation by surfaces of constant spacetime mean curvature for a fairly generic notion of asymptotic flatness. This foliation has Bondi energy $m$ and vanishing Bondi linear momentum. The authors have already established the existence of such a foliation in previous work \cite{kroenckewolff}, but proven uniqueness only in a very restrictive class of surfaces. Although this restrictive class of surfaces was necessary for the construction, here we show that the foliation is a posteriori unique under significantly weaker assumptions.
		\end{abstract}

	\section{Introduction}
	
		In their previous work, the authors constructed an asymptotic foliation of an asymptotically Schwarzschildean lightcone by surfaces of constant spacetime mean curvature \cite[Theorem 5.2]{kroenckewolff}. We recall that for a spacelike, codimension $2$ surface $(\Sigma,\gamma)$ in an ambient spacetime $(\overline{M},\overline{g})$, the \emph{spacetime mean curvature} $\mathcal{H}^2$ of $\Sigma$ as introduced by Cederbaum--Sakovich \cite{cederbaumsakovich} is defined as the Lorentzian length of the codimension $2$ mean curvature vector $\vec{\mathcal{H}}$, i.e.,
		\[
		\mathcal{H}^2=\overline{g}(\vec{\mathcal{H}},\vec{\mathcal{H}}).
		\]
		We say $(\Sigma,\gamma)$ is a \emph{surface of constant spacetime mean curvature} (STCMC) if $\mathcal{H}^2$ is constant along $\Sigma$. Although we will only consider surfaces where $\mathcal{H}^2$ is strictly positive here, we note that $\vec{\mathcal{H}}$ has a priori no fixed causal character and trapped surfaces, where $\mathcal{H}^2<0$ on $\Sigma$, naturally occur in the context of general relativity. Moreover, throughout this article we will restrict our attention to surfaces that arise as spacelike cross sections of a given null hypersurface. We recall that a null hypersurface $\mathcal{N}$ in a spacetime $(\overline{M},\overline{g})$ carries a degenerate induced metric and is ruled by null geodesics. In particular, null hypersurfaces model the  trajectories of light that is emanating off of a given light source, such as a far away celestial object. Thus, null hypersurfaces arise naturally in physical considerations and conserved quantities such as energy, linear momentum, angular momentum, and center of mass are studied similarly as for asymptotically flat and asymptotically hyperbolic initial data sets. 
		
		In the context of center of mass, which is generally much more subtle than energy and linear momentum, asymptotic foliations by surfaces of prescribed curvature and geometric notions of center of mass have been widely studied. This goes back to an idea first proposed by Huisken and Yau \cite{huiskenyau}. See for example \cite{cederbaumsakovich, eichmairkoerber1, huiskenyau, huang, lammmetzgerschulze, metzger, nerz, nevestian2,  tenan1, tenan2, ye} and the references given therein for a non-exhaustive list. Many of these results were first established for CMC surfaces. In \cite{cederbaumsakovich}, Cederbaum and Sakovich introduce STCMC surfaces as a natural generalization to CMC surfaces in the context of general relativity and establish the existence of an asymptotic foliation of asymptotically Euclidean initial data sets by STCMC surfaces. See also Tenan \cite{tenan1}. The corresponding result for asymptotically hyperboloidal initial data sets close to the Schwarzschild-Anti\,deSitter space has recently been established by Tenan \cite{tenan2}. In this context, our result in \cite{kroenckewolff} establishes the existence of such a foliation in the null setting, and we note that the spacetime mean curvature in fact presents itself as the natural analogue to the concept of mean curvature in the null setting. Here, we also mention recent work by Lambert and Scheuer \cite{lambertscheuer} which establish the existence of a (local) foliation of a null hypersurface by STCMC surfaces close to a stable marginally outer trapped surface.
		
		To define a center of mass in a well-defined manner, the asymptotic foliation by surfaces of prescribed curvature necessarily has to satisfy a suitable uniqueness property. Following a general strategy, the leaves of the foliation under consideration are usually both constructed and shown to be unique within an a priori class of surfaces. Closely motivated by the original work of Huisken and Yau, the authors considered an a priori class in \cite{kroenckewolff} that seems much more restrictive in the null setting compared to its Riemannian counterpart. In particular, the Bondi linear momentum has to vanish along any foliation within the a priori class. We note however, that this restrictive class is necessary for the construction in \cite{kroenckewolff} and that (Bondi-) energy and linear momentum transform equivariantly under Lorentz transformations, i.e., this restriction can be understood as a necessary gauge condition. In fact, an analogous gauge of balanced coordinates seems necessary for the construction of CMC and STCMC foliations in the asymptotically hyperbolic and hyperboidal case. See \cite{cederbaumcortiersakovich, nevestian1, nevestian2, tenan2}. We also refer the interested reader to a more detailed discussion in \cite[Section 6]{kroenckewolff}. Here, we simply observe that this gauge condition can be understood as fixing the time axis (at infinity) to be given in direction of the Bondi energy-momentum vector. From this point of view, considering notions of center of mass only in this restrictive setting appears to be justified both from a geometric and physical perspective. Motivated by special relativity, the authors propose a notion of center of mass in \cite{kroenckewolff} under the assumption that the linear momentum vanishes. In this sense, the conditional uniqueness of the STCMC foliation should already be sufficient for applications to center of mass. Again, we refer the interested reader to \cite[Section 6]{kroenckewolff} for details.
		
		We plan to closely investigate the proposed center of mass in subsequent work, but aim to answer a naturally arising tangential question here instead: Are the conditions necessary for the construction generically satisfied under the STCMC condition? In other words, if one expects STCMC surfaces to be optimally centered, does an asymptotic foliation by STCMC surfaces necessarily have vanishing linear momentum? Here, we give an affirmative answer within a significantly weaker a priori class of surfaces. We state a simplified version of our main theorem:
		
		\begin{thm}\label{thm_intro}
			Let $\mathcal{N}$ be an asymptotically Schwarzschildean lightcone of mass $m>0$. Then there exists a unique, asymptotically flat background foliation by surfaces of constant spacetime mean curvature. Moreover, the foliation has Bondi energy $m$ and vanishing Bondi linear momentum.
		\end{thm}
		
		We refer the reader to Theorem \ref{thm_main2} for a precise statement. We note that we require some comparability in $C^2$ to a family of (boosted) round spheres, so our result likely does not apply to general asymptotic flat background foliations. However, as one always heuristically expects an asymptotically flat background foliation to converge (upon rescaling) to a family of boosted spheres at infinity, our additional assumptions are satisfied for a fairly generic (strong) notion of asymptotic flatness. More precisely, we establish the uniqueness of STCMC surfaces (Theorem \ref{thm_main1}) within a significantly weaker a priori class compared to \cite{kroenckewolff}. While the a priori class is quite technical, we note that it naturally contains spacelike cross sections close to a boosted sphere, where the boost vector $\vec{a}$ is in fact allowed to diverge to infinity at a precisely controlled rate. See Lemma \ref{lem_classS_observation}. The a priori class moreover significantly simplifies in spherical symmetry, in particular in the exact Schwarzschild lightcone. We note however that the uniqueness of STCMC surfaces in the exact Schwarzschild lightcone (without any additional assumptions) has already been established by Chen--Wang \cite{chenwang}. On the other hand, we expect the uniqueness of STCMC surfaces without any additional assumptions to be false in general as counterexamples should arise by analogy to the work of Brendle--Eichmair \cite{brendleeichmair}.
		
		Our approach is strongly motivated by the work of Huisken--Yau \cite{huiskenyau} on the uniqueness of stable CMC surfaces in an asymptotically Schwarzschildean Riemannian manifold. See also \cite{chodosheichmair, eichmairkoerbermetzgerschulze, ma, qingtian} and the references given therein for other related uniqueness results. As in \cite{huiskenyau}, we show that an STCMC surface within the weaker a priori class (Definition \ref{defi_classS}) must lie in the stronger a priori class as in \cite[Definition 3.1]{kroenckewolff}. By the uniqueness of STCMC surfaces established in \cite[Proposition 3.15]{kroenckewolff}, Theorem \ref{thm_main1} follows. The main strategy to establish this inclusion is to derive some preliminary pointwise estimates for $\newbtr{\accentset{\circ}{A}}$ and $\newbtr{\nabla\accentset{\circ}{A}}$, where $\accentset{\circ}{A}$ denotes the trace-free part of the so-called \emph{scalar second fundamental form} $A$. Although these initial estimates are far from the desired ones, they are at a critical rate to employ quantitative estimates for conformally round surfaces (Corollary \ref{kor_quantitativecontrol}) which subsequently yield the improved estimates. The crucial, initial bound on $\newbtr{\accentset{\circ}{A}}$ is derived from a Stampacchia iteration, using a null Simons' identity (Proposition \ref{prop_prelim_nullsimon})  and a suitable Sobolev inequality (Definition \ref{defi_sobolev}).\\\\
		
		This paper is structured as follows: In Section \ref{sec_prelim} we introduce some necessary prerequisites about null hypersurfaces and the geometry of codimension $2$ surfaces, define our asymptotic assumptions and compare those to the notion of asymptotic flat background foliations. In Section \ref{section_sobolev} we formulate a Sobolev inequality with respect to the spacetime mean curvature and investigate under which conditions it holds with a uniformly controlled constant. We prove our main theorem in Section \ref{sec_uniqueness}. We collect some more technical statements and estimates in the Appendix, such as a contracted version of the null Simons' identity and curvature identities in our setting. We further give an outline of a proof for a quantitative estimate for almost round surfaces which we require for our analysis.\\
		
		\begin{center}
			\textbf{\large Acknowledgments}\\
		\end{center}
		K.K. is supported by the Verg foundation. The research of M.W. is funded by the Austrian Science Fund (FWF) 10.55776/ESP1094725.

	\section{Preliminaries}\label{sec_prelim}
	\subsection{Null Geometry}\label{subsec_prelim_nullgeometry}
	
	We give a brief overview of the relevant concepts in null geometry and the geometry of codimension $2$ surfaces. We refer the interested reader to \cite[Section 2.1]{kroenckewolff} for a more detailed summary and the references given therein for further details.
	
	We say a smooth, oriented hypersurface $\mathcal{N}$ in an ambient spacetime $(\overline{M},\overline{g})$ (of dimension $4$) is a \emph{null hypersurface} if $\mathcal{N}$ carries a degenerate induced metric. In particular, there exists a smooth null vector field $\ul{L}\in \Gamma(T\mathcal{N})$ such that $T_p\mathcal{N}=(\ul{L}_p)^\perp\subset T_p\overline{M}$ for any $p\in\mathcal{N}$. In this sense, $\ul{L}$ is both tangential and normal to $\mathcal{N}$ and we moreover observe that
	\begin{align}\label{eq_surfgrav}
		\overline{\nabla}_{\ul{L}}\ul{L}=\kappa\ul{L}
	\end{align}
	for a smooth function $\kappa$ on $\mathcal{N}$, where $\overline{\nabla}$ denotes the Levi-Civita connection of $(\overline{M},\overline{g})$. Hence, the integral curves of $\ul{L}$ are null pre-geodesics, and we call $\ul{L}$ a \emph{null generator} of $\mathcal{N}$.
	
	In the following, we consider spacelike codimension $2$ submanifolds $(\Sigma,\gamma)$ of $(\overline{M},\overline{g})$ as (spacelike) cross sections of $\mathcal{N}$, i.e., $\Sigma\subset \mathcal{N}$ and we always assume that any integral curve of $\ul{L}$ intersects $\Sigma$ transversally and exactly once. As $T\Sigma\subset T\mathcal{N}$, we have $\ul{L}\in\Gamma(T^\perp\Sigma)$ and there exists a uniquely determined null vector field $L$ such that $\{\ul{L},L\}$ forms a null frame of $\Gamma(T^\perp\mathcal{N})$ with $\overline{g}(\ul{L},L)=2$. We recall that the \emph{null second fundamental forms} of $\Sigma$ with respect to $\{\ul{L},L\}$ are defined as
	\begin{align*}
		\underline{\chi}(X,Y):=\overline{g}(\overline{\nabla}_X\ul{L},Y),\qquad\chi(X,Y):=\overline{g}(\overline{\nabla}_XL,Y),
	\end{align*}
	where $X,Y\in\Gamma(T\Sigma)$, and that the \emph{null expansions} of $\Sigma$ with respect to $\{\ul{L},L\}$ are defined as
	\begin{align*}
		\ul{\theta}:=\tr_\gamma\ul{\chi},\qquad
		\theta:=\tr_\gamma\chi.
	\end{align*}
	We observe that the vector-valued second fundamental form $\vec{\two}$ and codimension $2$ mean curvature vector $\vec{\mathcal{H}}$ of $(\Sigma,\gamma)$ in $(\overline{M},\overline{g})$ can be written as
	\begin{align*}
		\vec{\two}=-\frac{1}{2}\chi\ul{L}-\frac{1}{2}\ul{\chi}L,\qquad
		\vec{\mathcal{H}}=-\frac{1}{2}\theta\ul{L}-\frac{1}{2}\ul{\theta}L
	\end{align*}
	with respect to the null frame $\{\ul{L},L\}$. In particular, the \emph{spacetime mean curvature} $\mathcal{H}^2$ of $\Sigma$, defined as the Lorentzian length of $\vec{\mathcal{H}}$ \cite{cederbaumsakovich}, is given by
	\[
	\mathcal{H}^2:=\overline{g}(\vec{\mathcal{H}},\vec{\mathcal{H}})=\ul{\theta}\theta.
	\]
	Further, the \emph{connection-$1$ form} of $\Sigma$ with respect to $\{\ul{L},L\}$ is defined as
	\[
	\zeta(X):=\frac{1}{2}\overline{g}(\overline{\nabla}_X\ul{L},L),
	\]
	where $X\in\Gamma(T\Sigma)$. In addition, if $\ul{\theta}$ is non-vanishing along $\Sigma$, we define the \emph{scalar second fundamental form} $A$ of $\Sigma$ as
	\begin{align}\label{eq_prelim_defA}
		A:=\ul{\theta}\chi,
	\end{align}
	and \emph{torsion} $\tau$ of $\Sigma$ as
	\begin{align}\label{eq_prelim_deftau}
		\tau:=\zeta-\d\ln(\ul{\theta}),
	\end{align}
	see also \cite[Section 3]{wolff1} and \cite[Section 1]{roesch}, respectively. We observe that the quantities $\mathcal{H}^2=\tr_\gamma A$, $A$, $\tau$ are independent of the choice of null frame $\{\ul{L},L\}$ (if well-defined), and thus independent of the choice of null generator $\ul{L}$ for a spacelike cross section of $\mathcal{N}$.

	We recall the Gauss equations with respect to a null frame $\{\ul{L},L\}$:
	\begin{prop}[{\cite[Proposition 4.24]{wolff_thesis}}]\label{prop_prelim_nullgauss}
		Let $(x^i)$ denote local coordinates of $(\Sigma,\gamma)$. Then
		\begin{align}
			\overline{\Riem}_{ijkl}&=\Riem_{ijkl}-\frac{1}{2}\chi_{jl}\ul{\chi}_{ik}-\frac{1}{2}\ul{\chi}_{jl}\chi_{ik}+\frac{1}{2}\chi_{jk}\ul{\chi}_{il}+\frac{1}{2}\ul{\chi}_{jk}\chi_{il},\label{eq_prelim_gauss1}\\
			\overline{\Ric}_{ik}-\frac{1}{2}\overline{\Riem}_{i\ul{L}kL}-\frac{1}{2}\overline{\Riem}_{iLk\ul{L}}&=\Ric_{ik}-\frac{1}{2}\theta\ul{\chi}_{ik}-\frac{1}{2}\ul{\theta}\chi_{ik}+\frac{1}{2}(\chi\cdot\ul{\chi})_{ik}+\frac{1}{2}(\ul{\chi}\cdot\chi)_{ik},\label{eq_prelim_gauss2}\\
			\overline{\scal}-2\overline{\Ric}(L,\ul{L})+\frac{1}{2}\overline{\Riem}(\ul{L},L,L,\ul{L})&=\operatorname{R}-\mathcal{H}^2+\newbtr{\vec{\two}}^2.\label{eq_prelim_gauss3}
		\end{align}
	\end{prop}
	Here, $\overline{\Riem}$, $\overline{\Ric}$, $\overline{\scal}$, and $\Riem$, $\Ric$, $\scal$ denote the Riemann curvature tensor, Ricci tensor, and scalar curvature of $(\overline{M},\overline{g})$ and $(\Sigma,\gamma)$, respectively. Here, we use the following conventions for the Riemann curvature tensor $\Riem$, Ricci curvature $\Ric$, and scalar curvature $\scal$ of a semi-Riemannian manifold $(M,g)$:
	\begin{align*}
		\Riem(X,Y,W,Z)&=g(\nabla_X\nabla_YZ-\nabla_Y\nabla_XZ-\nabla_{[X,Y]}Z,W),\\
		\Ric(X,Y)&=\tr_g\Riem(X,\cdot,Y,\cdot),\\
		R&=\tr_g \Ric.
	\end{align*}
	
	We further recall the Codazzi Equation for $A$ (if well-defined), see \cite[Remark 2.3]{kroenckewolff}:
	\begin{prop}\label{prop_prelim_codazziA}
		Let $(x^i)$ denote local coordinates of $(\Sigma,\gamma)$. Then
		\begin{align*}
			\nabla_iA_{jk}-\nabla_jA_{ik}&=\ul{\theta}\overline{\Riem}_{ijkL}+\tau_jA_{ik}-\tau_iA_{jk}.
		\end{align*}
	\end{prop}
	
	Using the above Codazzi Equation for $A$, we recall the null Simons' identity established in \cite[Proposition 2.4]{kroenckewolff}. See also \cite{anderssonmetzger} for a similar notion of Simons' identity.
	\begin{prop}[Null Simons' Identity]\label{prop_prelim_nullsimon}
		\begin{align*}
			\nabla_k\nabla_l A_{ij}
			=&\nabla_i\nabla_j A_{kl}+\Riem_{kijm}A^m_l+\Riem_{kilm}A^m_j
			\\&
			+\tau_j\nabla_iA_{kl}+\tau_i\nabla_kA_{jl}-\tau_k\nabla_iA_{jl}-\tau_l\nabla_kA_{ij}
			\\	&
			+\nabla_i\tau_jA_{kl}+\nabla_k\tau_iA_{lj}-\nabla_i\tau_kA_{jl}-\nabla_k\tau_lA_{ij}
			\\		&
			+\nabla_i(\overline{\Riem}_{kjl(\ul{\theta}L)})+\nabla_k(\overline{\Riem}_{lij(\ul{\theta}L)}).
		\end{align*}
	\end{prop}
	\begin{bem}\label{bem_prelim_contractedNullSimons}
		Using some algebraic manipulations, we find that taking a trace yields the following \emph{contracted Null Simons' identity} for a $2$-dimensional STCMC surface:
		\begin{align*}
			\Delta\accentset{\circ}{A}_{ij}
			=&\,-2\tau^k\nabla_k\accentset{\circ}{A}_{ij}\\
			&\,+\left(\frac{1}{2}\mathcal{H}^2+\spann{\ul{\theta}^{-1}\accentset{\circ}{\ul{\chi}},\accentset{\circ}{A}}-\left(\overline{R}-3\overline{\Ric}(\ul{L},\ul{L})+\overline{\Riem}(\ul{L},L,L,\ul{L})\right)-\dive\tau-\btr{\tau}^2\right)\accentset{\circ}{A}_{ij}\\
			&\,+\left(\left(\overline{\Ric}_{jm}-\frac{1}{2}\overline{\Riem}_{\ul{L}jLm}-\frac{1}{2}\overline{\Riem}_{Lj\ul{L}m}\right)\accentset{\circ}{A}^m_i-\overline{\Riem}_{ikjm}\accentset{\circ}{A}^{mk}\right)+\left(\nabla_k\tau_i-\nabla_i\tau_k\right)\accentset{\circ}{A}^k_j\\
			&\,+\left(\overline{\Riem}_{k\ul{L}jL}-\overline{\Riem}_{j\ul{L}kL}\right)\accentset{\circ}{A}_i^k-\frac{1}{2}\left(\Riem_{ki\ul{L}L}\accentset{\circ}{A}^k_j+\overline{\Riem}_{kl\ul{L}L}\accentset{\circ}{A}_i^k\right)\\
			&\,+\frac{1}{2}\mathcal{H}^2\left(\left(\nabla_i\tau_j+\nabla_j\tau_i+2\tau_i\tau_j\right)-\left(\dive\tau+\btr{\tau}^2\right)\gamma_{ij}\right)\\
			&\,-\frac{1}{2}\mathcal{H}^2\left(\left(\overline{\Riem}_{i\ul{L}jL}+\overline{\Riem}_{j\ul{L}iL}\right)-\left(\overline{\Ric}(\ul{L},L)-\frac{1}{2}\overline{\Riem}(\ul{L},L,L,\ul{L})\right)\gamma_{ij}\right)\\
			&\,+\gamma^{kl}\left(\overline{\nabla}_i\overline{\Riem}_{kjl(\ul{\theta}L)}-\overline{\nabla}_k\overline{\Riem}_{ilj(\ul{\theta}L)}\right)+\frac{1}{2}\left(\overline{\Riem}_{i(\ul{\theta}L)j(\ul{\theta}L)}-\overline{\Ric}(\ul{\theta}L,\ul{\theta}L)\ul{\theta}^{-1}\ul{\chi}_{ij}\right).
		\end{align*}
		See Appendix \ref{appendix_nullsimon} for a brief outline of the proof.
	\end{bem}
	
	As in \cite{kroenckewolff}, we will identify objects on $\mathcal{N}$ with (families of) objects on the standard sphere $\Sbb^2$. In particular, we will perform computations and formulate our decay assumptions on $\Sbb^2$. We briefly recall these relations and our conventions for a given smooth null hypersurface, and refer the interested reader to \cite[Section 2.1]{kroenckewolff} for more details: Recall that for a fixed choice of null generator $\ul{L}$ and spacelike cross section $S$, $\mathcal{N}$ admits a (local) \emph{background foliation} given by a diffeomorphism
	\[
	\Phi_{{\ul{L},S}}\colon (r_1,r_2)\times S\to \mathcal{N}, (s,q)\mapsto \gamma^{\ul{L}}_q(s),
	\]
	where $\gamma^{\ul{L}}_q$ denotes the integral curve of $\ul{L}$ starting at $q\in S\subset \mathcal{N}$, and we always assume that $\Phi_{{\ul{L},S}}$ is surjective for convenience. In the following, we will always assume that there exists a diffeomorphism $\phi\colon \Sbb^2\to S$, which gives rise to a diffeomorphism
	\[
	\Phi_{\ul{L}}:=\Phi_{{\ul{L},S}}\circ(\mathrm{id},\phi)\colon (r_1,r_2)\times\Sbb^2\to\mathcal{N},\, (s,\vec{x})\mapsto p.
	\]
	Observe that for a given spacelike cross section $\Sigma$, we can now uniquely identify $\Sigma$ as the graph of a function $\omega\colon \Sbb^2\to (r_1,r_2)$ in the sense that 
	\[\Sigma=\{s=\omega\}=\{p\in\mathcal{N}\,\vert\, p=\Phi_{\ul{L}}(\omega(\vec{x}),\vec{x})\}.
	\]
	We will write $\Sigma=\Sigma_\omega$ and denote the induced metric on $\Sigma$ as $\gamma_{\omega}$.
	
	We denote the leaves of the background foliation by $S_r:=\{s=r\}$, and denote all objects defined on $S_r$ via a subscript $r$. In particular, we observe that the null vector field $L_r\in\Gamma(T^\perp\Sigma_r)$, null second fundamental forms $\ul{\chi}_r$, $\chi_r$, null expansions $\ul{\theta}_r$, $\theta_r$, and connection $1$-form $\zeta_r$ extend to smooth objects on $\mathcal{N}$. Recall that for any spacelike cross section $\Sigma$, $p\in\Sigma$, there exists a diffeomorphism
	\[
	T_pS_{r=\omega(\vec{x})}\to T_p\Sigma,\qquad V\mapsto \widetilde{V}:=V+V(\omega)\ul{L},
	\]
	and we note that
	\[
	L_p=(L_{r=\omega(\vec{x})})_p+\btr{\nabla\omega}_{\gamma_\omega}\ul{L}-2\nabla\omega(p),
	\]
	where $L$ denotes the unique null vector field in $\Gamma(T^\perp\Sigma)$ such that $\overline{g}(\ul{L},L)=2$, and $\nabla\omega$ the gradient of $\omega$ on $\Sigma$, respectively. The above diffeomorphism further implies that local coordinates $x^I$ on $\Sbb^2$ (extending to local coordinates on the leaves $S_r$ via $\Phi_{\ul{L}}$) gives rise to local coordinates on $\Sigma$ with tangent vector fields
	\[
		\partial_i=\partial_I+\partial_I(\omega)\ul{L}.
	\]
	In particular\footnote{Recall that $\omega$ is extended constantly along the integral of $\ul{L}$},
	\begin{align*}
		\nabla\omega&=(\gamma_\omega)^{ij}(\partial_j\omega)\partial_i=(\gamma_\omega)^{IJ}(\partial_J\omega)(\partial_I+\partial_I\omega\ul{L})=\nabla^I\omega\partial_I+\btr{\nabla\omega}^2_{\gamma_\omega}\ul{L}
	\end{align*}
	by slight abuse of notation, where $\nabla^I\omega:=(\gamma_\omega)^{IJ}\partial_J\omega$. Consequently, the frame $\{\partial_i,\ul{L},L\}$ of $TM$ along a spacelike cross section $\Sigma$ admits the following decomposition in terms of a given background foliation:
	\begin{align}
		\begin{split}\label{eq_prelim_nullgeometry_frame}
		\partial_i&=\partial_I+\partial_I\omega\ul{L},\\
		\ul{L}&=\ul{L},\\
		L&=L_r-\btr{\nabla\omega}^2_{\gamma_{\omega}}\ul{L}-2\nabla\omega^I\partial_I.
		\end{split}
	\end{align}

	Finally, we recall that using the above relations, it is a well-known fact that $\gamma_r$, $\ul{\chi}_r$ extend to well-defined transversal tensor fields on $\Gamma(T\mathcal{N})$, cf. \cite[Section 3]{marssoria}, and $\ul{\theta}_r$ to a well-defined function on $\mathcal{N}$ (independent of the choice of background foliation). Hence, we will often omit the subscript $r$ for these objects without ambiguity. In particular, the properties of $\ul{\theta}$ are fully determined by $\mathcal{N}$. We say $\mathcal{N}$ is a \emph{null cone} if $\ul{\theta}>0$ on $\mathcal{N}$, cf. \cite[Definition 2.1]{roesch2}, and note that $A$, $\tau$ are well-defined for any spacelike cross section of a null cone.
	
	As in \cite{kroenckewolff}, we extend\footnote{Here, we emphasize that this extension does depend on the choice of background foliation. See also \cite{lambertscheuer} for a different choice of extension.} $\chi_r$, $\zeta_r$ transversally to tensor fields on $\Gamma(T\mathcal{N})$, i.e.,
	\begin{align*}
		\chi_r(\cdot, \ul{L})=\chi_r(\ul{L},\cdot)&=0,\qquad 
		\zeta_r(\ul{L})=0,
	\end{align*}
	and thus find that $\chi_r(\widetilde{V},\widetilde{W})=\chi_r(V,W)$, $\zeta_r(\widetilde{V})=\zeta_r(V)$ when restricted to $\Gamma(T\Sigma)$, cf. \cite[Proposition 4.22]{wolff_thesis}.
	
	\subsection{Asymptotically Schwarzschildean lightcones}\label{subsec_prelim_asymSchw}
	
	We define asymptotically Schwarzschildean lightcones as in \cite{kroenckewolff}. For the convenience of the reader, we briefly collect the definition of our decay assumptions. For more details, we refer the interested reader to \cite[Sections 2.2 and 2.4]{kroenckewolff}.

	\begin{defi}\label{defi_prelim_O}
		Let $(T_r)$ be a (smooth) family of $(m,n)$-tensors on $\Sbb^2$, $m,n\in\mathbb{N}_0$, and let $k,l\in\mathbb{N}_0$, $\alpha\in\Z$. We say $(T_r)$ is in $O_{k,l}(r^{\alpha})$ (or simply $T_r=O_{k,l}(r^{\alpha})$) if
		\[
		\btr{\widehat{\nabla}^i\partial_r^jT_r}_{\hatgamma}\le Cr^{\alpha-j}
		\]
		for all $0\le i\le k$, $0\le j\le l$. We often omit the subscript $r$ for simplicity.
	\end{defi}
	
	Here and in the following $\widehat{\nabla}$ denotes the Levi--Civita connection of the standard round metric $\widehat{\gamma}$. See Appendix \ref{appendix_differencetensor} for our conventions which we use here and in the following sections. As we also require control on the curvature tensor of the ambient spacetime $\overline{M}$, we further make the following definition for tensor fields $T$ on $\overline{M}$ evaluated along the null hypersurface:
	
	\begin{defi}\label{defi_prelim_O_spacetime}
		Let $(S_r)\subset \overline{M}$ be a smooth family of topological codimension $2$ spheres in an ambient spacetime $(\overline{M},\overline{g})$, and let $\{\ul{L}_r,L_r\}$ be a (smooth) null frame of $S_r$, $\alpha\in\Z$. We say an $(m,n)$-tensor $T$ of $\overline{M}$ is in $\overline{O}_{k,l}(r^\alpha)$ (or simply $T=\overline{O}_{k,l}(r^{\alpha})$) (with respect to the null frame $\{\ul{L}_r,L_r\}$ and $\hatgamma$) if for all subsets $\mathcal{A}\subseteq \{1,\dotsc,n\}$, $\mathcal{B}\subseteq \{1,\dotsc,m\}$
		the (smooth) family of $(m-\btr{\mathcal{B}},n-\btr{\mathcal{A}})$-tensors
		\[
		(T_r)_{\alpha_1\dotsc\alpha_n}^{\beta_1\dotsc\beta_m}
		\]
		on $\Sbb^2$ defined via $\alpha_i\in\{\ul{L}_r,L_r\}$ $\forall i\in\mathcal{A}$, ${\beta_j\in\{\ul{L}_r^*,L_r^*\}}$ $\forall j\in\mathcal{B}$, and the pullback of $T$ with respect to $S_r$ is in $O_{k,l}(r^{\alpha+n-\btr{\mathcal{A}}-m+\btr{\mathcal{B}}})$ in the sense of Definition \ref{defi_prelim_O}.
	\end{defi}
	
	With these definitions at hand, we define an asymptotically Schwarzschildean lightcone as in \cite[Definition 2.10]{kroenckewolff}:
	
	\begin{defi}[Asymptotically Schwarzschildean lightcones]\label{defi_asymclassS}
		Let $\mathcal{N}$ be a null hypersurface in an ambient spacetime. We say $\mathcal{N}$ is an \emph{asymptotically Schwarzschildean lightcone} (of mass $m$), if there exists a geodesic null generator $\ul{L}$ and an (asymptotic) background foliation $(S_r)_{r\in(r_0,\infty)}$, $r_0>0$, of topological $2$-spheres, and a smooth, real-valued function $h:(r_0,\infty)\to(0,\infty)$ with $h(r)=1-\frac{2m}{r}+O_4(r^{-2})$, such that
		\begin{align*}
			\gamma_r&=r^2\hatgamma+O_{3,3}(1),\\
			\ul{\chi}_r&=r\hatgamma+O_{3,3}(r^{-1}),\\
			\chi_r&=rh(r)\hatgamma+O_{1,1}(r^{-1}),\\
			\zeta_r&=O_{2,2}(r^{-2}),
		\end{align*}
		and the ambient curvature tensor satisfies
		\[
		\overline{\Riem}-\overline{\Riem}^{Schw}=\overline{O}_{2,2}(r^{-4}),
		\]
		and
		\[
		\overline{\nabla}\overline{\Riem}-\overline{\nabla}\overline{\Riem}^{Schw}=\overline{O}_{1,1}(r^{-5}),
		\]
		where $\overline{\Riem}^{Schw}$, $\overline{\nabla}\overline{\Riem}^{Schw}$ denote the curvature in the Schwarzschild spacetime of mass $m$, 
		{and where we evaluate $\overline{\Riem}^{Schw}$, $\overline{\nabla}\overline{\Riem}^{Schw}$ in $\ul{L}$, $L_r$ with respect to the corresponding null vectors $\ul{L}$, $L_r$ in the Schwarzschild spacetime by slight abuse of notation.} 
	\end{defi}
	Note that it is direct to see that
	\begin{align}\label{eq_asym_inversemetric}
		\gamma_r^{kl}=\frac{1}{r^2}\hatgamma^{kl}+O_{3,3}(r^{-4})
	\end{align}
	which yields the following immediate consequences along the leaves of the background foliation $(S_r)$:
	\begin{lem}\label{lem_prelim_backgroundfol}
		Let $\mathcal{N}$ be an asymptotically Schwarzschildean lightcone. For a leave $S_r$ of the background foliation, we find
		\begin{align*}
			\ul{\theta}_r&=\frac{2}{r}+O_{3,3}(r^{-3}),\\
			\theta_r&=\frac{2}{r}h(r)+O_{1,1}(r^{-3}),
		\end{align*}
		and 
		\begin{align*}
			\accentset{\circ}{\ul{\chi}}_r&=O_{3,3}(r^{-1}),\\
			\accentset{\circ}{\chi}_r&=O_{1,1}(r^{-1}).
		\end{align*}
		In particular, we have
		\begin{align*}
			\accentset{\circ}{A}_r&=O_{1,1}(r^{-2}),\\
			\tau_r&=O_{2,2}(r^{-2}),\\
			\mathcal{H}^2_r&=\frac{4}{r^2}-\frac{8m}{r^3}+O_{1,1}(r^{-4})
		\end{align*}
		for $r$ sufficiently large.
	\end{lem}
	\begin{bem}\label{bem_prelim_backgroundfol}
		We recall that $\gamma_r$ is (uniformly) equivalent to  $\widetilde{\gamma}_r=r^2\hatgamma$ provided $r$ is sufficiently large. In particular, any $(n,m)$-tensor satisfies
		\[
		\btr{T}_{\gamma_r}\le C r^{n-m}\btr{T}_{\hatgamma}.
		\]
	\end{bem} 
	We now consider a cross-section $\Sigma=\Sigma_{\omega}$ in $\mathcal{N}$ with induced metric $\gamma=\gamma_\omega$, where here and in the following $\omega:\Sbb^2\to(r_0,\infty)$ always denotes the graph function with respect to our choice of background foliation given by Definition \ref{defi_asymclassS}.
	
	In contrast to \cite{kroenckewolff}, we will not assume an a priori control on the derivatives of $\omega$ here, and will derive some pointwise estimates on the tensor norms along $\Sigma$. As the estimates are derived in almost complete analogy to \cite[Section 2.4]{kroenckewolff}, using Lemma \ref{lem_appendix_derivativesomega} and Lemma \ref{lem_appendix_derivativesTR} from the Appendix, we will omit most of the proofs.
	
	\begin{lem}\label{lem_prelim_tau}
		Let $\mathcal{N}$ be an asymptotically Schwarzschildean lightcone, $\Sigma=\Sigma_\omega$. If $\omega\ge r_0$ sufficiently large, then
		\begin{align*}
			\newbtr{\tau}_\gamma&\le \frac{C}{\omega^3}\left(1+\newbtr{\widehat{\nabla}\ln\omega}_{\hatgamma}\right),\\
			\newbtr{\nabla\tau}_\gamma&\le \frac{C}{\omega^4}\left(1+\newbtr{\widehat{\nabla}\ln\omega}^2_{\hatgamma}+\newbtr{\widehat{\nabla}^2\ln\omega}_{\hatgamma}\right),\\
			\newbtr{\nabla^2\tau}&\le \frac{C}{\omega^5}\left(\left(1+\newbtr{\widehat{\nabla}\ln\omega}_{\hatgamma}\right)\left(1+\newbtr{\widehat{\nabla}\ln\omega}_{\hatgamma}^2+\newbtr{\widehat{\nabla}^2\ln\omega}_{\hatgamma}\right)+\newbtr{\widehat{\nabla}^3\ln\omega}_{\hatgamma}\right)	
		\end{align*}
		for some constant $C$ independent of $\omega$.
	\end{lem}
	\begin{proof}
		As in the proof of \cite[Lemma 2.15]{kroenckewolff}, the proof directly follows from the identity
		\[
			\tau=\tau_r-\accentset{\circ}{\ul{\chi}}(\cdot,\nabla\omega)-\frac{1}{\ul{\theta}}\left(\newbtr{\accentset{\circ}{\ul{\chi}}}^2+\overline{\Ric}(\ul{L},\ul{L})\right)\d\omega,
		\]
		and Lemma \ref{lem_appendix_derivativesomega} and Lemma \ref{lem_appendix_derivativesTR}, where we use that
		\[
			\overline{\Ric}(\ul{L},\ul{L})=O_{2,2}(r^{-4})
		\]
		by Proposition \ref{prop_appendix_curvatureidentities}.
	\end{proof}
	
	\begin{lem}\label{lem_prelim_A}
		Let $\mathcal{N}$ be an asymptotically Schwarzschildean lightcone, $\Sigma=\Sigma_\omega$. If $\omega\ge r_0$ sufficiently large, then
		\begin{align*}
			\newbtr{A}_\gamma&\le \frac{C}{\omega^2}\left(1+\newbtr{\widehat{\nabla}\ln\omega}_{\hatgamma}^2+\newbtr{\widehat{\nabla}^2\ln\omega}_{\hatgamma}\right),\\
			\newbtr{\nabla A}_\gamma&\le \frac{C}{\omega^3}\left(\left(1+\newbtr{\widehat{\nabla}\ln\omega}_{\hatgamma}\right)\left(1+\newbtr{\widehat{\nabla}\ln\omega}_{\hatgamma}^2+\newbtr{\widehat{\nabla}^2\ln\omega}_{\hatgamma}\right)+\newbtr{\widehat{\nabla}^3\ln\omega}_{\hatgamma}\right)
		\end{align*}
		for some constant $C$ independent of $\omega$.
	\end{lem}
	\begin{proof}
		As before, the claim is a straightforward application of Lemma \ref{lem_appendix_derivativesomega} and Lemma \ref{lem_appendix_derivativesTR} upon using the identity
		\[	
			A=A_r-2\ul{\theta}\nabla^2\omega+\btr{\nabla\omega}_\gamma^2\ul{\theta}\ul{\chi}-2\ul{\theta}\left(\d\omega\otimes \zeta_r+\zeta_r\otimes\d\omega\right),
		\]
		see \cite[Equation (3.12)]{kroenckewolff}, where we note that $A_r=O_{1,1}(1)$, $\ul{\chi}=O_{3,3}(r)$. 
	\end{proof}
	Conversely, the above identity for $A$ further implies that
	\begin{align}\label{eq_prelim_c2estimate_A}
			\newbtr{\widehat{\nabla}^2\ln\omega}_{\hatgamma}&\le C\left(1+\omega^2\btr{A}_\gamma+\newbtr{\widehat{\nabla}\ln\omega}_{\hatgamma}^2\right).
	\end{align}
	Using that $\newbtr{\nabla\chi}_\gamma\le \newbtr{\nabla\ul{\theta}}_\gamma+\newbtr{\nabla\accentset{\circ}{\ul{\chi}}}_\gamma$, and $\ul{\theta}=2\omega^{-1}+O_{3,3}(r^{-3})$, we in fact find that
		\[
			\newbtr{\nabla\ul{\chi}}_\gamma\le \frac{C}{\omega^{2}}\left(\omega^{-2}+\newbtr{\widehat{\nabla}\ln\omega}_{\hatgamma}\right).
		\]
	Using this slightly refined estimate, we further find that
	\begin{align}\label{eq_prelim_c3estimate_nablaA}
		\newbtr{\widehat{\nabla}^3\ln\omega}_{\hatgamma}\le C\left(\omega^3\btr{\nabla A}_\gamma+\left(1+\newbtr{\widehat{\nabla}\ln\omega}_{\hatgamma}\right)\left(1+\omega^2\btr{A}_\gamma+\newbtr{\widehat{\nabla}\ln\omega}_{\hatgamma}^2\right)\right).
	\end{align}
	Combining Equations \eqref{eq_prelim_c2estimate_A} and \eqref{eq_prelim_c3estimate_nablaA} with Lemma \ref{lem_prelim_tau}, we obtain the following:
	\begin{kor}\label{kor_prelim_nablatau}
		Let $\mathcal{N}$ be an asymptotically Schwarzschildean lightcone, $\Sigma=\Sigma_\omega$. If $\omega\ge r_0$ sufficiently large, then
		\[
			\newbtr{\nabla^2\tau}_\gamma\le \frac{C}{\omega^2}\newbtr{\nabla A}_\gamma+\frac{C}{\omega^3}\btr{A}_\gamma+\frac{C}{\omega^5}\left(1+\newbtr{\widehat{\nabla}\ln\omega}_{\hatgamma}^3\right)
		\]
		for some constant $C$ independent of $\omega$.
	\end{kor}
	
	\subsubsection{Asymptotically flat background foliations}\label{subsubsec_prelim_AF}
	We close this Subsection with a brief remark on the properties of asymptotically flat background foliations, in particular in contrast with the notion of asymptotically Schwarzschildean lightcones as considered here.
	
	Generally speaking, in a similar spirit to Definition \ref{defi_asymclassS} above, a null hypersurface is called an asymptotically flat lightcone if it admits an asymptotically flat background foliation, i.e., a background foliation that asymptotically behaves like a foliation of round spheres in the Minkowski lightcone. In particular, an asymptotically flat background foliation $(S_r)$ satisfies
	\[
		\lim_{r\to\infty}r^2K_r=1,
	\] 
	where $K_r$ denotes the Gauss curvature of $S_r$. See for example \cite{marssoria, sauter}. However, as the restricted Lorentz group is a subgroup of the isometries of the Minkowski spacetime that leaves the lightcone invariant, the Minkowski lightcone admits a plethora of foliations by round spheres related by such Lorentz transformations. In particular, an asymptotically flat lightcone also admits a plethora of asymptotically flat background foliations and applying a Lorentz transformation to a given asymptotically flat background foliation yields another asymptotically flat background foliation\footnote{We note however, that any two asymptotically flat background foliations are not necessarily directly related by a suitable Lorentz transformation. In particular, any arbitrary but sufficiently small perturbation of a given asymptotically flat background foliation also yields an asymptotically flat background foliation. Heuristically, this is due to the presence of so-called supertranslations in the BMS group, see e.g. \cite{maedlerwinicour}.}.
	
	As we may take an asymptotically Schwarzschildean lightcone to be in particular asymptotically flat, we may apply a Lorentz transformation to the foliation $(S_r)$ given in Definition \ref{defi_asymclassS} to obtain an asymptotically flat background foliation $(\widetilde{S}_{\widetilde{r}})$. We observe, that along the foliation $(\widetilde{S}_{\widetilde{r}})$ we find
	\[
		\widetilde{\chi}_{\widetilde{r}}=\left(\widetilde{r}-2\widetilde{m}(\vec{x})\right)\hatgamma+O_{3,3}(\widetilde{r}^{-1}),
	\]
	where $\widetilde{m}(\vec{x})$ is a (non-constant) function\footnote{This function is usually referred to as the so-called mass aspect function.} on $\Sbb^2$. Hence, $(\widetilde{S}_{\widetilde{r}})$ is not an asymptotically Schwarzschildean background foliation as in Definition \ref{defi_asymclassS}. More concretely, while the standard lightcone in the Schwarzschild spacetime admits many asymptotically flat background foliations of (exact) round spheres, the only asymptotically Schwarzschildean background foliation of round spheres is the foliation by centered round spheres, which is the unique foliation by STCMC surfaces, see \cite{chenwang}.
	
	Recall that for a given asymptotically flat lighcone, the Bondi energy $E_B$ and Bondi (linear) momentum $\vec{P}_B$ are defined via limits of surface integrals with respect to a given asymptotically flat background foliation, see e.g. \cite{sauter}. In particular, the Bondi energy $E_B$ is given as the limit of the Hawking energy with respect to the given foliation, i.e.,
	\[
		E_B((S_r))=\lim\limits_{r\to\infty}m_{H}(S_r)=\lim\limits_{r\to\infty}\sqrt{\frac{\btr{S_r}}{16\pi}}\left(1-\frac{1}{16\pi}\int_{S_r}\mathcal{H}^2_r\right).
	\]
	If $E_B((S_r))\ge \newbtr{\vec{P}_B((S_r))}$, then one defines the Bondi mass as $m_B:=\sqrt{E_B((S_r))^2-\newbtr{\vec{P}_B((S_r))}^2}$. Note that Bondi mass $m_B$ is independent of the given choice of asymptotically flat background foliation and (if well-defined) can alternatively be realized as the infimum over all possible Bondi energies, i.e.,
	\[
		m_B=\inf\limits_{(S_r)\text{ }A.F.}E_B((S_r)).
	\]
	Observe that any asymptotically flat background foliation arising from a Lorentz boost of a given asymptotically flat background foliation can be expressed as a family of graph functions $\omega_{\rho,\vec{a}}$ given by
	\[
		\omega_{\rho,\vec{a}}:=\frac{\rho}{\sqrt{1+\btr{\vec{a}}^2}-\vec{a}\cdot\vec{x}}
	\]
	for some vector $\vec{a}\in\R^3$, compare Section \ref{subsec_prelim_quantativeestimates} below. In particular, for an asymptotically Schwarzschildean lightcone of mass $m$, using the given asymptotically Schwarzschildean background foliation as reference, we find that\footnote{In particular, the infimum as above is realized as the infimum over all possible Lorentz transformations of the given asymptotically Schwarzschildean background foliation. While somewhat tedious, this can be explicitly checked in the setting of our strong asymptotics. More generally, this is related to the fact that the energy momentum vector is invariant under supertranslations. See e.g. \cite{maedlerwinicour}.}
	\[
		m_B=\inf_{\vec{a}\in\R^3}\lim\limits_{\rho\to\infty}m_H(\Sigma_{\rho,\vec{a}})=\inf_{\vec{a}\in\R^3}\sqrt{1+\btr{\vec{a}}^2}m=m.
	\]
	In particular, $E_B=m$ and $\vec{P}_B=\vec{0}$ for an asymptotically Schwarzschildean background foliation as in Definition \ref{defi_asymclassS}. 
	
	\subsection{Quantative estimates for almost round surfaces}\label{subsec_prelim_quantativeestimates}
	
	Recall that every spacelike cross section of the standard Minkowski lightcone is of the form $\Sigma=\{t=r=\omega\}=:\Sigma_\omega$ for a function $\omega:\Sbb^2\to\R$ and the  induced metric is $\omega^2\hatgamma$. We define the associated $4$-vector $\textbf{Z}=\textbf{Z}(\Sigma_{\omega})$ by setting its components as
	\begin{align*}
		\textbf{Z}^t:=\frac{1}{\btr{\Sigma}}\int_{\Sbb^2}\omega^3\d\widehat{\mu},\qquad
		\textbf{Z}^i:=\frac{1}{\btr{\Sigma}}\int_{\Sbb^2}\omega^3f^i\d\widehat{\mu},\quad i=1,2,3
	\end{align*}
	where $\d\widehat{\mu}$ is the volume element of the standard {round} metric $\widehat{\gamma}$ on $\Sbb^2$ and $f^i$ denote first spherical harmonics. This $4$-vector was {introduced for spacelike cross sections of the Minkowski lightcone} by the second named author in \cite{wolff4}, and we refer to \cite{wolff4} for more details. {We note that this $4$-vector is up to scaling equivalent to a notion of hyperbolic center independently defined by Cederbaum--Cortier--Sakovich, cf. \cite{cederbaumcortiersakovich}.}
	
	As shown in \cite{wolff4}, $\textbf{Z}$ is a timelike, {future-pointing} vector, so there exists a unique $\vec{a}\in\R^3$ such that
	\[
	\textbf{Z}=\btr{\textbf{Z}}\begin{pmatrix}
		\sqrt{1+\btr{\vec{a}}^2}\\\vec{a}
	\end{pmatrix}.
	\]
	{For a vector $\vec{a}\in \R^3$, let $\Lambda_{\vec{a}}$ be the Lorentz boost associated to $\vec{a}$}.
	The boosted surface $\Lambda_{\vec{a}}(\Sigma_{\omega})$ is now given by $\Sigma_{\omega_{\vec{a}}}$,
	where
	\[
	\omega_{\vec{a}}(\vec{x})=\frac{\omega\circ\Phi(\vec{x})}{\sqrt{1+\btr{\vec{a}}^2}-\vec{a}\cdot\vec{x}},
	\]
	and $\Phi\colon\Sbb^2\to\Sbb^2$ is a diffeomorphism\footnote{In fact, $\Phi$ is a M\"obius transformation in the M\"obius group, which is isomorphic to $\operatorname{SO}^+(1,3)$.} uniquely determined by $\Lambda_{\vec{a}}$. It was shown by the second author in \cite{wolff4} that $\textbf{Z}$ transforms equivariantly under Lorentz boosts, i.e.,
	\[
	\textbf{Z}(\Sigma_{\omega_{\vec{a}}})=\Lambda_{\vec{a}}(\textbf{Z}(\Sigma_\omega)).
	\]
	For $\rho>0$, $\vec{a}\in\R^3$,
	we set	
	\[
	b_{\rho,\vec{a}}(\vec{x}):=\frac{\rho}{\sqrt{1+\btr{\vec{a}}^2}-\vec{a}\cdot\vec{x}}
	\]
	and abbreviate $b_{\vec{a}}:=b_{1,\vec{a}}$. 
	
	As in \cite{kroenckewolff}, we require quantitative pointwise estimates that arise from a scalar curvature comparison. As we impose much less control on the surfaces under consideration, we need a slightly refined version of 
	\cite[Proposition 2.9]{kroenckewolff}:
	
	\begin{prop}\label{prop_gausscurvature_c1}
		Let $\gamma=\omega^2\hatgamma$ be a conformally round with Gauss curvature $K$ such that 
		\[
			\int_{\Sbb^2}\omega^2\d\widehat{\mu}=4\pi\text{, and } \int_{\Sbb^2} f_i\omega^3\d\widehat{\mu}=0
		\] 
		for $i=1,2,3$. There exist a uniform constant $\varepsilon>0$ , such that if
		\[
		\norm{{K}-1}_{C^1\left(\hatgamma\right)}\le \varepsilon
		\]
		then
		\begin{align*}
			\norm{\omega-1}_{C^{2,\alpha}\left(\hatgamma\right)}\le C\norm{{K}-1}_{C^\alpha\left(\hatgamma\right)}
		\end{align*}
		for a positive constant $C$ only depending on $\alpha$.
	\end{prop}
	
	We highlight that here we drop the uniform bounds on the conformal factor compared to \cite[Proposition 2.9]{kroenckewolff} in the assumptions, as we show them to be always satisfied by the following Proposition:
	
	\begin{prop}\label{prop_gausscurvature_c0}
		Let $\gamma=\omega^2\hatgamma$ be conformally round with Gauss curvature $K$ such that 
		\[
		\int_{\Sbb^2}\omega^2\d\widehat{\mu}=4\pi\text{, and } \int_{\Sbb^2} f_i\omega^3\d\widehat{\mu}=0
		\] 
		for $i=1,2,3$. For every $\delta>0$, there exists $\varepsilon>0$ such that if
		\[
			\norm{K-1}_{C^0(\hatgamma)}\le \varepsilon
		\]
		then $\btr{u}\le \delta$, where $u:=\ln\omega$.
	\end{prop}
	
	Proposition \ref{prop_gausscurvature_c0} follows in close analogy to an estimate by Klainerman--Szeftel \cite[Proposition 3.7]{klainermanszeftel} where the main difference is that we replace the condition of vanishing conformal center of mass
	\[
		\int_{\Sbb^2} f_i\omega^2\d\widehat{\mu}=0
	\]
	by the balancing condition 
	\[
		\int_{\Sbb^2} f_i\omega^3\d\widehat{\mu}=0
	\]
	motivated by the Minkowski $4$-vector $\textbf{Z}$. For the convenience of the reader, we provide an outline of the proof in Appendix \ref{appendix_roundness}. As the Minkowskian $4$-vector $\textbf{Z}$ of a spacelike cross section transforms equivariantly under a Lorentz boost as outlined above, this allows us to directly extend Proposition \ref{prop_gausscurvature_c1} to general spacelike cross sections:
	
	\begin{kor}\label{kor_quantitativecontrol}
		Let $\gamma=\omega^2\hatgamma$ be a conformally round metric with scalar curvature $\scal$, area radius $\rho$, and associated $4$-vector $\textbf{Z}$ corresponding to a vector $\vec{a}\in\R^3$. There exists uniform constants $\varepsilon>0$, $C>0$ such that if 
		\[
		\norm{\scal-\fint\scal }_{C^1\left(\hatgamma\right)}\le \frac{\varepsilon}{\rho^2\sqrt{1+\btr{\vec{a}}^2}}
		\]
		then
		\begin{align*}
			\norm{{\omega}-b_{\rho,\vec{a}}}_{C^{2}\left(\hatgamma\right)}&\le C(1+\btr{\vec{a}}^2)^2 \rho^3 \norm{\scal-\fint\scal }_{C^1\left(\hatgamma\right)}.
		\end{align*}
	\end{kor}
	\begin{proof}
		Consider the conformally round metric $\widetilde{\omega}^2\hatgamma$, where $\widetilde{\omega}:=\rho^{-1}\omega_{-\vec{a}}$, where we recall that
		\[
			\omega_{-\vec{a}}=b_{-\vec{a}}\,\omega\circ \Phi_{-\vec{a}}
		\]
		and $\Phi_{-\vec{a}}$ is a uniquely determined M\"obius transformation with $\det D\Phi_{-\vec{a}}=b_{-\vec{a}}$. As we may assume without loss of generality that $\vec{a}=(0,0,\btr{\vec{a}})$ (after a suitable rotation acting as an isometry on $\widehat{\gamma}$), one can confirm by direct computation that
		\[
			\btr{D^k \Phi_{-\vec{a}}}\le C_{1,k} \btr{\widehat{\nabla}^{k-1}b_{-\vec{a}}}\le C_{2,k}(1+\btr{\vec{a}}^2)^{\frac{k}{2}}
		\]
		for uniform constants $C_{1,k}$, $C_{2,k}$ only depending on $k$. As 
		\[
			\widetilde{K}=\frac{\rho^2}{2}\scal\circ \Phi_{-\vec{a}},
		\]
		and 
		\[
			\omega-b_{\rho,\vec{a}}=\rho b_{\vec{a}}\left(\widetilde{\omega}\circ \Phi_{\vec{a}}-1\right),
		\]
		where $\widetilde{K}$ denotes the Gauss curvature of $\widetilde{\omega}^2\hatgamma$, the above estimates directly imply the claim as $\widetilde{\omega}$ satisfies all assumptions of Proposition \ref{prop_gausscurvature_c1}.
		
	\end{proof}
	
	Observe that the dependency on $\vec{a}$ in the assumption 
	\[
	\norm{\scal-\fint\scal }_{C^1\left(\hatgamma\right)}\le \frac{\varepsilon}{\rho^2\sqrt{1+\btr{\vec{a}}^2}}
	\] 
	only arises from the derivatives of the scalar curvature, as the scalar curvature itself transforms by composition with a M\"obius transformation under a boost. Consequently, we can in fact use Proposition \ref{prop_gausscurvature_c0} to derive a rough estimate for the vector $\vec{a}$:
	
	\begin{kor}\label{kor_acontrol_rough}
		Let $\gamma=\omega^2\hatgamma$ be a conformally round metric with scalar curvature $\scal$, area radius $\rho$, and associated $4$-vector $\textbf{Z}$ corresponding to a vector $\vec{a}\in\R^3$. There exists a uniform constant $\varepsilon>0$ such that if 
		\[
			\norm{\scal-\fint\scal}_{C^0(\,\widehat{\gamma}\,)}\le \frac{\varepsilon}{\rho^2},
		\]
		then
		\[
			\min\left(2\btr{\vec{a}},\sqrt{1+\btr{\vec{a}}^2}\right)\le 2\min\left(\frac{\rho}{\omega_{\min}},\frac{\omega_{\max}}{\rho}\right),
		\]
		where $\omega_{\min}$ and $\omega_{\max}$ denote the (global) minimum and maximum of $\omega$ on $\Sbb^2$, respectively.
	\end{kor}
	\begin{proof}
		Wlog we assume $\vec{a}\not=0$, otherwise the estimate is trivial. As before, we note that $\widetilde{\omega}=\rho^{-1}b_{-\vec{a}}\cdot\left(\omega\circ \Phi_{-\vec{a}}\right)$ has area radius one and satisfies the balancing condition. Moreover,
		\[
			\norm{\widetilde{K}-1}_{C^0(\,\hatgamma\,)}=\frac{\rho^2}{2}\norm{\scal-\fint\scal}_{C^0(\,\hatgamma\,)}.
		\]
		We now choose $\varepsilon$ sufficiently small such that Proposition \ref{prop_gausscurvature_c0} holds with $\delta=\frac{1}{2}e^{-1}<1$. Using the mean value theorem, if is straightforward to check that $\btr{\ln\widetilde{\omega}}\le \delta$ implies
		\[
			\norm{\widetilde{\omega}-1}_{C^0(\hatgamma)}\le \frac{1}{2}.
		\]
		Rearranging the inequality and using the above identity for $\widetilde{\omega}$, we find
		\[
			b_{-\vec{a}}(\vec{x})\le \frac{3\rho}{2\omega\circ\Phi_{-\vec{a}}(\vec{x})}\le 2\frac{\rho}{\omega_{\min}}
		\]
		for all $\vec{x}\in \Sbb^2$. Now observe that
		\[
			\min\left(2\btr{\vec{a}},\sqrt{1+\btr{\vec{a}}^2}\right)\le \sqrt{1+\btr{\vec{a}}^2}+\btr{\vec{a}}=b_{-\vec{a}}\left(\frac{-\vec{a}}{\btr{\vec{a}}}\right)\le 2\frac{\rho}{\omega_{\min}}.
		\] 
		The corresponding inequality involving $\omega_{\max}$ is derived similarly.
	\end{proof}
	
	\section{A Sobolev-inequality on spacelike cross sections}\label{section_sobolev}
	
	While several formulations of a Sobolev inequality are well-known for topological $2$-spheres, our later analysis requires a notion with respect to the extrinsic geometry of a spacelike cross section. For the purpose of this paper, we thus make the following definition:
	\begin{defi}\label{defi_sobolev}
		Let $(\Sigma,\gamma)$ be a closed, orientable, spacelike codimension $2$ surface in an ambient $4$-dimensional spacetime $(\overline{M},\overline{g})$ with $\mathcal{H}^2\ge 0$. We say that $\Sigma$ carries a Sobolev inequality (with respect to $\mathcal{H}^2$) with constant $c_S>0$ if
		\begin{align}\label{eq_prelim_SobolevGoal}
			\int_\Sigma f^2\d\mu \le c_S\left(\int_\Sigma \btr{\nabla f}+\sqrt{\mathcal{H}^2}f\d\mu\right)^2
		\end{align}
		for all $f\in W^{1,2}(\Sigma)$.
	\end{defi}
	
	In their previous work, the authors have already established a similar Sobolev inequality in asymptotically Schwarzschildean lightcone for a restricted class of spacelike cross sections, see \cite[Lemma 3.11]{kroenckewolff}. It is not difficult to verify that in this setting, the spacelike cross sections under consideration also carry a Sobolev inequality in the sense of Definition \ref{defi_sobolev} for some uniform constant $c_S$. The main objective of this section is to establish more general criteria under which a spacelike cross section carries a Sobolev inequality with respect to $\mathcal{H}^2$ while retaining sufficient control on the constant $c_S$.
	
	To this end, we first note that a surface (with non-negative mean curvature) in $\R^3$ (understood as a time symmetric hypersurface in $\R^{1,3}$) carries a Sobolev inequality in the sense of Definition \ref{defi_sobolev} by the well-known Michael--Simon--Sobolev inequality, where $c_S$ is in fact a uniform constant indepenent of $\Sigma$. Exploiting the closeness to Euclidean space, Huisken--Yau then extend this inequality to their setting, see \cite[Proposition 5.4]{huiskenyau}. In close analogy, we first investigate the validity of a Sobolev inequality as in Definition \ref{defi_sobolev} for spacelike cross sections of the standard Minkowski lightcone.
	
	We observe the following direct consequence by the H\"older inequality:
	\begin{lem}\label{lem_prelim_Sobolev_consequence}
		Assume that $(\Sigma,\gamma)$ carries a Sobolev inequality with constant $c_S$. Then for any $f\in W^{1,2}(\Sigma)\cap L^{2p}(\Sigma)$, $p\ge 1$, we have
		\[
		\left(\int_\Sigma \btr{f}^{2p}\d\mu\right)^\frac{1}{p}\le 2 c_Sp^2\btr{\operatorname{supp}f}^\frac{1}{p}\int_\Sigma\btr{\nabla f}^2+ \mathcal{H}^2f^2\d\mu.
		\]
	\end{lem}

	\subsection{A Sobolev inequality in the Minkowski case}\label{subsec_sobolev_minkowski}
	
	We recall the well-known fact that there is a $1$-to-$1$ correspondence between spacelike cross sections of the Minkowski lightcone and conformally round metrics $\gamma=\omega^2\widehat{\gamma}$, where $\mathcal{H}^2=2\scal_\gamma$. See for example \cite{wolff1}. By direct computation, we find the first, preliminary estimate:
	\begin{prop}\label{prop_sobolev_mink1}
		Let $\Sigma$ be a spacelike cross section in the Minkowski lightcone with conformally round metric $\gamma=\omega^2\widehat{\gamma}$ and spacetime mean curvature $\mathcal{H}^2> 0$. Then
		\[
		\int_\Sigma f^2\d\mu\le (2\pi)^{-1}\left(\int_\Sigma \btr{\nabla f}+c\sqrt{\mathcal{H}^2}f\d\mu\right)^2
		\]
		for any $f\in W^{1,2}(\Sigma)$, where $c\ge 1$ is a constant depending on $\norm{\widehat{\nabla}\ln\omega}_{C^1(\,\widehat{\gamma}\,)}$.
	\end{prop}
	\begin{bem}\label{rem_sobolev_mink1}
		The proof relies on the Michael--Simon--Sobolev inequality for the round sphere, and we emphasize that the constant $c$ depends on $\Sigma$. Of course, the question if one can prove the above inequality for some uniform constant $c$ independent of $\Sigma$ naturally arises. We note that, as $\Sigma$ has positive Gauss curvature, there exists an isometric embedding into $\R^3$ with
		\[
		H^2=2\newbtr{\accentset{\circ}{h}}^2+2\scal=2\newbtr{\accentset{\circ}{h}}^2+\mathcal{H}^2
		\]
		by the Gauss equation in $\R^3$, where $h$ denotes the second fundamental form of the surface in $\R^3$. Recall that the Michael--Simon--Sobolev inequality for surfaces in $\R^3$ states that
		\[
		\int_\Sigma f^2\d\mu\le c\left(\int_\Sigma \btr{\nabla f}+\btr{H}f\d\mu\right)^2.
		\]
		Hence, the above Sobolev inequality in the Minkowski lightcone is equivalent to a modified version of the Sobolev inequality in $\R^3$ incorporating the trace-free part of $h$. Whether one can expect such a modified version to hold with a uniform constant is unclear to the authors and beyond the scope of this paper.
	\end{bem}
	\begin{proof}
		As smooth functions lie dense in $W^{1,2}(\Sigma)$, it suffices to prove the estimate for $f$ smooth. By the Micheal--Simon--Sobolev inequality we have
		\[
		\int_{\Sbb^2} f^2\d\widehat{\mu}\le c\left(\int_{\Sbb^2}\newbtr{\widehat{\nabla}f}+2f\d\widehat{\mu}\right)^2
		\]
		on the round sphere $\Sbb^2$ as $H=2$. In the following, we will use $c\le (2\pi)^{-1}$.\footnote{Note that $(2\pi)^{-1}$ is not optimal, but we use this explicit constant going forward for convenience.}
		Using the fact that $\Sigma$ is conformally round with $\d\mu=\omega^2\d\widehat{\mu}$, we find
		\[
		\int_\Sigma f^2\d\mu=\int_{\Sbb^2}(f\omega)^2\d\widehat{\mu}\le (2\pi)^{-1}\left(\int_{\Sbb^2}\newbtr{\widehat{\nabla}f\omega}+2f\omega\d\widehat{\mu}\right)^2.
		\]
		As 
		\begin{align*}
			\newbtr{\widehat{\nabla}f\omega}^2&=\omega^2\newbtr{\widehat{\nabla}f}^2+f^2\newbtr{\widehat{\nabla}\omega}^2+2\omega f\spann{\widehat{\nabla}\omega,\widehat{\nabla} f}\\
			&\le \omega^2\newbtr{\widehat{\nabla} f}^2+f^2\newbtr{\widehat{\nabla}\omega}^2+2\omega f\newbtr{\widehat{\nabla}\omega}\newbtr{\widehat{\nabla}f}\\
			&=\left(\omega \newbtr{\widehat{\nabla}f}+f\newbtr{\widehat{\nabla}\omega}\right)^2,
		\end{align*}
		and $\newbtr{\nabla f}^2=\omega^{-2}\newbtr{\widehat{\nabla}f}^2$, we can conclude that
		\[
		\int_\Sigma f\d\mu \le (2\pi)^{-1}\left(\int_\Sigma \newbtr{\nabla f}+\left(\frac{2}{\omega}+\frac{\newbtr{\nabla\omega}}{\omega}\right)f\d\mu \right)^2
		\]
		As $\mathcal{H}^2=2\scal$, we note that
		\[
		\sqrt{\mathcal{H}^2}=\sqrt{\frac{4}{\omega^2}+\frac{4\newbtr{\widehat{\nabla}\omega}^2}{\omega^4}-\frac{4\widehat{\Delta}\omega}{\omega^3}}=\frac{2}{\omega}\sqrt{1-\widehat{\Delta}\ln\omega}
		\]
		Hence, the claim follows by defining
		\[
		c:=\max\limits_{\Sbb^2}\frac{1+\frac{1}{2}\newbtr{\widehat{\nabla}\ln\omega}}{\sqrt{1-\widehat{\Delta}\ln\omega}}.
		\]
	\end{proof}
	
	While Proposition \ref{prop_sobolev_mink1} readily implies that any spacelike cross section of the Minkowski lightcone with $\mathcal{H}^2>0$ carries a Sobolev inequality with respect to $\mathcal{H}^2$, the constant derived in Proposition \ref{prop_sobolev_mink1} is far from optimal. In particular, while the constant is scaling invariant, it is not Lorentz invariant. More precisely, it is easy to construct a family of round spheres (for which the Sobolev constant holds with respect to a uniform constant) such that the constant in Proposition \ref{prop_sobolev_mink1} diverges to infinity. However, as Lorentz transformations are isometries of the Minkowski spacetime, one may directly remedy this fact by taking the infimum of the constant over all possibly isometric embeddings into the Minkowski lightcone\footnote{Note that the Sobolev inequality with respect to $\mathcal{H}^2$ is completely intrinsic in this setting as $\mathcal{H}^2=2\scal$. Since Lorentz transformations are equivalent to a conformal diffeomorphism on the given surface, the Sobolev inequality indeed holds for any of these constants and in particular their infimum.}.
	
	\begin{kor}\label{kor_sobolev_mink2}
		Let $\Sigma$ be a spacelike cross section of the Minkowski lightcone with $\mathcal{H}^2>0$. Then $\Sigma$ carries a Sobolev inequality with respect to $\mathcal{H}^2$ with constant 
		\[
			c_0(\Sigma):=(2\pi)^{-1}\inf\limits_{\vec{a}\in\R^3}\max\limits_{\Sbb^2}\frac{\left(1+\frac{1}{2}\newbtr{\widehat{\nabla}\ln\omega_{\vec{a}}}\right)^2}{{1-\widehat{\Delta}\ln\omega_{\vec{a}}}}\ge (2\pi)^{-1},
		\]
	\end{kor}
	
	\begin{bem}\label{bem_sobolev_mink2}
		Recall from Section \ref{subsec_prelim_quantativeestimates} that 
		\[
		\omega_{\vec{a}}=\frac{\omega\circ \Phi_{\vec{a}}}{\sqrt{1+\newbtr{\vec{a}}^2}-\vec{a}\,\vec{x}}
		\]
		with $\Phi_{\vec{a}}$ being a uniquely determined M\"obius transformation. While $c_0(\Sigma)$ is by definition scaling and Lorentz invariant, and clearly the optimal constant derived from our approach, we still desire criteria under which $c_0(\Sigma)$ is controlled from above. While the definition of $c_0(\Sigma)$ appears quite opaque at first sight, it is clear heuristically that the constant is well-controlled for any spacelike cross section that is sufficiently $C^2$-close to a boosted round sphere. We now give a precise criteria based on this:\newline
		Let $\vec{a}\in\R^3$. By construction
		\[
			\sqrt{2\pi\,c_0(\Sigma)}\le \max\limits_{\Sbb^2}\frac{1+\frac{1}{2}\newbtr{\widehat{\nabla}\ln\omega_{-\vec{a}}}}{\sqrt{1-\widehat{\Delta}\ln\omega_{-\vec{a}}}}
			=
			\max\limits_{\Sbb^2}\frac{1+\frac{1}{2}\newbtr{\widehat{\nabla}^{\vec{a}}\ln(\omega\cdot b_{\vec{a}}^{-1})}_{\hatgamma_{\vec{a}}}}{\sqrt{1-\widehat{\Delta}_{\vec{a}}\ln(\omega\cdot b_{\vec{a}}^{-1})}}.
		\]
		Here, we used that one can directly check that
			\begin{align*}
				\btr{\widehat{\nabla}\ln\omega_{-\vec{a}}}_{\widehat{\gamma}}\circ \Phi_{\vec{a}}&=\btr{\widehat{\nabla}^{\vec{a}}\ln(\omega\cdot b_{\vec{a}}^{-1})}_{\widehat{\gamma}_{\vec{a}}},\\
				\widehat{\Delta}\ln\omega_{-\vec{a}}\circ \Phi_{\vec{a}}&=\widehat{\Delta}_{\vec{a}}\ln(\omega\cdot b_{\vec{a}}^{-1}),
			\end{align*}
		where $\widehat{\nabla}^{\vec{a}}$ and $\widehat{\Delta}_{\vec{a}}$ denote the gradient and Laplacian with respect to the conformally round metric $\widehat{\gamma}_{\vec{a}}=b_{\vec{a}}^2\widehat{\gamma}$. Hence, $c_0(\Sigma)
		\le (2\pi)^{-1}(1+\varepsilon)$ if 
		\begin{align*}
			\newbtr{\widehat{\nabla}\ln\left(\omega\,b_{\vec{a}}{^{-1}}\right)}_{\hatgamma}&\le \frac{\delta}{\sqrt{1+\newbtr{\vec{a}}^2}},\\
			\newbtr{\widehat{\Delta}\ln\left(\omega\,b_{\vec{a}}{^{-1}}\right)}&\le \frac{\delta}{1+\newbtr{\vec{a}}^2}
		\end{align*}
		for some $\delta=\delta(\varepsilon)$ sufficiently small. While the choice of $\vec{a}$ is a priori arbitrary, the $C^2$-closeness to a boosted sphere infers that the canonical candidate to check for this criteria is the (unique) vector $\vec{a}$ associated to $\Sigma$ via the Minkowski $4$-vector $\textbf{Z}$ as in Section \ref{subsec_prelim_quantativeestimates}. Indeed, this choice of $\vec{a}$ achieves the infimum for round spheres.
	\end{bem}
	
	\subsection{A Sobolev inequality in asymptotically Schwarzschildean lightcones}\label{subsec_solobolev_asymSchw}

	We now extend the Sobolev inequality from the Minkowski lightcone in a direct way:
	Let $\Sigma=\Sigma_\omega$ be a spacelike cross section in an asymptotically Schwarzschildean lightcone. If $\omega\ge r_0$ sufficiently large, then $\gamma=\gamma_\omega$ is uniformly equivalent to the conformally round metric $\widetilde{\gamma}_\omega=\omega^2\widehat{\gamma}$. Note that $\widetilde{\gamma}$ is the induced metric of the corresponding spacelike cross section $\widetilde{\Sigma}=\widetilde{\Sigma}_\omega$ in the Minkowski lightcone, and let $\widetilde{R}$ denote the scalar curvature of $\widetilde{\gamma}_\omega$. Applying the considerations of the previous section to $\widetilde{\Sigma}$ we immediately obtain the following:
	
	\begin{prop}\label{prop_sobolev_asymSchw}
		Let $\Sigma$ be a spacelike cross section of an asymptotically Schwarzschildean lightcone $\mathcal{N}$ with $\mathcal{H}^2>0$. If 
		\[
			\btr{\mathcal{H}^2-2\widetilde{R}}\le \varepsilon\mathcal{H}^2
		\]
		for some $\varepsilon<1$, then $\Sigma$ carries a Sobolev inequality with respect to $\mathcal{H}^2$ with constant $C(c_0(\Sigma),\varepsilon)$ provided $\omega\ge r_0$ sufficiently large, where $C(c_0(\Sigma),\varepsilon)$ is a constant only depending on $c_0(\Sigma)$ and $\varepsilon$ and where we define $c_0(\Sigma):=c_0(\widetilde{\Sigma})$ by slight abuse of notation.
	\end{prop}
	
	For later convenience, we provide a direct comparison between $\mathcal{H}^2$ and $\widetilde{\scal}$ in the following:
	
	\begin{lem}\label{lem_comparrisonRtildemathcalH}
		Let $\mathcal{N}$ be an asymptotically Schwarzschildean lightcone, $\Sigma$ a spacelike cross section in $\mathcal{N}$. If $\omega\ge r_0$ sufficiently large, then
		\begin{align*}
			\btr{\mathcal{H}^2-2\widetilde{R}+\frac{8m}{\omega^3}}&\le \frac{C}{\omega^4}\left(1+\newbtr{\widehat{\nabla}\ln\omega}_{\hatgamma}^2+\newbtr{\widehat{\nabla}^2\ln\omega}_{\hatgamma}\right),\\
			\btr{\nabla\left(\mathcal{H}^2-2\widetilde{R}+\frac{8m}{\omega^3}\right)}_\gamma&\le \frac{C}{\omega^5}\left(\newbtr{\widehat{\nabla}^3\ln\omega}_{\hatgamma}+\left(1+\newbtr{\widehat{\nabla}\ln\omega}_{\hatgamma}\right)\left(1+\newbtr{\widehat{\nabla}\ln\omega}_{\hatgamma}^2+\newbtr{\widehat{\nabla}^2\ln\omega}_{\hatgamma}\right)\right)
		\end{align*}
		for some constant $C$ independent of $\Sigma$.
	\end{lem}
	\begin{proof}
		Given $\omega$, we recall that
		\begin{align*}
			2\widetilde{R}&=\frac{4}{\omega^2}-4\widetilde{\Delta}\ln\omega,\\
			\mathcal{H}^2&=\mathcal{H}^2_r-2\ul{\theta}\Delta\omega+\ul{\theta}^2\btr{\nabla\omega}^2-4\ul{\theta}\zeta_r(\nabla\omega)\\
			&=\mathcal{H}^2_r-2\ul{\theta}\Delta\ln\omega+\ul{\theta}\left(\ul{\theta}-\frac{2}{\omega}\right)\btr{\nabla\omega}^2-4\zeta_r(\nabla\omega),
		\end{align*}
		see for example \cite[Proposition 4.22 and Proposition 7.3]{wolff_thesis}. Using Definition \ref{defi_asymclassS} and Lemma \ref{lem_prelim_backgroundfol}, we hence find that
		\begin{align*}
			\mathcal{H}^2-2\widetilde{R}+\frac{8m}{\omega^3}=&\,O_{2,2}(r^{-4})\cdot \widehat{\Delta}\ln\omega+2\ul{\theta}\omega\left(\widetilde{\Delta}-\Delta\right)(\ln\omega)\\
			&\,+O_{1,1}(r^{-4})+\nabla\omega^AO_{2,2}(r^{-3})_A+\nabla\omega^A\nabla\omega^BO_{3,3}(r^{-2})_{AB},		
		\end{align*}
		where we used that $\widetilde{\Delta}=\omega^{-2}\widehat{\Delta}$ and where we use the same notation as in Proposition \ref{prop_appendix_curvatureidentities}. Observe that
		\begin{align*}
			\left(\widetilde{\Delta}-\Delta\right)(\ln\omega)
			=&\,(\widetilde{\gamma}^{ij}-\gamma^{ij})\widetilde{\nabla}^2_{ij}\ln\omega+\gamma^{ij}\widetilde{Q}_{ij}^k\d\ln\omega_k\\
			=&\,O_{3,3}(r^{-4})^{IJ}\left(\widehat{\nabla}^2_{IJ}\ln\omega-2\left(\d\ln\omega\otimes\d\ln\omega-\frac{1}{2}\newbtr{\widehat{\nabla}\ln\omega}_{\hatgamma}^2\widehat{\gamma}\right)_{IJ}\right)\\
			&\,+\gamma^{ij}\widetilde{Q}_{ij}^k\d\ln\omega_k,
		\end{align*}
		where the difference tensor $\widetilde{Q}$ satisfies
		\begin{align*}
			\widetilde{Q}_{IJ}^K=&\,\d\ln\omega_IO_2{2,3}(r^{-2})^K_J+\d\ln\omega_JO_2{2,3}(r^{-2})^K_I\\
			&\,-\d\ln\omega_LO_2{2,3}(r^{-2})^{KL}_{IJ}+O_2{3,2}(r^{-2})^K_{IJ}
		\end{align*}
		by Lemma \ref{lem_appx_differencetensors}. The claim then follows by Lemma \ref{lem_appendix_derivativesomega} and Lemma \ref{lem_appendix_derivativesTR}.
	\end{proof}
	
	\section{On the uniqueness of STCMC surfaces}\label{sec_uniqueness}
	
	We are now ready to prove Theorem \ref{thm_intro}. Before doing so, we will give a precise statement of the claim. We first define an a priori class of surfaces that is significantly weaker than in \cite[Definition 3.1]{kroenckewolff}:
	
	\begin{defi}\label{defi_classS}
	Let $0<\alpha<\frac{1}{2}$, $k\ge 1$, and let $\vec{v}\in(\R_{>0})^{k+1}:=\{\vec{v}\in\R^{k+1}\vert\vec{v}_l>0\}$. We say an STCMC surface $\Sigma$ with $\mathcal{H}^2>0$ lies in $\mathcal{S}_\alpha^k[\vec{v}]$ if
	\begin{align*}
		\omega\ge \vec{v}_1\left(\mathcal{H}^2\right)^{\alpha-\frac{1}{2}},\\
		\newbtr{\widehat{\nabla}^l\ln\omega}\le \vec{v}_{l+1}\left(\mathcal{H}^2\right)^{-\alpha\cdot l}
	\end{align*}
	for all $1\le l\le k$. For $k=1$, $0<\varepsilon<1,\,\kappa >0$, we additionally say that an STCMC surface $\Sigma$ with $\mathcal{H}^2>0$ lies in $\mathcal{S}^1_{\alpha, \varepsilon,\kappa}[v_1,v_2,V_1,V_2]$ if $\Sigma$ lies in $\mathcal{S}^1_\alpha[(v_1,v_2)]$,
	\begin{align*}
		\newbtr{\mathcal{H}^2-2\widetilde{R}}+\btr{\spann{\ul{\theta}^{-1}\accentset{\circ}{\ul{\chi}},\accentset{\circ}{A}}}\le \varepsilon\mathcal{H}^2,\text{ }\newbtr{\nabla\tau}_\gamma\le V_1(\mathcal{H}^2)^{1+\kappa}
	\end{align*}
	and $c_0(\Sigma)\le V_2$.\footnote{The above implies $\widetilde{R}>0$, which ensures that $c_0(\Sigma)$ as in Section \ref{section_sobolev} is well-defined.}
	\end{defi}
	By Proposition \ref{prop_sobolev_asymSchw}, it is straightforward to see that any $\Sigma$ in $\mathcal{S}^1_{\alpha, \varepsilon,\kappa}[v_1,v_2,V_1,V_2]$ carries a Sobolev inequality with respect to $\mathcal{H}^2$ in the sense of Definition \ref{defi_sobolev} with constant uniformly controlled by the given parameters.
	
	We now state the precise uniqueness statement for STCMC surfaces proven here:
	\begin{thm}\label{thm_main1}
		Let $\mathcal{N}$ be an asymptotically Schwarzschildean lightcone of mass $m>0$. Let $\kappa_0>0$. There exist positive constants $\alpha_0=\alpha_0(\kappa_0)<\frac{1}{2}$, $\varepsilon_0<1$ such that for all positive constants $v_1,v_2,V_1,V_2$, and  all $0<\alpha< \alpha_0$, $0<\varepsilon<\varepsilon_0$, $\kappa\ge \kappa_0$ there exists $0<\mathcal{H}_0^2=\mathcal{H}^2_0(m,\alpha,\varepsilon, \kappa, v_1,v_2,V_1,V_2)$ such that the following holds: For any $0<\mathfrak{H}^2\le \mathcal{H}^2_0$ there exists a unique STCMC surface $\Sigma$ in $\mathcal{S}_{\alpha,\varepsilon,\kappa}[v_1,v_2,V_1,V_2]$ with $\mathcal{H}^2=\mathfrak{H}^2$.
	\end{thm}
	\begin{bem}\label{bem_main1}
		Note that the additional assumptions in $\mathcal{S}^1_{\alpha, \varepsilon,\kappa}[v_1,v_2,V_1,V_2]$ are trivially satisfied in spherical symmetry, i.e., $\mathcal{S}^1_{\alpha, \varepsilon,\kappa}[v_1,v_2,V_1,V_2]=\mathcal{S}^1_{\alpha}[(v_1,v_2)]$. Hence, the assumptions significantly weaken in spherical symmetry, in particular in the exact Schwarzschild case. However, we recall that the uniqueness of STCMC surfaces (without any additional assumptions on the individual cross section) is known in spherical symmetry by the work of Chen--Wang \cite{chenwang} if the ambient spacetime satisfies the null energy condition.
	\end{bem}
	To give some intuition to the nature of the statement, we make the following observations:
	
	\begin{lem}\label{lem_classS_observation}
		Let $\mathcal{N}$ be an asymptotically Schwarzschildean lightcone of mass $m$, $\Sigma\subseteq \mathcal{N}$ an STCMC surface. 
		\begin{enumerate}
			\item[(i)] For all $1>\varepsilon>0$, $0<\alpha<\frac{1}{6}$, and positive constants $v_1,v_2,v_3$, there exists $\mathcal{H}^2_0=\mathcal{H}^2_0(m,\alpha,\varepsilon,\kappa,v_1,v_2,v_3)$ such that the following holds: If $0<\mathcal{H}^2\le \mathcal{H}^2_0$ and $\Sigma$ lies in $ \mathcal{S}^2_\alpha[(v_1,v_2,v_3)]$, then $\Sigma$ lies in $ \mathcal{S}^1_{\alpha,\varepsilon,1-6\alpha}[v_1,v_2,C(v_1,v_2,v_3),c_0(\Sigma)]$.
			\item[(ii)] Let $\alpha<\frac{1}{4}$, and let $\rho$ denote the area radius of $\Sigma$. Assume $\btr{\ln(\omega\cdot b_{\rho,\vec{a}}^{-1})}\le c$,
			\[
				\newbtr{\widehat{\nabla}^l\ln(\omega\cdot b_{\rho,\vec{a}}^{-1})}\le \frac{\delta}{(1+\btr{\vec{a}}^2)^{\frac{l}{2}}}
			\]
			for $1\le l\le 2$, where $\delta$ is a (fixed) suitably small constant, and $\newbtr{\vec{a}}\le C\rho^{2\alpha}$. Then $\mathcal{H}^2>0$ and $\Sigma$ lies in $\mathcal{S}^2_{\alpha}[(v_1,v_2,v_3)]$ for some constants $v_i=v_i(m,c)$ provided $\rho\ge \rho_0=\rho_0(m,c,\alpha)$ sufficiently large, and $c_0(\Sigma)\le 1$.
		\end{enumerate}
	\end{lem}
	
	\begin{bem}\,\label{bem_classS_observation}
		\begin{enumerate}
			\item[(i)] In view of Lemma \ref{lem_classS_observation}, Theorem \ref{thm_main1} yields a uniqueness statement for STCMC surfaces sufficiently close to a boosted sphere $b_{\rho,\vec{a}}$ under a mild control on $\btr{\vec{a}}$. In contrast, the uniqueness of STCMC surfaces established by the authors in their previous work \cite[Proposition 3.15]{kroenckewolff}, applies only to surfaces sufficiently close to a centered sphere in the sense that $\btr{\vec{a}}<<1$. In this sense, the uniqueness provided in Theorem \ref{thm_main1} constitutes a significant improvement.
			\item[(ii)] In the proof of Claim (ii), we will in fact show that $\widetilde{\rho}^2\widetilde{R}\ge 1$ by virtue of the argument below. As a lower bound on the (rescaled) scalar curvature is preserved under Lorentz transformations, the definition of $c_0(\Sigma)$ also yields the bound
			\[
				c_0(\Sigma)\le C\left(\btr{\vec{a}},\norm{\widehat{\nabla}\ln(\omega\,b_{\rho,\vec{a}}^{-1})}_{C^0(\,\widehat{\gamma}\,)}\right),
			\]
			and it is easy to check that $\Sigma$ already lies in $\mathcal{S}^2_{\alpha}[(v_1,v_2,v_3)]$ if $\newbtr{\widehat{\nabla}^l\ln(\omega\cdot b_{\rho,\vec{a}}^{-1})}\le C\rho^{2l\alpha}$. Hence, under the assumption that $\btr{\vec{a}}$ is uniformly bounded, Theorem \ref{thm_main1} applies also if $\Sigma$ is merely comparable to a boosted sphere in $C^2$ (with some uniform constant $C$).
		\end{enumerate}
	\end{bem}
	\begin{proof}[Proof of Lemma \ref{lem_classS_observation}]\,
		\begin{enumerate}
			\item[(i)] By definition, $\Sigma$ lies in $S_\alpha^1[(v_1,v_2)]$. Moreover, by Lemma \ref{lem_prelim_tau}, Lemma \ref{lem_prelim_A}, and Lemma \ref{lem_comparrisonRtildemathcalH}, it is straightforward to check that the claim holds for $\mathcal{H}\le\mathcal{H}^2_0$ sufficiently small.
			\item[(ii)] Note that 
			\[
				b_{\rho,\vec{a}}\ge \frac{\rho}{2\sqrt{1+\btr{\vec{a}}^2}},\text{ and }\newbtr{\widehat{\nabla}^l\ln b_{\rho,\vec{a}}}\le C_l\btr{\vec{a}}^l.
			\]
			In particular, we find 
			\[
				\omega\ge C(c,C)\rho^{1-2\alpha},\text{ and }\newbtr{\widehat{\nabla}^l\ln \omega}\le C(c,C)\rho^{2\alpha l}
			\]
			for $l=1,2$. Let $\widetilde{\rho}$ denote the area radius of $\widetilde{\gamma}_\omega$, and we recall that by \cite[Equation (A.1)]{kroenckewolff} we have $\btr{\rho^2-\widetilde{\rho}^2}\le C$ provided $\rho\ge \rho_0$ sufficiently large. As $\mathcal{H}^2$ is constant and $\d\widetilde{\mu}=\omega^2\d\widehat{\mu}$,
			we find
			\begin{align*}
				4\pi\rho^2\left(1+\frac{\widetilde{\rho}^2-\rho^2}{\rho^2}\right)\mathcal{H}^2&=
				\int_{\Sbb^2}\mathcal{H}^2\d\widetilde{\mu}\\
				&=16\pi-\int_{\Sbb^2}\frac{8m}{\omega}\d\widehat{\mu}+\int_{\Sbb^2}\omega^2\left(\mathcal{H}^2-2\widetilde{R}+\frac{8m}{\omega^3}\right)\d\widehat{\mu},
			\end{align*}
			where we used the Gauss--Bonnet theorem. By Lemma \ref{lem_comparrisonRtildemathcalH}, there exists a uniform constant $C>4$ such that
			\[
			C^{-1}\le \mathcal{H}^2\rho^2\le C
			\]
			provided $\rho\ge \rho_0$ sufficiently large. Hence, $\mathcal{H}^2>0$ and $\Sigma$ lies in $S_\alpha^2[(v_1,v_2,v_3)]$. By (i), $\widetilde{R}>0$ provided $\rho\ge\rho_0$. In particular $c_0(\Sigma)$ is well-defined. In view of Remark \ref{bem_sobolev_mink2}, choosing $\delta$ sufficiently small, we have
			\[
				c_0(\Sigma)\le 1,
			\]
			concluding the proof.
		\end{enumerate}
	\end{proof}
	
	Combining Theorem \ref{thm_main1}, Lemma \ref{lem_classS_observation} and Remark \ref{bem_classS_observation}, we obtain an improved uniqueness of the foliation by STCMC surfaces constructed by the authors in \cite[Theorem 5.2]{kroenckewolff}:
	
	\begin{thm}\label{thm_main2}
		Let $\mathcal{N}$ be an asymptotically Schwarzschildean lightcone of mass $m>0$. There exists $0<\alpha_0<\frac{1}{2}$ and an asymptotic foliation $(\Sigma_\sigma)_{\sigma\in(\sigma_0,\infty)}$ of $\mathcal{N}$ by STCMC surfaces that is unique in the following sense: Let $(\widetilde{\Sigma}_{s})_{s\in(s_0,\infty)}$ be an asymptotic foliation by STCMC surfaces and let $(\vec{a}_{s})_{s}$ be a family of vectors in $\R^3$ such that either
		\begin{align*}
			&\btr{\vec{a}_s}\le C\rho_{s}^{2\alpha}&&&&\btr{\vec{a}_s}\le C\\
			&\text{\,\,\,    \,\, and}&&\text{ or }&&\text{\,\,\,   \,\,  and}\\
			&\newbtr{\widehat{\nabla}^l\ln(\omega_s\cdot b_{\rho_{s},\vec{a}_{s}}^{-1})}\le \frac{\delta}{(1+\btr{\vec{a}_{s}}^2)^{\frac{l}{2}}}&&&& \newbtr{\widehat{\nabla}^l\ln(\omega_{s}\cdot b_{\rho_{s},\vec{a}_{s}}^{-1})}\le C
		\end{align*}
		for $0<\alpha\le \alpha_0$, $l=1,2$, where $\delta$ is a (fixed) suitably small constant and $C$ is some uniform constant independent of $s$. Then there exist $s_1\ge s_0$ sufficiently large such that the STCMC surfaces $\widetilde{\Sigma}_{s}$ are leaves of the foliation $(\Sigma_\sigma)_{\sigma\in(\sigma_0,\infty)}$ for all $s\ge s_1$. Moreover, the asymptotic foliation by STCMC surfaces (unique in this sense) is an asymptotically flat background foliation and has $E_B=m$ and $\vec{P}_B=\vec{0}$.
	\end{thm}
	\begin{bem}\,
		\begin{enumerate}
			\item[(i)] In \cite{kroenckewolff}, the authors established the uniqueness of the asymptotic foliation by STCMC surfaces among spacelike cross sections close to centered spheres with $\btr{\vec{a}}\le C\sigma^{-1}$. Theorem \ref{thm_main2} now extends this uniqueness to STCMC foliations sufficiently comparable to boosted spheres with only some mild control on $\btr{\vec{a}}$. 
			\item[(ii)] Recall our discussion about asymptotically flat background foliations in Section \ref{subsubsec_prelim_AF}. While the condition
			\[
				\lim_{r\to\infty}r^2K_r=1
			\]
			does in general not imply $C^2$-convergence, see Appendix \ref{appendix_roundness} below, one would still heuristically expect that the foliation is comparable to boosted spheres (after rescaling) and that $\btr{\vec{a}}$ remains bounded along the foliation. It thus seems natural to make this additional assumption for asymptotically flat background foliations which is likely generic for a suitably strong notion of asymptotic flatness. Assuming that this additional property is given, Theorem \ref{thm_main2} states that for an asymptotically Schwarzschildean lightcone of positive mass $m>0$ there is a unique asymptotically flat background foliation by STCMC surfaces, and that the Bondi energy agrees with the mass parameter $m$ and that the  Bondi linear momentum vanishes with respect to the foliation.
			\item[(iii)] Note that $\alpha_0<<\frac{1}{2}$. Indeed, we expect that one can construct counterexamples to Theorem \ref{thm_main1} for $\alpha$ close to $\frac{1}{2}$ in analogy to the work of Brendle--Eichmair \cite{brendleeichmair}.
		\end{enumerate}
	\end{bem}
	
	\subsection{Proof of Theorem \ref{thm_main1}}\label{subsec_uniqueness_proof}
	
	It remains to prove Theorem \ref{thm_main1}. We first observe the following:
	\begin{lem}\label{lem_proof_H2rho}
		Let $\mathcal{N}$ be an asymptotically Schwarzschildean lightcone of mass $m>0$, and let $\Sigma$ be an STCMC surface with $\mathcal{H}^2>0$ that lies within $\mathcal{S}_{\alpha,\varepsilon,\kappa}^1[v_1,v_2,V_1,V_2]$. If $\alpha<\frac{1}{6}$ and $\varepsilon<\varepsilon_0$ sufficiently small, then there exists a uniform constant $C>4$ such that
		\[
			C^{-1}\le \rho^2\mathcal{H}^2\le C
		\]
		provided $\mathcal{H}^2\le \mathcal{H}^2_0(m,\alpha,v_1,v_2)$.
	\end{lem}
	\begin{proof}
		We argue similarly to before. Using that $\mathcal{H}^2$ is constant, we find that
		\[
			4\pi\rho^2\mathcal{H}^2=16\pi+\int_\Sigma2\spann{\ul{\theta}\accentset{\circ}{\ul{\chi}},\accentset{\circ}{A}}-2\left(\overline{R}-2\overline{\Ric}(\ul{L},L)+\frac{1}{2}\overline{\Riem}(\ul{L},L,L,\ul{L})\right)\d\mu,
		\]
		where we used the Gauss Equation, Proposition \ref{prop_prelim_nullgauss}, and the Gauss--Bonnet theorem. As $\Sigma$ lies in $\mathcal{S}^1_{\alpha,\varepsilon,\kappa}[v_1,v_2,V_1,V_2]$, Proposition \ref{prop_appendix_curvatureidentities} together with Lemma \ref{lem_appendix_derivativesomega} and Lemma \ref{lem_appendix_derivativesTR}, we find
		\begin{align*}
			4\pi\rho^2\mathcal{H}^2&\ge\frac{16\pi}{1+\varepsilon+C(m,v_1)(\mathcal{H}^2)^{\frac{1}{2}-3\alpha}+C(v_1,v_2)(\mathcal{H}^2)^{1-6\alpha}},\\ 4\pi\rho^2\mathcal{H}^2&\le \frac{16\pi}{1-\varepsilon-C(m,v_1)(\mathcal{H}^2)^{\frac{1}{2}-3\alpha}-C(v_1,v_2)(\mathcal{H}^2)^{1-6\alpha}}.
		\end{align*}
		Hence, the claim holds for $\mathcal{H}^2\le \mathcal{H}^2_0$ sufficiently small.
	\end{proof}
	
	We note that by \cite[Proposition 3.15 and Theorem 4.8]{kroenckewolff}, we have existence and uniqueness of STCMC surfaces within the a priori class $B_\sigma(B_1,B_2,B_3)$, compare \cite[Definition 3.1]{kroenckewolff}, provided $\sigma\ge \sigma_0$ sufficiently large. It thus suffices to show that any STCMC surface in $\mathcal{S}_{\alpha,\varepsilon,\kappa}^1[v_1,v_2,V_1,V_2]$ lies in $B_\rho(B_1,B_2,B_3)$ provided $0<\mathcal{H}^2\le \mathcal{H}^2_0$ sufficiently small. More precisely, we need to show that 
	\begin{align}\label{eq_proof_goal1}
		\max\left(\mathcal{H}^2\rho^2,\text{ }\norm{\omega\rho^{-1}}_{C^3(\,\widehat{\gamma}\,)}\right)\le 10,
	\end{align}
	and
	\begin{align}\label{eq_proof_goal2}
		\btr{\omega-\rho}\le B_1,\text{ }\newbtr{\accentset{\circ}{A}}_\gamma\le \frac{B_2}{\rho^4},\text{ }\newbtr{\nabla\accentset{\circ}{A}}\le \frac{B_3}{\rho^5}.
	\end{align}
	Note that the a priori estimates \cite[Proposition 3.3]{kroenckewolff} readily imply the converse inclusion.
	
	We first establish the following consequence of the contracted null Simons' identity:
	\begin{lem}\label{lem_proof_sobolev1}
		Let $\mathcal{N}$ be an asymptotically Schwarzschildean lightcone of mass $m>0$, $\Sigma$ an STCMC surface with $\mathcal{H}^2>0$. Let $\kappa_0>0$. There exists $\alpha_0<\frac{1}{2}$, $\varepsilon_0<1$ and a uniform constant $C$ such that for all $0<\alpha<\alpha_0$, $0<\varepsilon<\varepsilon_0$, $\kappa\ge \kappa_0$
		\[	
		\int_\Sigma \mathcal{H}^2\newbtr{\accentset{\circ}{A}}^2+\newbtr{\nabla \newbtr{\accentset{\circ}{A}}}^2\d\mu \le C \int_\Sigma (\mathcal{H}^2)^{-1}\btr{G}^2\d\mu
		\]
		if $\Sigma$ lies in $\mathcal{S}^1_{\alpha,\varepsilon,\kappa}[v_1,v_2,V_1,V_2]$ provided $\mathcal{H}^2\le \mathcal{H}^2_0(m,\alpha,\varepsilon,\kappa,v_1,v_2, V_1)$, where $G$ is as in Proposition \ref{prop_appendix_NullSimonAsymSchwarzschild}. In addition, for any $p\ge 1$, we find that
		\[
		\left(\int_\Sigma \newbtr{\accentset{\circ}{A}}^{2p}\d\mu\right)^{\frac{1}{p}}\le C(p,V_2) \newbtr{\operatorname{supp}\newbtr{\accentset{\circ}{A}}}^{\frac{1}{p}}\int_\Sigma (\mathcal{H}^2)^{-1}\btr{G}^2\d\mu.
		\]
	\end{lem}
	
	\begin{proof}
		Since
		\begin{align*}
			\Delta\newbtr{\accentset{\circ}{A}}^2&=2\spann{\Delta\accentset{\circ}{A},\accentset{\circ}{A}}+2\newbtr{\nabla \accentset{\circ}{A}}^2,\\
			\nabla\newbtr{\accentset{\circ}{A}}^2&=2\spann{\nabla\accentset{\circ}{A},\accentset{\circ}{A}},
		\end{align*}
		we have $\newbtr{\nabla \newbtr{\accentset{\circ}{A}}}\le \newbtr{\nabla \accentset{\circ}{A}}$ and
		\[
		2\newbtr{\accentset{\circ}{A}}\Delta \newbtr{\accentset{\circ}{A}}+2\newbtr{\nabla \newbtr{\accentset{\circ}{A}}}^2=2\spann{\Delta\accentset{\circ}{A},\accentset{\circ}{A}}+2\newbtr{\nabla \accentset{\circ}{A}}^2
		\]
		at any point where $\newbtr{\accentset{\circ}{A}}\not=0$.
		Hence, at any point where $\newbtr{\accentset{\circ}{A}}\not=0$ we have 
		\begin{align}\label{eq_simon_point}
			\newbtr{\accentset{\circ}{A}}\Delta \newbtr{\accentset{\circ}{A}}\ge \spann{\Delta \accentset{\circ}{A},\accentset{\circ}{A}}.
		\end{align}
		In fact, as in Schoen--Simon--Yau \cite{schoensimonyau}, the above inequality holds everywhere in the distributional sense, i.e.,
		\begin{align}\label{eq_simon_distribution}
			\int_\Sigma \spann{\Delta\accentset{\circ}{A},\accentset{\circ}{A}}+\newbtr{\nabla \newbtr{\accentset{\circ}{A}}}^2\d\mu \le 0.
		\end{align}
		By Proposition \ref{prop_appendix_NullSimonAsymSchwarzschild} we have 
		\[
			\Delta\accentset{\circ}{A}_{ij}=f\accentset{\circ}{A}_{ij}-2\tau^k\nabla_k\accentset{\circ}{A}_{ij}+\left(F\ast \accentset{\circ}{A}\right)_{ij}+G_{ij}.
		\]
		Hence, combining this with \eqref{eq_simon_distribution} we have
		\[
			\int_\Sigma f\newbtr{\accentset{\circ}{A}}^2+\newbtr{\nabla \newbtr{\accentset{\circ}{A}}}^2\d\mu \le \int_\Sigma 2\tau^k\spann{\nabla_k\accentset{\circ}{A},\accentset{\circ}{A}}+\newbtr{\accentset{\circ}{A}}\cdot\left(\newbtr{F\ast \accentset{\circ}{A}}+\btr{G}\right)\d\mu.
		\]
		Notice that, as the integrand is trivial at all points where $\newbtr{\accentset{\circ}{A}}=0$, one finds
		\[
			\int_\Sigma 2\tau^k\spann{\nabla_k\accentset{\circ}{A},\accentset{\circ}{A}}\d\mu=\int_\Sigma2\tau^k\nabla_k\newbtr{\accentset{\circ}{A}}^2\d\mu=\int_\Sigma 2\tau^k\nabla_k\newbtr{\accentset{\circ}{A}}\cdot\newbtr{\accentset{\circ}{A}}\d\mu.
		\]
		Moreover, as $\Sigma$ lies in $\mathcal{S}_{\alpha,\varepsilon,\kappa}^1[v_1,v_2,V_1,V_2]$, Proposition \ref{prop_appendix_NullSimonAsymSchwarzschild} together with Lemma \ref{lem_prelim_tau}, Lemma \ref{lem_appendix_derivativesomega} and Lemma \ref{lem_appendix_derivativesTR} yield that 
		\begin{align*}
			\btr{\tau}&\le C(v_1,v_2)(\mathcal{H}^2)^{\frac{3}{2}-4\alpha},\\
			f&\ge \frac{1}{2}\mathcal{H}^2-2\varepsilon\mathcal{H}^2-C(v_1,v_2)(\mathcal{H}^2)^{\frac{3}{2}-4\alpha}-C(v_1,v_2)(\mathcal{H}^2)^{2-6\alpha},\\
			\newbtr{F\ast\accentset{\circ}{A}}&\le C \newbtr{\accentset{\circ}{A}}\left(\varepsilon\mathcal{H}^2+C(v_1,v_2)(\mathcal{H}^2)^{2-6\alpha}\right),
		\end{align*}
		where we also used that $\btr{\nabla\tau}\le \varepsilon\mathcal{H}^2$ by assumption provided $\mathcal{H}_0^2$ is sufficiently small in terms of $\varepsilon$, $\kappa$, and $V_1$. The first claim then follows by Youngs' inequality for $\alpha_0=\frac{1}{8}$, and $\varepsilon_0$, $\mathcal{H}^2_0$ sufficiently small. The second claim directly follows from Lemma \ref{lem_prelim_Sobolev_consequence}, where we note again that the Sobolev constant $c_S$ is uniformly controlled by assumption.
	\end{proof}
	
	We first prove an initial, pointwise upper bound on $\newbtr{\accentset{\circ}{A}}$:
	\begin{prop}\label{prop_proof_Abound_1}
		Let $\mathcal{N}$ be an asymptotically Schwarzschildean lightcone of mass $m>0$, $\Sigma$ an STCMC surface with $\mathcal{H}^2>0$. Let $\kappa_0>0$. There exists $\alpha_0<\frac{1}{2}$, $\varepsilon_0$, and a small number $\eta>0$, which only depends on $\kappa_0$, such that for all $0<\alpha<\alpha_0$, $0<\varepsilon<\varepsilon_0$, $\kappa\ge \kappa_0$
		\[	
		\newbtr{\accentset{\circ}{A}}\le (\mathcal{H}^2)^{1+\eta}
		\]
		if $\Sigma$ lies in $\mathcal{S}^1_{\alpha,\varepsilon,\kappa_0}[v_1,v_2,V_1,V_2]$ provided $\mathcal{H}^2\le \mathcal{H}^2_0(m,\alpha,\varepsilon,\kappa,v_1,v_2,V_1,V_2)$.
	\end{prop}
	\begin{bem}\label{bem_proof_boundA1}
		For later convenience, we remark that by virtue of the proof we have in fact shown that
		\[	
		\newbtr{\accentset{\circ}{A}}\le \frac{C(V_2)}{\mathcal{H}^2}\max\limits_\Sigma\left[\frac{C(m)}{\omega^5}\btr{\nabla\omega}^2+\mathcal{H}^2\left(\btr{\nabla\tau}+\btr{\tau}^2+\frac{C}{\omega^4}\left(1+\btr{\nabla\omega}^2\right)\right)+\frac{C}{\omega^6}\left(1+\btr{\nabla\omega}^4\right)\right]
		\]
		if $\Sigma$ lies in $\mathcal{S}^1_{\alpha,\varepsilon,\kappa}[v_1,v_2,V_1,V_2]$ provided $\mathcal{H}^2\le \mathcal{H}^2_0(m,\alpha,\varepsilon,\kappa,v_1,v_2,V_1,V_2)$.
	\end{bem}
	
	As in \cite{huiskenyau}, we use the Lemma of Stampacchia, which we state here for the convenience of the reader:
	
	\begin{lem}[Lemma of Stampacchia]\label{lem_stampaccia}
		Let $\varphi\colon [k_0,\infty)\to\R$ be a non-negative, non-increasing function such that for positive constants $\alpha,\gamma,c$ with $\gamma>1$ we have
		\[
		\varphi(h)\le \frac{c}{(h-k)^\alpha}\varphi(k)^\gamma
		\]
		for all $h>k\ge k_0$. Then $\varphi(k_0+d)=0$ where
		\[
		d^\alpha=c2^{\frac{\alpha\gamma}{\gamma-1}}\varphi(k_0)^{\gamma-1}.
		\]
	\end{lem}

	\begin{proof}[Proof of Proposition \ref{prop_proof_Abound_1}]
		For convenience, we define the function $u=\newbtr{\accentset{\circ}{A}}$, and further define $u_k:=\max(u-k,0)$, and the subset $A(k)=\{u>k\}=\{u_k>0\}$ of $\Sigma$. Let $p> 1$, $k\ge 0$. Multiplying \eqref{eq_simon_point} with $u_k^p$, the contracted null Simons' identity, Proposition \ref{prop_appendix_NullSimonAsymSchwarzschild}, yields
		\begin{align}\label{eq_proof_a02_1}
			u_k^pu\Delta u\ge fu^2u_k^2-2\tau(\nabla u)u_k^pu+u_k^p\spann{\accentset{\circ}{A},F\ast\accentset{\circ}{A}+G},
		\end{align}
		where the inequality again holds in the distributional sense.
		Integration by parts yields
		\begin{align*}
			\int_{A(k)}u_k^pu\Delta u\d\mu+\int_{A(k)}(u_k^p+puu_k^{p-1})\btr{\nabla u}^2\d\mu=0.
		\end{align*}
		Thus, by similar estimates as in the proof of Lemma \ref{lem_proof_sobolev1}, \eqref{eq_proof_a02_1} yields
		\[
			\int_{A(k)}(u_k^p+puu_k^{p-1})\btr{\nabla u}^2+\mathcal{H}^2u^2u_k^p\d\mu\le C\int_{A(k)}\btr{G}u_k^pu\d\mu
		\]
		provided $\varepsilon_0$ and $\mathcal{H}^2_0$ are sufficiently small.
		Now consider $U:=u_k^{\frac{p}{2}+1}$ and as $p> 1$, we observe that 
		\[
		\frac{1}{2}\btr{\nabla U}^2\le 2p(p+1)u_k^p\btr{\nabla u}^2
		\]
		on $A(k)$. As $u\ge u_k$, we conclude that
		\[
		\int_{A(k)}\btr{\nabla U}^2+\mathcal{H}^2 U^2\d\mu\le C(p)\int_{A(k)} \btr{G}u u_k^p\d\mu.
		\]
		Thus, Proposition \ref{prop_sobolev_asymSchw} and Lemma \ref{lem_prelim_Sobolev_consequence} yield
		\begin{align}\label{eq_proof_a02_2}
			\left(\int_{A(k)}\btr{U}^{2q}\d\mu \right)^{\frac{1}{q}}\le C(p,q,V_2)\btr{A(k)}^\frac{1}{q}\int_{A(k)}\btr{G}u u_k^p\d\mu
		\end{align}
		for any $1\le q<\infty$. Similarly, Lemma \ref{lem_prelim_Sobolev_consequence} implies that
		\begin{align}\label{eq_proof_a02_3}
			\left(\int_{A(k)} u_k^{2a}\right)^{\frac{1}{a}}\le C(a,V_2) \btr{A(k)}^{\frac{1}{a}}\int_{\Sigma}(\mathcal{H}^2)^{-1}\btr{G}^2\d\mu,
		\end{align}
		for any $1\le a<\infty$, where we used $u_k\le u$, $\nabla u_k=\nabla u$ on $A(k)$ and the first inequality in Lemma \ref{lem_proof_sobolev1}.
		
		To control the remaining integral term in \eqref{eq_proof_a02_2} on the right hand side, let $w>1$ and using the H\"older inequality twice we find
		\begin{align*}
			\int_{A(k)}u u_k^p\d\mu &\le \left(\int_{A(k)}u^2\d\mu\right)^{\frac{1}{2}}\left(\left(\int_{A(k)}u_k^{2p}\d\mu\right)^\frac{1}{p}\right)^{\frac{p}{2}}\\
			&\le \btr{A(k)}^{\frac{w-1}{w}}\left(\int_\Sigma u^{2w}\d\mu\right)^{\frac{1}{2w}} \left(\left(\int_{A(k)}u_k^{2p}\d\mu\right)^\frac{1}{p}\right)^{\frac{p}{2}}
		\end{align*}
		Applying Lemma \ref{lem_proof_sobolev1} to the $u$-integral and \eqref{eq_proof_a02_3} to the $u_k$-integral, we find that
		\[
		\int_{A(k)}u u_k^p\d\mu \le C(p,w,V_2) \btr{A(k)}^{\frac{2w-1}{2w}}\btr{\Sigma}^{\frac{1}{2w}}\left(\int_\Sigma (\mathcal{H}^2)^{-1}\btr{G}^2\d\mu\right)^\frac{p+1}{2}.
		\]
		Using that $\btr{\Sigma}\le C(\mathcal{H}^2)^{-1}$ by Lemma \ref{lem_proof_H2rho}, combining \eqref{eq_proof_a02_2} for $q=1$ with the above inequality, we find
		\[
		\int_{A(k)} f^2\d\mu \le C(p,w,V_2)\btr{\Sigma}^{\frac{1}{2w}} (\mathcal{H}^2)^{-(p+1)}\mathcal{G}^{p+2}\btr{A(k)}^{\frac{4w-1}{2w}},
		\]
		where $\mathcal{G}:=\max\limits_\Sigma\btr{G}$.
		Finally, note that for any $h>k$ we find
		\begin{align*}
			(h-k)^{p+2}\btr{A(h)}&\le \int_{A(h)}\btr{u-k}^{p+2}\\&\le \int_{A(k)}f^2\d\mu\\
			&\le C(p,w,V)\btr{\Sigma}^{\frac{1}{2w}} (\mathcal{H}^2)^{-(p+1)}\mathcal{G}^{p+2}\btr{A(k)}^{\frac{4w-1}{2w}}.
		\end{align*}
		Hence, the Lemma of Stampacchia \ref{lem_stampaccia} yields that $\btr{A(d)}=0$ for
		\[
		d^{p+2}=C(p,w,V_2)\btr{\Sigma}^{\frac{1}{2w}} (\mathcal{H}^2)^{-(p+1)}\mathcal{G}^{p+2}\btr{A(0)}^{\frac{2w-1}{2w}}.
		\]
		In particular, $u\le d$ almost everywhere and as $u$ is continuous we obtain the sup-bound $u\le d$. Invoking Lemma \ref{lem_proof_H2rho} again, we observe that
		\[
		\btr{\Sigma}^{\frac{1}{2w}}\btr{A(0)}^{\frac{2w-1}{2w}}\le \btr{\Sigma}\le C(\mathcal{H}^2)^{-1},
		\]
		so fixing any choice of $w$, e.g. $w=2$, we obtain 
		\[
		d\le C(p,V_2) \mathcal{G}(\mathcal{H}^2)^{-1},
		\]
		where we may now also take $p=2$.
		Note that by Proposition \ref{prop_appendix_NullSimonAsymSchwarzschild}, we have
		\[
		\btr{G}\le \frac{C(m)}{\omega^5}\btr{\nabla\omega}^2+\mathcal{H}^2\left(\newbtr{\nabla\tau}+\btr{\tau}^2+C\omega^{-4}\left(1+\btr{\nabla\omega}^2\right)\right)+\frac{C}{\omega^6}\left(1+\btr{\nabla\omega}^4\right).
		\]
		As $\Sigma$ lies in $\mathcal{S}^1_{\alpha,\varepsilon,\kappa}[v_1,v_2,V_1,V_2]$, we find
		\[
			\mathcal{G}\le C(m,v_1,v_2)(\mathcal{H}^2)^{\frac{5}{2}-7\alpha}+V_1(\mathcal{H}^2)^{2+\kappa}+C(v_1,v_2)\left((\mathcal{H}^2)^{4-8\alpha}+(\mathcal{H}^2)^{3-6\alpha}+(\mathcal{H}^2)^{3-10\alpha}\right).
		\]
		In particular, for $\alpha_0=\frac{1}{14}$, there exists $\eta(\alpha,\kappa)\le \frac{\kappa_0}{2}$ for which the claim holds. Thus, taking $\alpha_0=\frac{1}{16}<\frac{1}{14}$ the claim holds for some $\eta>0$ only depending on $\kappa_0$.
	\end{proof}
	
	Next, we aim to obtain an initial bound on $\newbtr{\nabla\accentset{\circ}{A}}$. We first establish the following technical lemma:
	
	\begin{lem}\label{lem_proof_deltaA_deltanablaA}
		Let $\mathcal{N}$ be an asymptotically Schwarzschildean lightcone of mass $m>0$, $\Sigma$ an STCMC surface with $\mathcal{H}^2>0$. Let $\kappa_0>0$. There exists $\alpha_0<\frac{1}{2}$, $\varepsilon_0$ and universal constants $C_1$, $C_2$ such that for all $0<\alpha<\alpha_0$, $0<\varepsilon\le \varepsilon_0$, $\kappa\ge \kappa_0$
		\begin{align*}
			\Delta\newbtr{\accentset{\circ}{A}}^2&\ge \frac{3}{2}\newbtr{\nabla\accentset{\circ}{A}}^2-2\btr{G}\newbtr{\accentset{\circ}{A}},\\
			\Delta\newbtr{\nabla\accentset{\circ}{A}}^2&\ge -C_1\mathcal{H}^2\newbtr{\nabla\accentset{\circ}{A}}^2-C_2(\mathcal{H}^2)^{-1}\left(f_1+f_2\newbtr{\accentset{\circ}{A}}^2\right)
		\end{align*}
		if $\Sigma$ lies in $S_{\alpha,\varepsilon,\kappa}^1[v_1,v_2,V_1,V_2]$ provided $\mathcal{H}^2\le \mathcal{H}_0^2(m,\alpha,\varepsilon,\kappa,v_1,v_2,V_1,V_2)$, where $G$ is as in Proposition \ref{prop_appendix_NullSimonAsymSchwarzschild} and $f_1$, $f_2$ are positive functions with
		\begin{align*}
			f_1\le&\,\frac{C(m)}{\omega^{12}}\newbtr{\widehat{\nabla}\ln\omega}_{\hatgamma}^2\left(\newbtr{\widehat{\nabla}^2\ln\omega}_{\hatgamma}^2+\newbtr{\widehat{\nabla}\ln\omega}_{\hatgamma}^4\right)+\frac{C}{\omega^{14}}\left(1+\newbtr{\widehat{\nabla}\ln\omega}_{\hatgamma}^{10}\right)\\
			&\,+\frac{C}{\omega^{10}}(\mathcal{H}^2)^2\left(1+\newbtr{\widehat{\nabla}\ln\omega}_{\hatgamma}^6\right)+\frac{C}{\omega^6}(\mathcal{H}^2)^4\left(1+\newbtr{\widehat{\nabla}\ln\omega}_{\hatgamma}^2\right),\\
			f_2\le&\,\frac{C(m)}{\omega^8}\newbtr{\widehat{\nabla}\ln\omega}_{\hatgamma}^2+\frac{C}{\omega^6}(\mathcal{H}^2)^2\left(1+\newbtr{\widehat{\nabla}\ln\omega}_{\hatgamma}^2\right)+\frac{C}{\omega^{10}}\left(1+\newbtr{\widehat{\nabla}\ln\omega}_{\hatgamma}^6\right).
		\end{align*}
	\end{lem}
	
	\begin{proof}
		We first recall the well-known formula
		\[
		\Delta\btr{T}^2=2\spann{\Delta T,T}+2\btr{\nabla T}^2,
		\]
		for a $(0,k)$-tensor $T$ and the commutator identity
		\begin{align*}
			2\left(\spann{\nabla \Delta T,\nabla T}-\spann{\Delta\nabla T,\nabla T}\right)
			=&\, -\frac{1}{2}\scal \newbtr{\nabla T}^2+2T\left(\nabla \scal,\dive T\right)-2\spann{ T\otimes \d\scal, \nabla T}\\
			&\,+2\scal \newbtr{\dive T}^2-2\scal \nabla_m T_{in}\nabla^iT^{mn},
		\end{align*}
		for a symmetric $(0,2)$-tensor $T$. Using Proposition \ref{prop_appendix_NullSimonAsymSchwarzschild} and estimating similar to before, it is straightforward to obtain 
		\[
			\Delta\newbtr{\accentset{\circ}{A}}^2\ge \frac{3}{2}\newbtr{\nabla\accentset{\circ}{A}}^2-2\btr{G}\newbtr{\accentset{\circ}{A}}
		\]
		provided $\mathcal{H}^2\le\mathcal{H}^2_0$ sufficiently small. Taking a derivative of the contracted null Simons' identity, Proposition \ref{prop_appendix_NullSimonAsymSchwarzschild}, and using the commutator identity as above, we obtain
		\begin{align*}
			\Delta \newbtr{\nabla\accentset{\circ}{A}}^2\ge &\,
			(2f-4\btr{\nabla\tau}-2\btr{\tau}^2-C\btr{\scal})\newbtr{\nabla\accentset{\circ}{A}}^2\\
			&\,-\left(\btr{\nabla f}\newbtr{\accentset{\circ}{A}}+2\newbtr{\nabla (F\ast\accentset{\circ}{A})}+\btr{\nabla\scal}\newbtr{\accentset{\circ}{A}}+\btr{\nabla G}\right)\newbtr{\nabla\accentset{\circ}{A}}
		\end{align*}
		Note that by the Gauss Equation, Proposition \ref{prop_prelim_nullgauss}, we have
		\[
			\scal=\frac{1}{2}\mathcal{H}^2-\spann{\ul{\theta}^{-1}\accentset{\circ}{\ul{\chi}},\accentset{\circ}{A}}+\left(\overline{\scal}-2\overline{\Ric}(\ul{L},L)+\frac{1}{2}\overline{\Riem}(\ul{L},L,L,\ul{L})\right).
		\]
		As $\Sigma$ lies in $\mathcal{S}^1_{\alpha,\varepsilon,\kappa}[v_1,v_2,V_1,V_2]$, the curvature identities, Proposition \ref{prop_appendix_curvatureidentities}, readily imply that $\btr{\scal}\le \mathcal{H}^2$ provided $\mathcal{H}^2\le \mathcal{H}^2_0$ sufficiently small. Arguing in a similar fashion for the other terms in the first line, we find
		\begin{align*}
			\Delta \newbtr{\nabla\accentset{\circ}{A}}^2\ge &\,
			-c_1\mathcal{H}^2\newbtr{\nabla\accentset{\circ}{A}}^2-\left(\btr{\nabla f}\newbtr{\accentset{\circ}{A}}+2\newbtr{\nabla (F\ast\accentset{\circ}{A})}+\btr{\nabla\scal}\newbtr{\accentset{\circ}{A}}+\btr{\nabla G}\right)\newbtr{\nabla\accentset{\circ}{A}}
		\end{align*}
		for some universal constant $c_1$. As $\mathcal{H}^2$ is constant, the Gauss equation further implies that
		\begin{align*}
			\btr{\nabla\scal}\le &\,\newbtr{\ul{\theta}^{-1}\accentset{\circ}{\ul{\chi}}}\newbtr{\nabla\accentset{\circ}{A}}+\newbtr{\nabla\left(\ul{\theta}^{-1}\accentset{\circ}{\ul{\chi}}\right)}\newbtr{\accentset{\circ}{A}}+\frac{C(m)}{\omega^4}\btr{\nabla\omega}\\
			&\,+\btr{\nabla\left(O_{2,2}(r^{-4})+\nabla\omega^AO_{2,2}(r^{-3})_A+\nabla\omega^A\nabla\omega^BO_{2,2}(r^{-2})_{AB}\right)},
		\end{align*}
		where we once again used the curvature identities, Proposition \ref{prop_appendix_curvatureidentities}. Similarly, we find that
		\begin{align*}
			\btr{\nabla f}\le &\,\newbtr{\ul{\theta}^{-1}\accentset{\circ}{\ul{\chi}}}\newbtr{\nabla\accentset{\circ}{A}}+\newbtr{\nabla\left(\ul{\theta}^{-1}\accentset{\circ}{\ul{\chi}}\right)}\newbtr{\accentset{\circ}{A}}+\frac{C(m)}{\omega^4}\btr{\nabla\omega}+\btr{\nabla^2\tau}+2\btr{\tau}\btr{\nabla\tau}\\
			&\,+\btr{\nabla\left(O_{2,2}(r^{-4})+\nabla\omega^AO_{2,2}(r^{-3})_A+\nabla\omega^A\nabla\omega^BO_{2,2}(r^{-2})_{AB}\right)}.
		\end{align*}
		Observe that as $\mathcal{H}^2$ is constant, Proposition \ref{prop_proof_Abound_1} and Corollary \ref{kor_prelim_nablatau} imply that
		\[
			\btr{\nabla^2\tau}\le \frac{C}{\omega^2}\newbtr{\nabla\accentset{\circ}{A}}+\frac{C}{\omega^3}\mathcal{H}^2+\frac{C}{\omega^5}\left(1+\newbtr{\widehat{\nabla}\ln\omega}^3_{\hatgamma}\right).
		\]
		Using this bound on $\btr{\nabla^2\tau}$ also to control $\newbtr{\nabla (F\ast\accentset{\circ}{A})}$ as given in Proposition \ref{prop_appendix_NullSimonAsymSchwarzschild}, a straightforward application of Youngs' inequality yields
		\begin{align*}
			\left(\btr{\nabla f}\newbtr{\accentset{\circ}{A}}+2\newbtr{\nabla (F\ast\accentset{\circ}{A})}+\btr{\nabla\scal}\newbtr{\accentset{\circ}{A}}+\btr{\nabla G}\right)\newbtr{\nabla\accentset{\circ}{A}}\le c_2\mathcal{H}^2 \newbtr{\nabla\accentset{\circ}{A}}^2+C_2(\mathcal{H}^2)^{-1}\left(f_1+f_2\newbtr{\accentset{\circ}{A}}^2\right),
		\end{align*}
		where $f_1$, $f_2$ are as claimed by Lemma \ref{lem_prelim_backgroundfol}, Lemma \ref{lem_appendix_derivativesomega}, and Lemma \ref{lem_appendix_derivativesTR}, and where we bound the arising second derivative terms by \eqref{eq_prelim_c2estimate_A}.
	\end{proof}
	
	Lemma \ref{lem_proof_deltaA_deltanablaA} now allows us to obtain the desired bound on $\newbtr{\nabla\accentset{\circ}{A}}$:
	
	\begin{prop}\label{prop_proof_nablaA1}
		Let $\mathcal{N}$ be an asymptotically Schwarzschildean lightcone of mass $m>0$, $\Sigma$ an STCMC surface with $\mathcal{H}^2>0$. Let $\kappa_0>0$. There exists $\alpha_0<\frac{1}{2}$, $\varepsilon_0$, and a small number $\eta>0$, which only depends on $\kappa_0$, such that for all $0<\alpha<\alpha_0$, $0<\varepsilon<\varepsilon_0$, $\kappa\ge \kappa_0$
		\[	
		\newbtr{\nabla\accentset{\circ}{A}}\le (\mathcal{H}^2)^{\frac{3}{2}+\eta}
		\]
		if $\Sigma$ lies in $\mathcal{S}^1_{\alpha,\varepsilon,\kappa}[v_1,v_2,V_1,V_2]$ provided $\mathcal{H}^2\le \mathcal{H}^2_0(m,\alpha,\varepsilon,\kappa, v_1,v_2,V_1,V_2)$.
	\end{prop}
	
	\begin{proof}
		Define $f:=\newbtr{\nabla\accentset{\circ}{A}}^2+C_1\mathcal{H}^2\newbtr{\accentset{\circ}{A}}^2$. By Lemma \ref{lem_proof_deltaA_deltanablaA}, we have
		\[
			\Delta f\ge \frac{C_1}{2}\mathcal{H}^2\newbtr{\nabla\accentset{\circ}{A}}^2-2\mathcal{H}^2\btr{G}\newbtr{\accentset{\circ}{A}}-C_2(\mathcal{H}^2)^{-1}\left(f_1+f_2\newbtr{\accentset{\circ}{A}}^2\right).
		\]
		In particular
		\[
			\Delta f\ge \frac{C_1}{2}\mathcal{H}^2f-\frac{C_1}{2}(\mathcal{H}^2)^2\newbtr{\accentset{\circ}{A}}^2-2\mathcal{H}^2\btr{G}\newbtr{\accentset{\circ}{A}}-C_2(\mathcal{H}^2)^{-1}\left(f_1+f_2\newbtr{\accentset{\circ}{A}}^2\right).
		\]
		As $0\ge \Delta f(p)$ at a point $p\in\Sigma$ where $f$ takes its global maximum, we have
		\begin{align}\label{eq_proof_nablaA_improvement}
			\newbtr{\nabla\accentset{\circ}{A}}^2\le f\le f(p)\le C\max\limits_{\Sigma}\left(\mathcal{H}^2\newbtr{\accentset{\circ}{A}}^2+\btr{G}\newbtr{\accentset{\circ}{A}}+(\mathcal{H}^2)^{-2}\left(f_1+f_2\newbtr{\accentset{\circ}{A}}^2\right)\right).
		\end{align}
		Using Proposition \ref{prop_proof_Abound_1}, the expansion of $G$ as in Proposition \ref{prop_appendix_NullSimonAsymSchwarzschild}, and the estimates for $f_1$, $f_2$ as in Lemma \ref{lem_proof_deltaA_deltanablaA}, the claim follows for a $\eta>0$ only depending on $\kappa_0$ by combining Lemma \ref{lem_appendix_derivativesomega}, Lemma \ref{lem_appendix_derivativesTR} and the fact that $\Sigma$ lies in  $\mathcal{S}^1_{\alpha,\varepsilon,\kappa}[v_1,v_2,V_1,V_2]$ if $\alpha_0$ is chosen sufficiently small, where we use \eqref{eq_prelim_c2estimate_A} to control the second derivative term in $f_1$ in terms of $\mathcal{H}^2$ and first derivatives.
	\end{proof}
	
	Although the initial bounds on $\newbtr{\accentset{\circ}{A}}$ and $\newbtr{\nabla\accentset{\circ}{A}}$ are still far from the desired bounds \eqref{eq_proof_goal2}, we have precisely obtained the critical decay to argue as in the proof of the a priori estimates, \cite[Proposition 3.3]{kroenckewolff}, by using the quantitative estimates in Section \ref{subsec_prelim_quantativeestimates}. This will allow us to conclude that \eqref{eq_proof_goal1} and \eqref{eq_proof_goal2} are satisfied, and thus conclude the proof of Theorem \ref{thm_main1}:
	
	\begin{prop}\label{prop_proof_end}
		Let $\mathcal{N}$ be an asymptotically Schwarzschildean lightcone of mass $m>0$, $\Sigma$ an STCMC surface with $\mathcal{H}^2>0$. Assume there exists $\beta>1$ such that
		\[
			\newbtr{\accentset{\circ}{A}}\le (\mathcal{H}^2)^{\beta},\text{ and }\newbtr{\nabla\accentset{\circ}{A}}\le (\mathcal{H}^2)^{\frac{3}{2}\beta}.
		\]
		Let $\kappa_0\ge 0$. There exists $\alpha_0=\alpha_0(\beta)<\frac{1}{2}$, $\varepsilon_0$ such that for all $0<\alpha<\alpha_0$, $0<\varepsilon\le \varepsilon_0$, $\kappa\ge \kappa_0$ the following holds: If $\Sigma$ lies in $\mathcal{S}^1_{\alpha,\varepsilon,\kappa}[v_1,v_2,V_1,V_2]$, then $\Sigma$ satisfies \eqref{eq_proof_goal1} and \eqref{eq_proof_goal2} for suitable (fixed) constants $B_1$, $B_2$, $B_3$ provided $\mathcal{H}^2\le \mathcal{H}^2_0(m,\alpha,\varepsilon,\kappa,\beta,v_1,v_2,V_1,V_2)$.
	\end{prop}

	\begin{proof}
		In the following, we will frequently make use of the estimates in Section \ref{subsec_prelim_asymSchw} and estimate the error terms using Lemma \ref{lem_appendix_derivativesomega} and Lemma \ref{lem_appendix_derivativesTR} without calling explicit attention to it. 
		
		First, by Lemma \ref{lem_proof_H2rho} we find that
		\[
		C^{-1}\rho^{-2}\le \mathcal{H}^2\le C\rho^{-2}
		\]
		if $\alpha_0$ and $\mathcal{H}^2\le \mathcal{H}_0^2$ sufficiently small, where $C$ is some universal constant. In particular, we can express the above assumptions in terms of the area radius $\rho$, for $\rho\ge \rho_0$ sufficiently large.
		
		In fact, we will first consider the conformally round metric $\widetilde{\gamma}=\omega^2\hatgamma$ with area radius $\widetilde{\rho}$ and vector $\vec{a}$ associated to the Minkowski $4$-vector $\textbf{Z}$ as defined in Section \ref{subsec_prelim_quantativeestimates}. Note that by the asymptotic expansion of the volume form $\d\mu$, see \cite[Equation (A.1)]{kroenckewolff}, we find $2\widetilde{\rho}\ge \rho\ge \frac{1}{2}\widetilde{\rho}$. In particular, we may first formulate everything with respect to $\widetilde{\gamma}$ and $\widetilde{\rho}$. In the following, $\widetilde{\rho}\ge \widetilde{\rho}_0$ will always be chosen sufficiently large (possibly depending on all other parameters).
		
		For the convenience of the reader, we split the proof into several steps:\newline
		\textbf{Step 1: Comparison to a round sphere of radius $\widetilde{\rho}$ and associated vector $\vec{a}$}\newline
		
		As $\mathcal{H}^2$ is constant, taking a trace of the Codazzi Equation, Proposition \ref{prop_prelim_codazziA}, yields 
		\[
			\dive\accentset{\circ}{A}_j=\ul{\theta}\gamma^{ij}\overline{\Riem}_{ijkL}+\frac{1}{2}\mathcal{H}^2\tau_j-\accentset{\circ}{A}(\vec{\tau},\partial_j),
		\]
		and we further observe that the curvature identities, Proposition \ref{prop_appendix_curvatureidentities}, imply that
		\[
			\btr{\nabla\left(\frac{4m}{\omega^3}\right)-\ul{\theta}\gamma^{ij}\overline{\Riem}_{ijkL}}\le \frac{C}{\omega^5}\left(1+\newbtr{\widehat{\nabla}\ln\omega}_{\hatgamma}^3\right).
		\]
		Combining the above with Lemma \ref{lem_comparrisonRtildemathcalH}, we obtain
		\begin{align}\label{eq_proof_proof_1}
			\newbtr{\widetilde{\nabla}\widetilde{R}}_{\widetilde{\gamma}}\le C\left[2\newbtr{\nabla\accentset{\circ}{A}}+\frac{1}{\omega^5}\left(1+\newbtr{\widehat{\nabla}\ln\omega}_{\hatgamma}^3\right)+\frac{\mathcal{H}^2}{\omega^3}\left(1+\newbtr{\widehat{\nabla}\ln\omega}_{\hatgamma}\right)\right],
		\end{align}
		provided $\widetilde{\rho}\ge \widetilde{\rho}_0$ sufficiently large, where we used that $\gamma$, $\widetilde{\gamma}$ are uniformly equivalent, and Equations \eqref{eq_prelim_c2estimate_A} and \eqref{eq_prelim_c3estimate_nablaA} to bound the higher derivatives. By assumption, \eqref{eq_proof_proof_1} implies that
		\[
			\newbtr{\widetilde{\nabla}\widetilde{\scal}}_{\widetilde{\gamma}}\le C\widetilde{\rho}^{-3\beta}
		\]
		provided $\widetilde{\rho}\ge \widetilde{\rho}_0$ if $\alpha<\alpha_0(\beta)$ sufficiently small.
		As $\fint\widetilde{\scal}=2\widetilde{\rho}^{-2}$ by the Gauss--Bonnet Theorem, there exists a point $\vec{x}\in\Sbb^2$ such that $\widetilde{\scal}(\vec{x})=\fint\widetilde{\scal}=2\widetilde{\rho}^{-2}$. Let $\vec{y}$ be any point on $(\Sbb^2,\widetilde{\gamma})$ that can be connected to $\vec{x}$ by a unit speed geodesic $\mathfrak{s}$ of length $L\le 10\widetilde{\rho}$. Then
		\[	
		\btr{\widetilde{\scal}(\vec{y})-\widetilde{\scal}(\vec{x})}\le \int_{\mathfrak{s}} \spann{\widetilde{\nabla}\widetilde{\scal},\dot{\mathfrak{s}}}\d s\le C\widetilde{\rho}^{1-3\beta}.
		\]
		In particular, $\widetilde{\scal}(\vec{y})\ge \widetilde{\rho}^{-2}$ if $\widetilde{\rho}\ge \widetilde{\rho}_0$ sufficiently large. By the theorem of Bonnet-Myers, we can now conclude that such geodesics in fact cover all of $(\Sbb^2,\widetilde{\gamma})$, i.e., $\operatorname{diam}(\Sbb^2,\widetilde{\gamma})\le 10\widetilde{\rho}$. We conclude that
		\[
		\newnorm{\widetilde{\scal}-\fint\widetilde{\scal}}_{C^1(\,\widetilde{\gamma}\,)}\le C\widetilde{\rho} \max\limits_{\Sbb^2}\newbtr{\widetilde{\nabla}\widetilde{\scal}}_{\widetilde{\gamma}},
		\]
		where we consider a weighted $C^1$ norm as in \cite[Section 2.2]{kroenckewolff}. By the definition of the weighted norm, and the fact that $\omega_{\max}\ge \widetilde{\rho}$ and  $\btr{\widehat{\nabla}\widetilde{\scal}}_{\widehat{\gamma}}=\omega\btr{\widetilde{\nabla}\widetilde{\scal}}_{\widetilde{\gamma}}$, we find
		\[
		\newnorm{\scal-\fint\scal}_{C^1(\widehat{\gamma})}\le C\omega_{\max} \max\limits_{\Sbb^2}\newbtr{\widetilde{\nabla}\widetilde{\scal}}_{\widetilde{\gamma}}.
		\]
		Using the diameter bound, we find
		\begin{align}\label{eq_proof_lemApriori_1}
			\omega\le \widetilde{\rho}(1+10\btr{\widetilde{\nabla}\omega}_{\widetilde{\gamma}})=\rho(1+10\btr{\widehat{\nabla}\ln\omega}_{\hatgamma})\le C(v_2)\widetilde{\rho}^{1+2a}
		\end{align}
		if $\rho\ge\rho_0$ sufficiently large. As $\newbtr{\widetilde{\nabla}\widetilde{\scal}}_{\widetilde{\gamma}}\le C\widetilde{\rho}^{-3\beta}$, Corollary \ref{kor_acontrol_rough} and Equation \eqref{eq_proof_lemApriori_1} further imply that $\btr{\vec{a}}\le C\rho^{2\alpha}$. Hence, if $\alpha_0(\beta)$ is chosen sufficiently small and if $\widetilde{\rho}\ge\widetilde{\rho}$ is sufficiently large, then Corollary \ref{kor_quantitativecontrol} implies that
		\begin{align}\label{eq_proof_proof_2}
			\newnorm{\omega-b_{\widetilde{\rho},\vec{a}}}_{C^2(\,\widehat{\gamma}\,)}\le C\left(1+\newbtr{\vec{a}}^2\right)^2\widetilde{\rho}^3\omega_{\max}\max\limits_{\Sbb^2}\newbtr{\widetilde{\nabla}\widetilde{\scal}}.
		\end{align}
		Hence
		\[
		\norm{\frac{\omega}{\widetilde{\rho}}-b_{\vec{a}}}_{C^2(\,\widehat{\gamma}\,)}\le C\rho^{-\delta}
		\]
		for some uniform $\delta>0$ if $\alpha_0(\beta)$ is chosen sufficiently small. \newline\newline
		\textbf{Step 2: Improved estimate on the associated vector $\vec{a}$}\newline
		Assume $\vec{a}\not=0$, as the desired estimate below will be trivially satisfied otherwise. Consider the function $f=f_{\vec{a}}$ that is induced by a family of M\"obius transformations in direction $\frac{\vec{a}}{\btr{\vec{a}}}$ as in \cite[Section 5]{wolff_stability}. We briefly recall that
		\[
		\int_{\Sbb^2} f\d\widetilde{\mu}=\int_{\Sbb^2}f\widetilde{\scal}\d\widetilde{\mu}=0,
		\]
		where the first identity is due to the fact that the area is preserved under a M\"obius transformation and the second one is a consequence of the Kazdan--Warner identity, see \cite[Proposition 5.3]{wolff_stability}. In particular, as $\d\widetilde{\mu}=\omega^2\d\widehat{\mu}$ and $\mathcal{H}^2$ is constant, we conclude that 
		\[
		8m\widetilde{\rho}^{-1}\int_{\Sbb^2}\frac{\widetilde{\rho}}{\omega}f\d\widehat{\mu}=\int_{\Sbb^2}f\left(\mathcal{H}^2-2\widetilde{\scal}+\frac{8m}{\omega^3}\right)\d\widetilde{\mu}.
		\]
		We assume without loss of generality that $\vec{a}=(0,0,\btr{\vec{a}})$ (as we can always achieve this by a suitable rotation acting as an isometry on the round metric). In particular,
		\[
		b_{\vec{a}}=\frac{1}{\sqrt{1+\btr{\vec{a}}^2}-\btr{\vec{a}}\cos\theta},
		\]
		and 
		\[
		f=2\hatgamma\left(\widehat{\nabla}\ln\omega,\sin\theta\partial_\theta\right)+2\cos\theta=2\cos\theta-2\widehat{\gamma}(\widehat{\nabla}\ln\omega,\widehat{\nabla}\cos\theta).
		\]
		Using partial integration and the fact that $\cos\theta$ is a first eigenfunction, we thus also find
		\begin{align*}
			\int_{\Sbb^2}\frac{\widetilde{\rho}}{\omega}f\d\widehat{\mu}=6\int_{\Sbb^2}\frac{\widetilde{\rho}}{\omega}\cos\theta\d\widehat{\mu}=6\int_{\Sbb^2}b_{\vec{a}} ^{-1}\cos\theta\d\widehat{\mu}+6\int_{\Sbb^2}\left(\frac{\widetilde{\rho}}{\omega}-b_{\vec{a}}^{-1}\right)\cos\theta\d\widehat{\mu}
		\end{align*}
		Explicit integration yields
		\[
		6\int_{\Sbb^2}b_{\vec{a}}^{-1}\cos\theta\d\widehat{\mu}=-8\pi\btr{\vec{a}},
		\]
		so combining the above two identities and using Lemma \ref{lem_comparrisonRtildemathcalH} together with Equation \eqref{eq_prelim_c2estimate_A}, we obtain
		\begin{align}
			\begin{split}\label{eq_proof_proof_3}
			\btr{\vec{a}}\le&\, C\sqrt{1+\btr{\vec{a}}^2}\max\left(\frac{\widetilde{\rho}}{\omega}\right)\norm{\frac{\omega}{\widetilde{\rho}}-b_{\vec{a}}}_{C^0(\widehat{\gamma})}\\
			&\,+\frac{C\widetilde{\rho}}{m}\left(1+\max\newbtr{\widehat{\nabla}\ln\omega}_{\hatgamma}\right)\left(\mathcal{H}^2+\frac{C}{\omega_{\min}^2}\left(1+\max\btr{\widehat{\nabla}\ln\omega}_{\widehat{\gamma}}^2\right)\right),
			\end{split}
		\end{align}
		where we also used that $b_{\vec{a}}^{-1}\le 2\sqrt{1+\btr{\vec{a}}^2}$ and $\btr{\widehat{\nabla}b_{\vec{a}}}\le C(1+\btr{\vec{a}}^2)$. By assumption, and the already established upper bounds on $\omega$ and $\vec{a}$, we thus find
		\[
		\btr{\vec{a}}\le C\widetilde{\rho}^{-\frac{\delta}{2}}.
		\]
		provided $\widetilde{\rho}\ge \widetilde{\rho}$ sufficiently large if $\alpha<\alpha_0(\beta)$ is chosen sufficiently small. Since 
		\[
			\btr{1-b_{\vec{a}}}\le  C\btr{\vec{a}}\sqrt{1+\btr{\vec{a}}^2},\text{ and } \btr{\widehat{\nabla}b_{\vec{a}}}\le C\btr{\vec{a}}\sqrt{1+\btr{\vec{a}}^2},
		\]
		this improved bound on $\vec{a}$ together with the already established comparison between $\omega$ and $b_{\widetilde{\rho},\vec{a}}$ in Step 1 now yields that
		\begin{align*}
			\btr{\omega-\rho}&\le \widetilde{\rho}^{1-\frac{1}{3}\delta},\\
			\btr{\widehat{\nabla}\frac{\omega}{\widetilde{\rho}}}&\le C\widetilde{\rho}^{-\frac{1}{3}\delta}
		\end{align*}
		provided $\widetilde{\rho}\ge \widetilde{\rho}_0$ sufficiently large if $\alpha<\alpha_0(\beta)$ is chosen sufficiently small.
		In particular, $\frac{1}{2}\widetilde{\rho}\le \omega\le 2\widetilde{\rho}$, which further implies that
		\[
		\btr{\widehat{\nabla}\ln\omega}\le \rho^{-\frac{1}{3}\delta}.
		\]
		\noindent\textbf{Step 3: Revisiting previous estimates}
		We are now in the position to feed the newly obtained $C^2$ estimates on $\omega$ into the previous estimates. This will result in drastically improved estimates, yielding the claim. Using $\frac{1}{2}\rho\le \omega\le 2\rho$ and $\max(\btr{\vec{a}}, \newbtr{\widehat{\nabla}\ln\omega}_{\hatgamma})\le 1$ for $\rho\ge \rho_0$ sufficiently large, Equations \eqref{eq_proof_proof_1}, \eqref{eq_proof_proof_2}, and \eqref{eq_proof_proof_3} now yield
		\begin{align*}
			\norm{\frac{\omega}{\rho}-b_{\vec{a}}}_{C^{2}(\widehat{\gamma})}&\le C\widetilde{\rho}^3\max\newbtr{\nabla\accentset{\circ}{A}}+\frac{C}{\widetilde{\rho}^2},\\
			\btr{\vec{a}}&\le C\norm{\frac{\omega}{\rho}-b_{\vec{a}}}_{C^1(\widehat{\gamma})}+\frac{C}{\widetilde{\rho}}\le  C\widetilde{\rho}^3\max\newbtr{\nabla\accentset{\circ}{A}}+\frac{C}{\widetilde{\rho}},
		\end{align*}
		which in particular yield that
		\[
		\btr{\widehat{\nabla}^i\ln\omega}\le C\widetilde{\rho}^{3}\max \newbtr{\nabla\accentset{\circ}{A}}+\frac{C}{\widetilde{\rho}}
		\]
		for $i=1,2$. By Remark \ref{bem_proof_boundA1}, we can now conclude that
		\[
		\newbtr{\accentset{\circ}{A}}\le\frac{C}{\widetilde{\rho}^3}\left(\frac{C}{\widetilde{\rho}}+\widetilde{\rho}^3\max\newbtr{\nabla\accentset{\circ}{A}}\right)^2+ \frac{C}{\widetilde{\rho}^4}\le C\widetilde{\rho}^3\max\newbtr{\nabla\accentset{\circ}{A}}^2+\frac{C}{\widetilde{\rho}^4},
		\]
		where we used Youngs' inequality. Finally, in a similar fashion, Equation \eqref{eq_proof_nablaA_improvement} yields 
		\[
			\newbtr{\nabla\accentset{\circ}{A}}^2\le \frac{C}{\widetilde{\rho}^{10}}+\frac{C}{\widetilde{\rho}^2}\max\newbtr{\nabla\accentset{\circ}{A}}^2\left(C\widetilde{\rho}^{-1}+\widetilde{\rho}^6\max\newbtr{\nabla\accentset{\circ}{A}}^2+\widetilde{\rho}^{12}\max\newbtr{\nabla\accentset{\circ}{A}}^4\right).
		\]
		By assumption $\widetilde{\rho}^3\newbtr{\nabla\accentset{\circ}{A}}\le 1$ provided $\widetilde{\rho}\ge \widetilde{\rho}_0$ sufficiently large. Hence, the above bounds imply that
		\[
			\newbtr{\accentset{\circ}{A}}\le \frac{\widetilde{B}_2}{\widetilde{\rho}^4},\text{ and }\newbtr{\nabla\accentset{\circ}{A}}\le \frac{\widetilde{B}_3}{\widetilde{\rho}^5}
		\]
		for suitable (fixed) constants $\widetilde{B}_2$, $\widetilde{B}_3$ for $\widetilde{\rho}\ge \widetilde{\rho}_0$ sufficiently large. Using the $C^2$-bounds derived above, \eqref{eq_prelim_c3estimate_nablaA} and Lemma \ref{lem_comparrisonRtildemathcalH}, we further find
		\[
			\max\left(\btr{\omega-\widetilde{\rho}},\, \norm{\frac{\omega}{\widetilde{\rho}}}_{C^3(\widehat{\gamma})},\,\widetilde{\rho}^2\mathcal{H}^2\right)\le 7
		\]
		provided $\widetilde{\rho}\ge \widetilde{\rho}_0$ sufficiently large. Finally, we note that $\btr{\rho-\widetilde{\rho}}\le C\widetilde{\rho}^{-1}$ by \cite[Equation (A.1)]{kroenckewolff}, so we may replace $\widetilde{\rho}$ by $\rho$ in all estimates to obtain \eqref{eq_proof_goal1} and \eqref{eq_proof_goal2}. This concludes the proof.
	\end{proof}
	
	\appendix
	
	\section{Difference Tensor revisited}\label{appendix_differencetensor}
	
	We revisit our analysis of the difference tensor in \cite[Appendix: Difference tensor estimates]{kroenckewolff} and provide some more general estimates. In particular, we will not assume that the derivatives of $\omega$ are bounded up to some order $k$ ($k=3$ in \cite{kroenckewolff}), but only assume that $\omega \ge r_0$ for some sufficiently large $r_0$. Up to not assuming an a priori bound on the derivatives of $\omega$, all estimates are derived in complete analogy to \cite{kroenckewolff}, so we will skip most of the proofs for brevity. As above, we will always consider the graph of a spacelike cross section $\omega\colon\Sbb^2\to(r_0,\infty)$ with respect to a background foliation as in Definition \ref{defi_asymclassS}, and (the pullback of) the induced the metric $\gamma_\omega$ on $\Sbb^2$ defined via $(\gamma_\omega)_p=(\gamma_{\omega(p)})_p$. Moreover, we consider the conformally round metric $\widetilde{\gamma}_\omega:=\omega^2\hatgamma$, where $\widehat{\gamma}$ denotes the standard round metric. One can check that
	\begin{align}\label{eq_appx_inverse}
		\gamma_\omega^{KL}=\tildegamma_\omega^{KL}+O_{3,3}(r^{-4})^{KL}.
	\end{align}
	In the following, $\nabla$, $\widetilde{\nabla}$, and $\widehat{\nabla}$ will always denote the gradient of a function or the tensor derivative with respect to $\gamma$, $\widetilde{\gamma}$, and $\widehat{\gamma}$, respectively.
	Let $\Gamma_{IJ}^K$, $\widetilde{\Gamma}_{IJ}^K$, $\widehat{\Gamma}_{IJ}^K$ denote Christoffel symbols with respect to $\gamma_\omega$, $\tildegamma_\omega$, and $\hatgamma$, respectively\footnote{Here, capital letters denote local coordinates on $\Sbb^2$ to emphasize that we work with the pullback of all objects as smooth (families of) tensor fields on $\Sbb^2$.}. Then the difference tensors
	\begin{align*}
		\widehat{Q}_{IJ}^K&:=\Gamma_{IJ}^K-\widehat{\Gamma}_{IJ}^K,\\
		\widetilde{Q}_{IJ}^K&:=\Gamma_{IJ}^K-\widetilde{\Gamma}_{IJ}^K,
	\end{align*}
	satisfy the following identities, see \cite[Lemma A.1]{kroenckewolff}:
	\begin{lem}\label{lem_appx_differencetensors}
		\begin{align*}
			\widehat{Q}_{IJ}^K=
			&\,\gamma_{\omega}^{KL}\left(w\d\omega_I\hatgamma_{JL}+\omega\d\omega_J\hatgamma_{IL}-\omega\d\omega_L\hatgamma_{IJ}\right)\\
			&\,+\frac{1}{2}\gamma_{\omega}^{KL}\left(\d\omega_IO_{2,3}(r^{-1})_{JL}+\d\omega_JO_{2,3}(r^{-1})_{IL}-\d\omega_LO_{2,3}(r^{-1})_{IJ}\right)\\
			&\,+\gamma_{\omega}^{KL}\left(O_{3,2}(1)_{IJL}+O_{3,2}(1)_{JIL}-O_{3,2}(1)_{LIJ}\right),\\
			\,\\
			\widetilde{Q}_{IJ}^K=
			&\,\frac{1}{2}\gamma_\omega^{KL}\left(d\omega_IO_{2,3}(r^{-1})_{JL}+d\omega_JO_{2,3}(r^{-1})_{IL}-d\omega_LO_{2,3}(r^{-1})_{IJ}\right)\\
			&\,+\gamma^{KL}_\omega\left(O_{3,2}(1)_{IJL}+O_{3,2}(1)_{JIL}-O_{3,2}(1)_{LIJ}\right)\\
			&\,+\frac{1}{\omega}\gamma_\omega^{KL}\left(\widehat{\nabla}\omega^M\hatgamma_{IJ}-\d\omega_I\delta_J^M-\d\omega_J\delta_I^M\right)O_{3,3}(1)_{LM},
		\end{align*}
	\end{lem}
	This yields the following estimates on the difference tensors:
	
	\begin{lem}\label{lem_appendix_differencetensor2}
		For $\omega\ge r_0$ sufficiently large, there exists a uniform constant $C$ such that
		\begin{align*}
			\newbtr{\widehat{Q}}_{\widehat{\gamma}}&\le C\left(\newbtr{\widehat{\nabla}\ln\omega}_{\hatgamma}+\omega^{-2}\right),\\
			\newbtr{\widehat{\nabla}\widehat{Q}}_{\hatgamma}&\le C\left(\newbtr{\widehat{\nabla}^2\ln\omega}_{\hatgamma}+\omega^{-2}\left(1+\newbtr{\widehat{\nabla}\ln\omega}^2_{\hatgamma}\right)\right),\\
			\newbtr{\widehat{\nabla}^3\widehat{Q}}_{\hatgamma}&\le  C\left(\newbtr{\widehat{\nabla}^3\ln\omega}_{\hatgamma}+\omega^{-2}\left(1+\newbtr{\widehat{\nabla}\ln\omega}_{\hatgamma}\right)\left(1+\newbtr{\widehat{\nabla}\ln\omega}^2_{\hatgamma}+\newbtr{\widehat{\nabla}^2\ln\omega}_{\hatgamma}\right)\right).
		\end{align*}
	\end{lem}
	\begin{bem}
		Note that the estimates here and in the following are understood as pointwise bounds on the tensor norm, i.e,
		\[
			\newbtr{\widehat{Q}}_{\widehat{\gamma}}(\vec{x})\le C\left(\newbtr{\widehat{\nabla}\ln\omega}_{\hatgamma}(\vec{x})+\omega^{-2}(\vec{x})\right)
		\]
		for all $\vec{x}\in\Sbb^2$.
	\end{bem}
	\begin{proof}
		As $\gamma_{\omega}^{-1}=\widetilde{\gamma}_\omega^{-1}+O_{3,3}(r^{-4})^{IJ}$, we in particular have $\gamma_\omega^{-1}=O_{3,3}(r^{-2})^{IJ}$. Using both, it is straightforward to simplify the identity for $\widehat{Q}$ in Lemma \ref{lem_appx_differencetensors} to
		\begin{align*}
			\widehat{Q}_{IJ}^K=&\,\left(\delta_J^K+O_{2,3}(r^{-2})_J^K\right)\d\ln\omega_I+\left(\delta_I^K+O_{2,3}(r^{-2})_I^K\right)\d\ln\omega_J\\
			&\,-\left(\widehat{\gamma}^{KL}\widehat{\gamma}_{IJ}+O_{2,3}(r^{-2})_{IJ}^{KL}\right)\d\ln\omega_L+O_{3,2}(r^{-2})_{IJ}^K.
		\end{align*}
		This readily implies the estimate for $\widehat{Q}$. To estimate the tensor derivative of $\widehat{Q}$ we observe that for any smooth family of tensor fields $(T_r)$ evaluated along $r=\omega$, we have for $T=T_r\vert_{r=\omega}$
		\[
			\widehat{\nabla}_KT=\omega\d\ln\omega_K\partial_rT_r\vert_{r=\omega}+\widehat{\nabla}_KT_r\vert_{r=\omega}.
		\]
		Hence, if $T=O_{k,l}(r^\alpha)$ for $k,l\ge 1$, we find
		\[
			\newbtr{\widehat{\nabla} T}_{\hatgamma}\le C\omega^{\alpha}\left(1+\newbtr{\widehat{\nabla}\ln\omega}_{\hatgamma}\right).
		\]
		If $k,l\ge 2$, we similarly derive that
		\[
			\newbtr{\widehat{\nabla}^2T}_{\hatgamma}\le C\omega^{\alpha}\left(1+\newbtr{\widehat{\nabla}\ln\omega}_{\hatgamma}^2+\newbtr{\widehat{\nabla}^2\ln\omega}_{\hatgamma}\right).
		\]
		Combining this with the above identity for $\widehat{Q}$ implies the estimate for $\widehat{\nabla}\widehat{Q}$ and $\widehat{\nabla}^2\widehat{Q}$.
	\end{proof}
	
	From this, it is straightforward to establish the following comparison between the tensor derivatives of $\omega$:
	\begin{lem}\label{lem_appendix_derivativesomega}
		For $\omega\ge r_0$ sufficiently large, there exist uniform positive constants $C_1$, $C_2$, $C_3$ such that
		\[
			C_1^{-1}\newbtr{\widehat{\nabla}\ln\omega}_{\hatgamma}\le \newbtr{\nabla\omega}_\gamma\le C_1\newbtr{\widehat{\nabla}\ln\omega}_{\hatgamma},
		\]
		and
		\begin{align*}
			\newbtr{\nabla^2\omega}_\gamma\le&\, C_2\omega^{-1}\left(\newbtr{\widehat{\nabla}^2\ln\omega}_{\hatgamma}+\newbtr{\widehat{\nabla}\ln\omega}_{\hatgamma}\left(\newbtr{\widehat{\nabla}\ln\omega}_{\hatgamma}+\omega^{-2}\right)\right),\\
			\newbtr{\widehat{\nabla}^2\ln\omega}_{\hatgamma}\le&\, C_2\left(\omega \newbtr{\nabla^2\omega}_\gamma+\newbtr{\widehat{\nabla}\ln\omega}_{\hatgamma}\left(\newbtr{\widehat{\nabla}\ln\omega}_{\hatgamma}+\omega^{-2}\right)\right),\\
			\newbtr{\nabla^3\omega}_\gamma\le&\, C_3\omega^{-2}\left(\newbtr{\widehat{\nabla}^3\ln\omega}_{\hatgamma}+\newbtr{\widehat{\nabla}^2\ln\omega}_{\hatgamma}\left(\newbtr{\widehat{\nabla}\ln\omega}_{\hatgamma}+\omega^{-2}\right)\right)\\
			&+C_3\omega^{-2}\left(\newbtr{\widehat{\nabla}\ln\omega}_{\hatgamma}\left(\newbtr{\widehat{\nabla}\ln\omega}_{\hatgamma}+\omega^{-2}\right)^2+\omega^{-{2}}\newbtr{\widehat{\nabla}\ln\omega}_{\hatgamma}\right),\\
			\newbtr{\widehat{\nabla}^3\ln\omega}_{\hatgamma}\le&\, C_3\left(\omega^2\newbtr{{\nabla}^3\omega}_{\gamma}+\newbtr{\widehat{\nabla}^2\ln\omega}_{\hatgamma}\left(\newbtr{\widehat{\nabla}\ln\omega}_{\hatgamma}+\omega^{-2}\right)\right)\\
			&+C_3\left(\newbtr{\widehat{\nabla}\ln\omega}_{\hatgamma}\left(\newbtr{\widehat{\nabla}\ln\omega}_{\hatgamma}+\omega^{-2}\right)^2+\omega^{-{2}}\newbtr{\widehat{\nabla}\ln\omega}_{\hatgamma}\right).\\
		\end{align*}
	\end{lem}
	Similarly, we find the following estimates for a family of tensor field $(T_r)$ evaluated along $r=\omega$:
	\begin{lem}\label{lem_appendix_derivativesTR}
		Let $T_r=O_{k,l}(r^{\alpha})$ be a family of $(m,n)$-tensors. For $\omega\ge r_0$ sufficiently small, there exists a positive $C$ (independent of $\omega$) such that
		\begin{align*}
			\newbtr{T_r}_\gamma&\le C\omega^{\alpha+m-n},\\
			\newbtr{\nabla T_r}_\gamma&\le C\omega^{\alpha+m-n-1}\left(1+\newbtr{\widehat{\nabla}\ln\omega}_{\hatgamma}\right),\\
			\newbtr{\nabla^2 T_r}_\gamma&\le C\omega^{\alpha+m-n-2}\left(1+\newbtr{\widehat{\nabla}\ln\omega}^2_{\hatgamma}+\newbtr{\widehat{\nabla}^2\ln\omega}_{\hatgamma}\right),\\
			\newbtr{\nabla^3T_r}_\gamma&\le C\omega^{\alpha+m-n-3}\left(\left(1+\newbtr{\widehat{\nabla}\ln\omega}_{\hatgamma}\right)\left(1+\newbtr{\widehat{\nabla}\ln\omega}^2_{\hatgamma}+\newbtr{\widehat{\nabla}^2\ln\omega}_{\hatgamma}\right)+\newbtr{\widehat{\nabla}^3\ln\omega}_{\hatgamma}\right)
		\end{align*}
		assuming that $k,l$ are sufficiently large to control this number of derivatives.
	\end{lem}
	
	\section{Curvature estimates revisited}\label{appendix_curvature}

	We restate the curvature identities in \cite[Appendix: Curvature estimates]{kroenckewolff} in a more general form. Using the curvature components in the exact Schwarzschild spacetime, see e.g. \cite[Lemma B.1 and B.2]{kroenckewolff}, we find the following in close analogy to \cite[Lemma B.3]{kroenckewolff}:
	
	\begin{prop}\label{prop_appendix_curvatureidentities}
		Let $\mathcal{N}$ be an asymptotically Schwarzschildean lightcone. If $\Sigma=\Sigma_\omega$ is a spacelike cross section of $\mathcal{N}$ with $\omega\ge r_0$ sufficiently large, then
		\begin{align*}
			\overline{\Riem}_{ikjl}=&\,2m\omega\left(\hatgamma_{IJ}\hatgamma_{KL}-\hatgamma_{IL}\hatgamma_{KJ}\right)+O_{2,2}(1)_{IKJL}\\
			&\,+\nabla\omega^AO_{2,2}(r)_{IKJLA}+\nabla\omega^A\nabla\omega^BO_{2,2}(r^2)_{IKJLAB},\\
			\overline{\Riem}_{i\ul{L}jL}=&\,-\frac{2m}{\omega}\hatgamma_{IJ}+O_{2,2}(r^{-2})_{IJ}+\nabla\omega^AO_{2,2}(r^{-1})_{IJA}+\nabla\omega^A\nabla\omega^BO_{2,2}(1)_{IJAB},\\
			\overline{\Riem}_{ijkL}=&\,-\frac{4m}{\omega^3} \nabla\omega^a\left(\gamma_{ik}\gamma_{ja}-\gamma_{ia}\gamma_{kj}\right)-\frac{2m}{\omega^3}\d\omega_j\gamma_{ik}+\frac{2m}{\omega^3}\d\omega_i\gamma_{jk}\\
			&\,+\nabla\omega^AO_{2,2}(1)_{IKJA}+\nabla\omega^A\nabla\omega^BO_{2,2}(r)_{IKJAB}+\nabla\omega^A\nabla\omega^B\nabla\omega^CO_{2,2}(r^2)_{IKJABC},\\
			\overline{\Riem}_{iLjL}=&-\frac{24}{\omega^3}\accentset{\circ}{\left(\d\omega\otimes\d\omega\right)}+O_{2,2}(r^{-2})_{IJ}+\nabla\omega^AO_{2,2}(r^{-1})_{IJA}+\nabla\omega^A\nabla\omega^BO_{2,2}(1)_{IJAB}\\
			&\,+\nabla\omega^A\nabla\omega^B\nabla\omega^CO_{2,2}(r)_{IJABC}+\nabla\omega^A\nabla\omega^B\nabla\omega^C\nabla\omega^DO_{2,2}(r^2)_{IJABCD},\\
			\overline{\Riem}_{ij\ul{L}L}=&\,O_{2,2}(r^{-2})_{IJ}+\nabla\omega^AO_{2,2}(r^{-1})_{IJA}+\nabla\omega^A\nabla\omega^BO_{2,2}(1)_{IJAB}\\
			\overline{\Riem}(\ul{L},L,L,\ul{L})=&\,\frac{8m}{\omega^3}+O_{2,2}(r^{-4})+\nabla\omega^AO_{2,2}(r^{-3})_A+\nabla\omega^A\nabla\omega^BO_{2,2}(r^{-2})_{AB},\\
			\overline{\Ric}(\ul{L},L)=&\,O_{2,2}(r^{-4})+\nabla\omega^AO_{2,2}(r^{-3})_A+\nabla\omega^A\nabla\omega^BO_{2,2}(r^{-2})_{AB},\\
			\overline{\Ric}(\ul{L},\ul{L})=&\,O_{2,2}(r^{-4}),\\
			\overline{\scal}=&\,O_{2,2}(r^{-4})+\nabla\omega^AO_{2,2}(r^{-3})_A+\nabla\omega^A\nabla\omega^BO_{2,2}(r^{-2})_{AB}
		\end{align*}
		along $\Sigma$, where $\nabla\omega^\mathcal{A}O_{k,l}(r^\alpha)_{\mathcal{I}}$ denotes contractions of $\nabla\omega$ and $(0,\btr{\mathcal{I}})$-tensors in $O_{k,l}(r^{\alpha})$ for capital multiindices $\mathcal{A}$, $\mathcal{I}$ with $\mathcal{A}\subseteq \mathcal{I}$.
	\end{prop}
	\begin{proof}
		Using the decomposition \eqref{eq_prelim_nullgeometry_frame} of the frame $\{\partial_i,\ul{L},L\}$ of $TM$ along $\Sigma$ into the frame $\{\partial_I,\ul{L},L_r\}$ of $TM$ along $\mathcal{N}$ with respect to the background foliation, the identities follow from a lengthy, but direct calculation using Definition \ref{defi_asymclassS}.
	\end{proof}
	
	Under the assumption that the derivatives of $\omega$ are controlled up to third order, Proposition \ref{prop_appendix_curvatureidentities} directly recovers the corresponding identities in \cite[Lemma B.3]{kroenckewolff} using Lemma \ref{lem_appendix_derivativesomega} and Lemma \ref{lem_appendix_derivativesTR}.
	
	\section{Contracted Null Simons' identity}\label{appendix_nullsimon}
	
	We prove the contracted Null Simons' identity for an STCMC surface as stated in Remark \ref{bem_prelim_contractedNullSimons}:
	
	\begin{prop}[Contracted Null Simons' identity]\label{prop_appendix_contractedNullSimons}
		Let $(\Sigma,\gamma)$ be a spacelike cross section of a null hypersurface $\mathcal{N}$ in an ambient $4$-dimensional spacetime $(\overline{M},\overline{g})$, $\ul{\theta}>0$ along $\Sigma$. If $\Sigma$ is an STCMC surface, then
		\begin{align*}
			\Delta\accentset{\circ}{A}_{ij}
			=&\,-2\tau^k\nabla_k\accentset{\circ}{A}_{ij}\\
			&\,+\left(\frac{1}{2}\mathcal{H}^2+\spann{\ul{\theta}^{-1}\accentset{\circ}{\ul{\chi}},\accentset{\circ}{A}}-\left(\overline{R}-3\overline{\Ric}(\ul{L},\ul{L})+\overline{\Riem}(\ul{L},L,L,\ul{L})\right)-\dive\tau-\btr{\tau}^2\right)\accentset{\circ}{A}_{ij}\\
			&\,+\left(\left(\overline{\Ric}_{jm}-\frac{1}{2}\overline{\Riem}_{\ul{L}jLm}-\frac{1}{2}\overline{\Riem}_{Lj\ul{L}m}\right)\accentset{\circ}{A}^m_i-\overline{\Riem}_{ikjm}\accentset{\circ}{A}^{mk}\right)+\left(\nabla_k\tau_i-\nabla_i\tau_k\right)\accentset{\circ}{A}^k_j\\
			&\,+\left(\overline{\Riem}_{k\ul{L}jL}-\overline{\Riem}_{j\ul{L}kL}\right)\accentset{\circ}{A}_i^k-\frac{1}{2}\left(\Riem_{ki\ul{L}L}\accentset{\circ}{A}^k_j+\overline{\Riem}_{kl\ul{L}L}\accentset{\circ}{A}_i^k\right)\\
			&\,+\frac{1}{2}\mathcal{H}^2\left(\left(\nabla_i\tau_j+\nabla_j\tau_i+2\tau_i\tau_j\right)-\left(\dive\tau+\btr{\tau}^2\right)\gamma_{ij}\right)\\
			&\,-\frac{1}{2}\mathcal{H}^2\left(\left(\overline{\Riem}_{i\ul{L}jL}+\overline{\Riem}_{j\ul{L}iL}\right)-\left(\overline{\Ric}(\ul{L},L)-\frac{1}{2}\overline{\Riem}(\ul{L},L,L,\ul{L})\right)\gamma_{ij}\right)\\
			&\,+\gamma^{kl}\left(\overline{\nabla}_i\overline{\Riem}_{kjl(\ul{\theta}L)}-\overline{\nabla}_k\overline{\Riem}_{ilj(\ul{\theta}L)}\right)+\frac{1}{2}\left(\overline{\Riem}_{i(\ul{\theta}L)j(\ul{\theta}L)}-\overline{\Ric}(\ul{\theta}L,\ul{\theta}L)\ul{\theta}^{-1}\ul{\chi}_{ij}\right).
		\end{align*}
	\end{prop}
	\begin{proof}
		First, we recall that as $\Sigma$ is a two-dimensional surface we have that
		\[
			\Riem_{kijm}=\frac{1}{2}\scal\left(\gamma_{kj}\gamma_{im}-\gamma_{ij}\gamma_{km}\right).
		\]
		Using this identity, it is straightforward to check that upon taking a trace of the Null Simons' identity, Proposition \ref{prop_prelim_nullsimon}, we obtain
		\begin{align*}
			\Delta\accentset{\circ}{A}_{ij}
			=&\,\scal\accentset{\circ}{A}_{ij}+\tau_i\dive\accentset{\circ}{A}_j-\tau^k\nabla_i\accentset{\circ}{A}_{jk}-\tau^k\nabla_k\accentset{\circ}{A}_{ij}+\nabla_k\tau_i\accentset{\circ}{A}_j^k-\nabla_i\tau_k\accentset{\circ}{A}^k_j-\dive\tau\accentset{\circ}{A}_{ij}\\
			&\,+\frac{1}{2}\mathcal{H}^2\left(\nabla_i\tau_j+\nabla_j\tau_i-\dive\tau\gamma_{ij}\right)+\gamma^{kl}\left(\nabla_i\left(\overline{\Riem}_{kjl(\ul{\theta}L)}\right)+\nabla_k\left(\overline{\Riem}_{lij(\ul{\theta}L)}\right)\right),
		\end{align*}
		where some terms cancel as $\mathcal{H}^2$ is constant along $\Sigma$. By the Gauss equation, Proposition \ref{prop_prelim_nullgauss}, we have
		\[
			R=\frac{1}{2}\mathcal{H}^2+\spann{\ul{\theta}^{-1}\accentset{\circ}{\ul{\chi}},\accentset{\circ}{A}}-\left(\overline{R}-2\overline{\Ric}(\ul{L},\ul{L})+\frac{1}{2}\overline{\Riem}(\ul{L},L,L,\ul{L})\right).
		\]
		As $\mathcal{H}^2$ is constant, we have $\nabla A=\nabla\accentset{\circ}{A}$, so by the Codazzi Equation, Proposition \ref{prop_prelim_codazziA}, we find
		\begin{align*}
			\tau_i\dive\accentset{\circ}{A}_j-\tau^k\nabla_i\accentset{\circ}{A}_{jk}-\tau^k\nabla_k\accentset{\circ}{A}_{ij}
			=&\,-2\tau^k\nabla_k\accentset{\circ}{A}_{ij}+\mathcal{H}^2\left(\tau_i\tau_j-\frac{1}{2}\btr{\tau}^2\gamma_{ij}\right)-\btr{\tau}^2\accentset{\circ}{A}_{ij}\\
			&\,+\left(\tau_i\gamma^{kl}\overline{\Riem}_{kjl(\ul{\theta}L)}-\tau^k\overline{\Riem}_{ikj(\ul{\theta}L)}\right).
		\end{align*}
		Finally, relating ambient and tensor derivatives on $\Sigma$ as in \cite[Equation (C.2)]{kroenckewolff}, one can check by direct computation that
		\begin{align*}
			&\gamma^{kl}\left(\nabla_i\left(\overline{\Riem}_{kjl(\ul{\theta}L)}\right)+\nabla_k\left(\overline{\Riem}_{lij(\ul{\theta}L)}\right)\right)\\
			&\,=\,\gamma^{kl}\left(\overline{\nabla}_i\overline{\Riem}_{kjl(\ul{\theta}L)}-\overline{\nabla}_k\overline{\Riem}_{ilj(\ul{\theta}L)}\right)
			-\tau_i\gamma^{kl}\overline{\Riem}_{kjl(\ul{\theta}L)}+\tau^k\overline{\Riem}_{ikj(\ul{\theta}L)}\\
			&\,+\left(\overline{\Ric}(\ul{L},L)-\frac{1}{2}\overline{\Riem}(\ul{L},L,L,\ul{L})\right)\accentset{\circ}{A}_{ij}
			+\frac{1}{2}\left(\overline{\Riem}_{i(\ul{\theta}L)j(\ul{\theta}L)}-\overline{\Ric}(\ul{\theta}L,\ul{\theta}L)\ul{\theta}^{-1}\ul{\chi}_{ij}\right)\\
			&\,-\frac{1}{2}\left(\Riem_{ki\ul{L}L}\accentset{\circ}{A}^k_j+\overline{\Riem}_{kl\ul{L}L}\accentset{\circ}{A}_i^k\right)+\left(\gamma^{kl}\overline{\Riem}_{kjlm}\accentset{\circ}{A}^m_i-\overline{\Riem}_{ikjm}\accentset{\circ}{A}^{mk}\right)\\
			&\,+\left(\overline{\Riem}_{k\ul{L}jL}-\overline{\Riem}_{j\ul{L}kL}\right)\accentset{\circ}{A}_i^k\\
			&\,-\frac{1}{2}\mathcal{H}^2\left(\left(\overline{\Riem}_{i\ul{L}jL}+\overline{\Riem}_{j\ul{L}iL}\right)-\left(\overline{\Ric}(\ul{L},L)-\frac{1}{2}\overline{\Riem}(\ul{L},L,L,\ul{L})\right)\gamma_{ij}\right).
		\end{align*}
		Putting all of these identities together yields the claim.
	\end{proof}
	
	This form of the contracted Null Simons identity now yields the following in an asymptotically Schwarzschildean lightcone:
	
	\begin{prop}\label{prop_appendix_NullSimonAsymSchwarzschild}
		Let $\mathcal{N}$ be an asymptotically Schwarzschildean lightcone of mass $m>0$, $\Sigma\subseteq \mathcal{N}$ an STCMC surface. Then
		\[
			\Delta\accentset{\circ}{A}_{ij}=f\accentset{\circ}{A}_{ij}-2\tau^k\nabla_k\accentset{\circ}{A}_{ij}+\left(F\ast\accentset{\circ}{A}\right)_{ij}+G_{ij},
		\]
		where
		\begin{align*}
			f=&\,\frac{1}{2}\mathcal{H}^2+\spann{\ul{\theta}^{-1}\accentset{\circ}{\ul{\chi}},\accentset{\circ}{A}}-\left(\overline{R}-3\overline{\Ric}(\ul{L},L)+\overline{\Riem}(\ul{L},L,L,\ul{L})\right)+\frac{4m}{\omega^3}-\dive\tau-\btr{\tau}^2_\gamma,\\
			G_{ij}=&\,\left(\frac{44m}{\omega^4}\ul{\theta}-\frac{12m}{\omega^3}\ul{\theta}^2\right)\accentset{\circ}{\left(\d\omega\otimes\d\omega\right)}_{ij}+\frac{1}{2}\mathcal{H}^2{\left(\accentset{\circ}{\nabla_i\tau_j}+\accentset{\circ}{\nabla_j\tau_i}+2\accentset{\circ}{\tau_i\tau_j}\right)}\\
			&\,+\mathcal{H}^2\left(O_{2,2}(r^{-2})_{IJ}+\nabla\omega^AO_{2,2}(r^{-1})_{IJA}+\nabla\omega^A\nabla\omega^BO_{2,2}(1)_{IJAB}\right)\\
			&\,+O_{1,1}(r^{-4})_{IJ}+\nabla\omega^AO_{1,1}(r^{-3})_{IJA}+\nabla\omega^A\nabla\omega^BO_{1,1}(r^{-2})_{IJAB}\\
			&\,+\nabla\omega^A\nabla\omega^B\nabla^CO_{1,1}(r^{-1})_{IJABC}+\nabla\omega^A\nabla\omega^B\nabla\omega^C\nabla\omega^DO_{1,1}(1)_{IJABCD},
		\end{align*}
		and where $F_1\ast \accentset{\circ}{A}$ is a sum of contractions such that
		\begin{align*}
			\newbtr{F\ast\accentset{\circ}{A}}&\le C\left(\newbtr{T^{(2)}}+\newbtr{T^{(4)}}+\newbtr{   \nabla\tau}\right)\newbtr{\accentset{\circ}{A}},\\
			\newbtr{\nabla(F\ast\accentset{\circ}{A})}&\le C\left(\left(\newbtr{\nabla T^{(2)}}+\newbtr{\nabla T^{(4)}}+\newbtr{\nabla^2\tau}\right)\newbtr{\accentset{\circ}{A}}+\left(\newbtr{T^{(2)}}+\newbtr{T^{(4)}}+\newbtr{\nabla\tau}\right)\newbtr{\nabla\accentset{\circ}{A}}\right)
		\end{align*}
		for tensors $T^{(2)}$, $T^{(4)}$ with 
		\begin{align*}
			T^{(2)}_{IJ}&=O_{2,2}(r^{-2})_{IJ}+\nabla\omega^AO_{2,2}(r^{-1})_{IJA}+\nabla\omega^A\nabla\omega^BO_{2,2}(1)_{IJAB},\\
			T^{(4)}_{IJKL}&=O_{2,2}(1)_{IJKL}+\nabla\omega^AO_{2,2}(r)_{IJKLA}+\nabla\omega^A\nabla\omega^BO_{2,2}(r^{2})_{IJKLAB}.
		\end{align*}
	\end{prop}
	\begin{proof}
		Using the contracted Simons identity, Proposition \ref{prop_appendix_contractedNullSimons}, it is immediate that
		\[
			\Delta\accentset{\circ}{A}_{ij}=\widetilde{f}\accentset{\circ}{A}_{ij}-2\tau^k\nabla_k\accentset{\circ}{A}_{ij}+\left(\widetilde{F}\ast\accentset{\circ}{A}\right)_{ij}+G_{ij},
		\]
		which is already of the structure that we claim. Using the curvature identities, Proposition \ref{prop_appendix_curvatureidentities}, the expansion for $G$ follows from a lengthy, but direct computation. Note that in the pure Schwarzschild case
		\[
			\widetilde{f}\accentset{\circ}{A}_{ij}+\left(\widetilde{F}\ast\accentset{\circ}{A}\right)_{ij}={f}\accentset{\circ}{A}_{ij}.
		\]
		Using Proposition \ref{prop_appendix_curvatureidentities} again to expand out the curvature terms and identify the leading order term, we find
		\[
			\left(\widetilde{F}\ast\accentset{\circ}{A}\right)_{ij}=\frac{4m}{\omega^3}\accentset{\circ}{A}_{ij}+\left({F}\ast\accentset{\circ}{A}\right)_{ij},
		\]
		with ${F}\ast\accentset{\circ}{A}$ as claimed.
	\end{proof}
	
	\section{$C^0$-estimates for almost round surfaces under a balancing condition}\label{appendix_roundness}
	
	The main goal of this section is to prove Proposition \ref{prop_gausscurvature_c0}:
	
	\begin{prop}\label{prop_appendix_gausscurvature_c0}
		Let $\gamma=\omega^2\hatgamma$ be a conformally round with Gauss curvature $K$ such that 
		\[
		\int_{\Sbb^2}\omega^2\d\widehat{\mu}=4\pi\text{, and } \int_{\Sbb^2} f_i\omega^3\d\widehat{\mu}=0
		\] 
		for $i=1,2,3$. For every $\delta>0$, there exists $\varepsilon>0$ such that if
		\[
		\norm{K-1}_{C^0(\,\hatgamma\,)}\le \varepsilon
		\]
		then $\btr{u}\le \delta$, where $u:=\ln\omega$.
	\end{prop}
	
	To this end, we first establish a uniform $W^{1,2}(\widehat{\gamma})$ bound on $u$ as in \cite[Chpater 2.4]{klainermanszeftel}. In fact, we closely follow the strategy outlined in \cite[Chapter 2 and 3]{klainermanszeftel}, highlighting the necessary modifications.
	
	Recall the Onofri functional
	\[
		S[u]:=\fint_{\Sbb^2}\newbtr{\widehat\nabla u}^2+2u\d\widehat{\mu},
	\]
	see \cite{onofri}. It is a well-known fact that $S[u]$ is invariant under 
	\[
		u\mapsto u_\Phi:=u\circ\Phi+\frac{1}{2}\ln\left(\btr{\det D\Phi}\right),
	\]
	where $\Phi$ is a M\"obius transformation. In particular, $S[u]$ with $u=\ln\omega$ is invariant under Lorentz transformations acting on $\omega$ as described in Section \ref{subsec_prelim_quantativeestimates}.
	
	As a consequence, we may always assume that $u$ satisfies the balancing condition
	\begin{align}\label{eq_appendix_balancing}
		\fint_{\Sbb^2}f_ie^{3u}\d\widehat{\mu}=0
	\end{align}
	for first spherical harmonics $f_i$,  $i=1,2,3$, while leaving $S[u]$ invariant.
	
	As in \cite{klainermanszeftel}, the key ingredient to the desired $W^{1,2}$-control is the following improved Onofri-inequality:
	
	\begin{lem}\label{lem_appendix_improvedonofri}
		There exists an $a<1$ such that for all $u\in W^{1,2}(\widehat{\gamma})$ satisfying \eqref{eq_appendix_balancing}, we have
		\[
			\fint_{\Sbb^2} e^{2u}\d\widehat{\mu}\le \exp\left(a\fint_{\Sbb^2}\newbtr{\widehat{\nabla} u}^2\d\widehat{\mu}+2\fint_{\Sbb^2}u\d\widehat{\mu}\right).
		\]
	\end{lem}
	We postpone the proof of Lemma \ref{lem_appendix_improvedonofri} to the end of this section and quickly outline how this first yields the desired $W^{1,2}$-control.
	
	\begin{kor}\label{kor_appendix_l2_balanced}
		Assume that $u\in W^{1,2}(\widehat{\gamma})$ satisfies \eqref{eq_appendix_balancing} and
		\[
			\fint_{\Sbb^2} e^{2u}\d\widehat{\mu}=1.
		\]
		Then 
		\[
			\int_{\Sbb^2}\newbtr{\widehat{\nabla}u}^2\d\widehat{\mu}\le \frac{4\pi}{a}S[u]
		\]
	\end{kor}
	\begin{proof}
		By Lemma \ref{lem_appendix_improvedonofri}, we have
		\[
			\exp\left(a\fint_{\Sbb^2}\btr{\widehat{\nabla} u}^2\d\widehat{\mu}+2\fint_{\Sbb^2}u\d\widehat{\mu}\right)\ge 1
		\]
		by assumption. Hence
		\[
			a\fint_{\Sbb^2}\newbtr{\widehat{\nabla} u}^2\d\widehat{\mu}+2\fint_{\Sbb^2}u\d\widehat{\mu}\ge 0.
		\]
		In particular
		\[
			(1-a)\fint \newbtr{\widehat{\nabla}u}^2\d\widehat{\mu}=S[u]-a\fint_{\Sbb^2}\newbtr{\widehat{\nabla} u}^2\d\widehat{\mu}-2\fint_{\Sbb^2}u\d\widehat{\mu}\le S[u].
		\]
	\end{proof}
	\begin{prop}\label{prop_appendix_w12_balancing}
		For all $\delta>0$, there exists $\varepsilon>0$ such that if 
		\[
			\norm{K-1}_{C^0(\widehat{\gamma})}\le \varepsilon,
		\]
		where $K$ denotes the Gauss curvature of the metric $e^{2u}\widehat{\gamma}$, and $u$ satisfies \eqref{eq_appendix_balancing} and 
		\[
		\fint_{\Sbb^2} e^{2u}\d\widehat{\mu}=1,
		\]
		then 
		\[
			\norm{u}_{W^{1,2}(\widehat{\gamma})}\le \delta.
		\]
	\end{prop}
	\begin{proof}
		Using \cite[Corollary 2.16]{klainermanszeftel}, the proof follows verbatim as in \cite[Corollary 2.17]{klainermanszeftel}.
	\end{proof}
	Arguing as in the proof of \cite[Proposition 3.17]{klainermanszeftel} up to some minor modifications, Proposition \ref{prop_appendix_w12_balancing} readily implies Proposition \ref{prop_appendix_gausscurvature_c0}. It remains to prove Lemma \ref{lem_appendix_improvedonofri}, i.e., an improved Onofri-Inequality under the balancing condition \eqref{eq_appendix_balancing}. Here, we closely follow the strategy of Chang--Yang in \cite{changyang}. We first prove a preliminary version:
	
	\begin{lem}\label{lem_appendix_onofri_prelim}
		For all $\varepsilon>0$, there exists a constant $C_\varepsilon>0$ such that for all $u\in W^{1,2}(\widehat{\gamma})$ satisfying \eqref{eq_appendix_balancing} it holds that
		\[
			\fint_{\Sbb^2} e^{2u}\d\widehat{\mu}\le C_\varepsilon\exp\left(\left(\frac{3}{4}+\varepsilon\right)\fint_{\Sbb^2}\newbtr{\widehat{\nabla}{u}}^2\d\widehat{\mu}+2\fint_{\Sbb^2}u\d\widehat{\mu}\right).
		\]
	\end{lem}
	\begin{proof}
		Applying \cite[Proposition 2.3]{klainermanszeftel} to $\frac{3}{2}u$, we find
		\begin{align*}
			\fint_{\Sbb^2} e^{3u}\d\widehat{\mu}&\le C_\varepsilon\exp\left(\frac{9}{4}\left(\frac{1}{2}+\varepsilon\right)\fint_{\Sbb^2}\newbtr{\widehat{\nabla}{u}}^2\d\widehat{\mu}+\frac{3}{2}\cdot2\fint_{\Sbb^2}u\d\widehat{\mu}\right)\\
			&=C_\varepsilon\exp\left(\left(\frac{3}{4}+\frac{3}{2}\varepsilon\right)\fint_{\Sbb^2}\newbtr{\widehat{\nabla}{u}}^2\d\widehat{\mu}+2\fint_{\Sbb^2}u\d\widehat{\mu}\right)^{\frac{3}{2}}
		\end{align*}
		Using the H\"older inequality, we find
		\[
			\fint_{\Sbb^2}e^{2u}\le \left(\fint_{\Sbb^2}e^{3u}\right)^{\frac{2}{3}}.
		\]
		This concludes the proof.
	\end{proof}
	
	We are now ready to prove Lemma \ref{lem_appendix_improvedonofri}:
	\begin{proof}[Proof of Lemma \ref{lem_appendix_improvedonofri}]
		We argue as in \cite[\S 3]{changyang} and consider the functional
		\[
			J_a(u)=\ln\left(\fint_{\Sbb^2}e^{2u}\d\widehat{\mu}\right)-\left(a\fint_{\Sbb^2}\newbtr{\widehat{\nabla}u}^2\d\widehat{\mu}+2\fint_{\Sbb^2}u\d\widehat{\mu}\right)
		\]
		for $\frac{3}{4}<a<1$. Let $M_a$ denote the maximum of $J_a(u)$ for all $u\in W^{1,2}(\widehat{\gamma})$ satisfying \eqref{eq_appendix_balancing}. Arguing as in \cite{changyang}, and using Lemma \ref{lem_appendix_onofri_prelim}, the maximum $M_a$ is achieved by a function $u_a$ for all $a>\frac{3}{4}$. 
		
		Moreover, for each $\eta>0$, there exists a constant $C_\eta$ such that
		\begin{align}\label{eq_appendix_onofrilemma1}
			\fint_{\Sbb^2}\newbtr{\widehat{\nabla} u_a}^2\d\widehat{\mu}\le C_\eta
		\end{align}
		for all $\frac{3}{4}+\eta\le a\le 1$, and $u_a$ (weakly) solves the corresponding Euler-Lagrange equation under the constraint \eqref{eq_appendix_balancing}, i.e.,
		\begin{align}\label{eq_appendix_onofrilemma2}
			a\widehat{\Delta}u_a+e^{2u_a}=1+\sum\limits_{i=1}^3\alpha_i f_ie^{3u_a}
		\end{align}
		for constants $\alpha_i$, where $f_i$ denote first spherical harmonics as above. We now claim that $u_a=0$ is the only solution satisfying \eqref{eq_appendix_onofrilemma1} and \eqref{eq_appendix_onofrilemma2} for $a$ sufficiently close to $1$, see Lemma \ref{lem_appendix_solution} below. Then $M_a=0$ for all $1\ge a>1-\delta$ for some sufficiently small $\delta>0$, implying the claim.
	\end{proof}
	
	\begin{lem}\label{lem_appendix_solution}
		Let $u_a$ be a solution of \eqref{eq_appendix_onofrilemma2} satisfying \eqref{eq_appendix_onofrilemma1}. Then
		\begin{enumerate}
			\item[\emph{(a)}] $\alpha_i=0$ for all $i=1,2,3$,
			\item[\emph{(b)}] $u_a=0$ for $a$ sufficiently close to $1$.
		\end{enumerate}
	\end{lem}
	\begin{proof}
		
		To prove (a), we recall a general Kazdan--Warner identity, see \cite[Equation (2.11)]{changyang}: If $v$ satisfies $\widehat{\Delta}v+he^v=c$ for a function $h$ and constant $c$, then
		\begin{align}\label{eq_appendix_generalKW}
			\fint_{\Sbb^2}e^v\spann{\widehat{\nabla}h,\widehat{\nabla}f_i}\d\widehat{\mu}=(2-c)\fint_{\Sbb^2} e^v f_i,
		\end{align}
		for first spherical harmonics $f_i$, $i=1,2,3$. We define $f:=\sum\limits_{i=1}^3\alpha_i f_i$. As $u_a$ satisfies \eqref{eq_appendix_onofrilemma2}, we find that $v_1=2u_a$, $v_2=3u_a$ satisfy
		\begin{align*}
			\widehat{\Delta}v_1+e^{v_1}h_1&=\frac{2}{a},\\
			\widehat{\Delta}v_2+e^{v_2}h_2&=\frac{3}{a},\\
		\end{align*}
		for $h_1=\frac{2}{a}(1-e^{u_a}f)$, $h_2=\frac{3}{a}(e^{-u_a}-f)$. In the following, we write $u=u_a$ for simplicity. 
		
		We first observe that
		\begin{align}\label{eq_appendix_solution1}
			\fint_{\Sbb^2}e^{v_1}h_1f\d\widehat{\mu}=\frac{2}{3}\fint_{\Sbb^2}e^{v_2}h_2f\d\widehat{\mu}.
		\end{align}
		Next, we compute that
		\begin{align*}
			\fint_{\Sbb^2}e^{v_1}\spann{\widehat{\nabla}h_1,\widehat{\nabla}f}\d\widehat{\mu}
			&=-\frac{2}{a}\fint_{\Sbb^2}e^{2u}\spann{\widehat{\nabla}e^uf,\widehat{\nabla}f}\d\widehat{\mu}\\
			&=-\frac{2}{a}\fint_{\Sbb^2}e^{3u}\newbtr{\widehat{\nabla}{f}}^2+\frac{1}{6}\spann{\widehat{\nabla}e^{3u},\widehat{\nabla}f^2}\d\widehat{\mu}\\
			&=-\frac{2}{a}\fint_{\Sbb^2}e^{3u}\newbtr{\widehat{\nabla}{f}}^2-\frac{1}{6}e^{3u}\widehat{\Delta}f^2\d\widehat{\mu}\\
			&=-\frac{2}{a}\fint_{\Sbb^2}e^{3u}\left(\frac{2}{3}\newbtr{\widehat{\nabla}{f}}^2-\frac{1}{3}e^{3u}f\widehat{\Delta}f\right)\d\widehat{\mu}\\
			&=-\frac{4}{3a}\fint_{\Sbb^2}e^{3u}\left(\newbtr{\widehat{\nabla}{f}}^2+f^2\right)\d\widehat{\mu},
		\end{align*}
		where we used that $f$ is a linear combination of first spherical harmonics, i.e., $-\widehat{\Delta}f=2f$. A similar integration by parts yields
		\[
			\fint_{\Sbb^2}e^{v_2}\spann{\widehat{\nabla}h_2,\widehat{\nabla}f}\d\widehat{\mu}=-\frac{3}{a}\fint_{\Sbb^2}e^{3u}\newbtr{\widehat{\nabla}{f}}^2\d\widehat{\mu}-\frac{3}{a}\fint_{\Sbb^2}e^{2u}f\d\widehat{\mu},
		\]
		and we observe that
		\[
			\fint_{\Sbb^2}e^{v_2}h_2f\d\widehat{\mu}=\frac{3}{a}\fint_{\Sbb^2}e^{2u}f\d\widehat{\mu}-\frac{3}{a}\fint_{\Sbb^2}e^{3u}f^2\d\widehat{\mu}.
		\]
		Combining the two identities derived by partial integration, the generalized Kazdan--Warner identity \eqref{eq_appendix_generalKW}, and \eqref{eq_appendix_solution1}, we find that
		\begin{align*}
			-\frac{3}{a}\fint_{\Sbb^2}e^{3u}\left(\newbtr{\widehat{\nabla}{f}}^2+f^2\right)\d\widehat{\mu}
			&=\fint_{\Sbb^2}e^{v_2}\spann{\widehat{\nabla}h_2,\widehat{\nabla}f}\d\widehat{\mu}+\fint_{\Sbb^2}e^{v_2}h_2f\d\widehat{\mu}\\
			&=\left(3-\frac{3}{a}\right)\fint_{\Sbb^2}e^{v_2}h_2f\d\widehat{\mu}\\
			&=\left(2-\frac{2}{a}\right)\fint_{\Sbb^2}e^{v_2}h_1f\d\widehat{\mu}\\
			&=\fint_{\Sbb^2}e^{v_1}\spann{\widehat{\nabla}h_1,\widehat{\nabla}f}\d\widehat{\mu}\\
			&=-\frac{4}{3a}\fint_{\Sbb^2}e^{3u}\left(\newbtr{\widehat{\nabla}{f}}^2+f^2\right)\d\widehat{\mu}.
		\end{align*}
		Hence
		\[
			\fint_{\Sbb^2}e^{3u}\left(\newbtr{\widehat{\nabla}{f}}^2+f^2\right)\d\widehat{\mu}=0,
		\]
		so $f=0$. As the first spherical harmonics $f_i$ are $L^2$-orthogonal this implies $\alpha_i=0$ as desired. This concludes the proof of (a).
		
		To prove (b), we now note that it suffices to show that $u_a=0$ is the only solution to
		\[
		a\widetilde{\Delta}u_a+e^{2u_a}-1=0
		\]
		satisfying \eqref{eq_appendix_balancing} for $a$ sufficiently close to $1$ by (a). A proof of this follows up to some modifications as in \cite{changyang}. For completeness, we include a different argument using the implicit function theorem. Here, we only invoke that $\norm{u_a}_{C^0(\widehat{\gamma})}\to 0$ as $a\to 1$ which follows almost verbatim as in \cite[Lemma 3.2]{changyang}. Consider the map
		\[
			\mathcal{F}\colon(0,\infty)\times W^{2,2}(\widehat{\gamma})\to L^2(\widehat{\gamma})\times \R^3,\,(a,u) \mapsto\left(a\widehat{\Delta}u+e^{2u}-1,\int_{\Sbb^2}e^{3u}f_i\d\widehat{\mu}\right)
		\]
		By construction, (b) is equivalent to showing that $\mathcal{F}(a,u_a)=(0,0)$ implies $u_a=0$ for $a$ sufficiently close to $1$. Observe that
		\[
			D_u\mathcal{F}_{(a,u)}(\varphi)=\left(a\widehat{\Delta}\varphi+2e^{2u}\varphi,3\int_{\Sbb^2}e^{3u}\varphi f_i\d\widehat{\mu}\right),
		\]
		in particular $D_u\mathcal{F}$ only depends pointwise on $(a,u)$. At $(1,0)$, which satisfies $\mathcal{F}(1,0)=(0,0)$, one finds
		\[
			D_u\mathcal{F}_{(1,0)}(\varphi)=\left(L(\varphi):=\widehat{\Delta}\varphi+2\varphi,3\int_{\Sbb^2}\varphi f_i\d\widehat{\mu}\right).
		\]
		Note that the kernel of $L$ is spanned by the first spherical harmonics, and the last three components are the projection of $\varphi$ onto that span. Hence, $	D_u\mathcal{F}_{(1,0)}$ is invertible by construction, and the implicit function theorem implies that for $a$ sufficiently close to $1$ there is a $W^{2,2}(\widehat{\gamma})$-neighborhood of $0$ in which there is a unique solution $u_a$, which is necessarily $u_a=0$. Moreover, since all solutions $u_a$ satisfy $\norm{u_a}_{C^0(\widehat{\gamma})}\to 0$ as $a\to 1$, by continuity $D_u\mathcal{F}_{(a,u_a)}$ remains invertible for $a$ sufficiently small. In particular, the implicit function theorem also yields a (continuous) path from $u_a$ to a solution $u_1$ for $a=1$. Since $u_1=0$ is the unique solution for $a=1$ and the local uniqueness of the implicit function theorem prevents the branching of paths of solutions, $u_a=0$ for $a$ sufficiently small. 
	\end{proof}

	\bibliography{bib_schwarzschildlightcones}

@article {cederbaumsakovich,
    AUTHOR = {Cederbaum, Carla and Sakovich, Anna},
     TITLE = {{On center of mass and foliations by constant spacetime mean
              curvature surfaces for isolated systems in general relativity}},
   JOURNAL = {Calc. Var. Partial Differential Equations},
  FJOURNAL = {Calculus of Variations and Partial Differential Equations},
    VOLUME = {60},
      YEAR = {2021},
    NUMBER = {6},
     PAGES = {Paper No. 214, 57},
      ISSN = {0944-2669,1432-0835},
   MRCLASS = {53C21 (58J37 83C05)},
  MRNUMBER = {4305436},
MRREVIEWER = {Ettore\ Minguzzi},
       DOI = {10.1007/s00526-021-02060-z},
       URL = {https://doi.org/10.1007/s00526-021-02060-z},
}

@article {anderssonmetzger,
    AUTHOR = {Andersson, Lars and Metzger, Jan},
     TITLE = {{Curvature estimates for stable marginally trapped surfaces}},
   JOURNAL = {J. Differential Geom.},
  FJOURNAL = {Journal of Differential Geometry},
    VOLUME = {84},
      YEAR = {2010},
    NUMBER = {2},
     PAGES = {231--265},
      ISSN = {0022-040X,1945-743X},
   MRCLASS = {53C42 (53C50 53C80)},
  MRNUMBER = {2652461},
MRREVIEWER = {John\ Urbas},
       URL = {http://projecteuclid.org/euclid.jdg/1274707313},
}

@article {kroenckewolff,
    AUTHOR = {Kr\"oncke, Klaus and Wolff, Markus},
     TITLE = {{Foliations of asymptotically {S}chwarzschildean lightcones by
              surfaces of constant spacetime mean curvature}},
   JOURNAL = {Math. Ann.},
  FJOURNAL = {Mathematische Annalen},
    VOLUME = {394},
      YEAR = {2026},
    NUMBER = {3},
     PAGES = {Paper No. 73, 71},
      ISSN = {0025-5831,1432-1807},
   MRCLASS = {53C21 (53C50 53E10)},
  MRNUMBER = {5036070},
       DOI = {10.1007/s00208-026-03331-w},
       URL = {https://doi.org/10.1007/s00208-026-03331-w},
}

@phdthesis{wolff_thesis,
  title        = {{On the Spacetime Mean Curvature of Surfaces in General Relativity}},
  author       = {Markus Wolff},
  year         = {2023},
  month        = {December},
  address      = {T\"ubingen, Germany},
  note         = {Available at \url{https://publikationen.uni-tuebingen.de/xmlui/handle/10900/148726}},
  school       = {University of T\"ubingen},
  type         = {PhD thesis}
}

@phdthesis{sauter,
title = {{Foliations of null hypersurfaces and the Penrose inequality}},
author = {Johannes Sauter},
year = {2008},
address = {Z\"urich, Switzerland},
note = {Available at \url{https://doi.org/10.3929/ethz-a-005713669}},
school = {ETH Z\"urich},
type = {PhD thesis}
}

@article {wolff4,
    AUTHOR = {Wolff, Markus},
     TITLE = {A {D}e {L}ellis--{M}\"uller type estimate on the {M}inkowski
              lightcone},
   JOURNAL = {Calc. Var. Partial Differential Equations},
  FJOURNAL = {Calculus of Variations and Partial Differential Equations},
    VOLUME = {63},
      YEAR = {2024},
    NUMBER = {7},
     PAGES = {Paper No. 185, 30},
      ISSN = {0944-2669,1432-0835},
   MRCLASS = {53C21 (53C50 53E10)},
  MRNUMBER = {4775719},
MRREVIEWER = {Francisco\ J.\ Palomo},
       DOI = {10.1007/s00526-024-02784-8},
       URL = {https://doi.org/10.1007/s00526-024-02784-8},
}

@article {wolff1,
    AUTHOR = {Wolff, Markus},
     TITLE = {Ricci flow on surfaces along the standard lightcone in the
              {$3+1$}-{M}inkowski spacetime},
   JOURNAL = {Calc. Var. Partial Differential Equations},
  FJOURNAL = {Calculus of Variations and Partial Differential Equations},
    VOLUME = {62},
      YEAR = {2023},
    NUMBER = {3},
     PAGES = {Paper No. 90, 22},
      ISSN = {0944-2669,1432-0835},
   MRCLASS = {53E20 (35Q75)},
  MRNUMBER = {4541080},
MRREVIEWER = {Julian\ Scheuer},
       DOI = {10.1007/s00526-022-02415-0},
       URL = {https://doi.org/10.1007/s00526-022-02415-0},
}

@article {onofri,
    AUTHOR = {Onofri, E.},
     TITLE = {On the positivity of the effective action in a theory of
              random surfaces},
   JOURNAL = {Comm. Math. Phys.},
  FJOURNAL = {Communications in Mathematical Physics},
    VOLUME = {86},
      YEAR = {1982},
    NUMBER = {3},
     PAGES = {321--326},
      ISSN = {0010-3616,1432-0916},
   MRCLASS = {58E15 (53C21 58G40)},
  MRNUMBER = {677001},
MRREVIEWER = {J.\ L.\ Kazdan},
       URL = {http://projecteuclid.org/euclid.cmp/1103921772},
}

@article {chenwang,
    AUTHOR = {Chen, Po-Ning and Wang, Ye-Kai},
     TITLE = {Two rigidity results for surfaces in {S}chwarzschild
              spacetimes},
   JOURNAL = {Math. Res. Lett.},
  FJOURNAL = {Mathematical Research Letters},
    VOLUME = {32},
      YEAR = {2025},
    NUMBER = {5},
     PAGES = {1373--1397},
      ISSN = {1073-2780,1945-001X},
   MRCLASS = {53C24 (53C42 53C50)},
  MRNUMBER = {5009213},
       DOI = {10.4310/mrl.251214113954},
       URL = {https://doi.org/10.4310/mrl.251214113954},
}

@article {eichmairkoerbermetzgerschulze,
    AUTHOR = {Eichmair, Michael and Koerber, Thomas and Metzger, Jan and
              Schulze, Felix},
     TITLE = {Huisken-{Y}au-type uniqueness for area-constrained {W}illmore
              spheres},
   JOURNAL = {Duke Math. J.},
  FJOURNAL = {Duke Mathematical Journal},
    VOLUME = {173},
      YEAR = {2024},
    NUMBER = {9},
     PAGES = {1677--1730},
      ISSN = {0012-7094,1547-7398},
   MRCLASS = {53C42 (49Q10 83C40)},
  MRNUMBER = {4766841},
MRREVIEWER = {Peng\ Wang},
       DOI = {10.1215/00127094-2023-0045},
       URL = {https://doi.org/10.1215/00127094-2023-0045},
}

@article {huiskenyau,
    AUTHOR = {Huisken, Gerhard and Yau, Shing-Tung},
     TITLE = {Definition of center of mass for isolated physical systems and
              unique foliations by stable spheres with constant mean
              curvature},
   JOURNAL = {Invent. Math.},
  FJOURNAL = {Inventiones Mathematicae},
    VOLUME = {124},
      YEAR = {1996},
    NUMBER = {1-3},
     PAGES = {281--311},
      ISSN = {0020-9910,1432-1297},
   MRCLASS = {53C20 (83C99)},
  MRNUMBER = {1369419},
MRREVIEWER = {Alan\ D.\ Rendall},
       DOI = {10.1007/s002220050054},
       URL = {https://doi.org/10.1007/s002220050054},
}

@article {roesch,
    AUTHOR = {Roesch, Henri P.},
     TITLE = {Proof of a null {P}enrose conjecture using a new quasi-local
              mass},
   JOURNAL = {Comm. Anal. Geom.},
  FJOURNAL = {Communications in Analysis and Geometry},
    VOLUME = {29},
      YEAR = {2021},
    NUMBER = {8},
     PAGES = {1847--1915},
      ISSN = {1019-8385,1944-9992},
   MRCLASS = {53C20 (53C50)},
  MRNUMBER = {4429247},
MRREVIEWER = {Marcus\ A.\ Khuri},
       DOI = {10.4310/cag.2021.v29.n8.a5},
       URL = {https://doi.org/10.4310/cag.2021.v29.n8.a5},
}

@article {roesch2,
    AUTHOR = {Roesch, Henri},
     TITLE = {Quasi-round {MOTS}s and stability of the {S}chwarzschild null
              {P}enrose inequality},
   JOURNAL = {Ann. Henri Poincar\'e},
  FJOURNAL = {Annales Henri Poincar\'e. A Journal of Theoretical and
              Mathematical Physics},
    VOLUME = {22},
      YEAR = {2021},
    NUMBER = {6},
     PAGES = {1937--1978},
      ISSN = {1424-0637,1424-0661},
   MRCLASS = {83C30 (83C40 83C57 83C75)},
  MRNUMBER = {4264889},
MRREVIEWER = {Veselin\ T.\ Videv},
       DOI = {10.1007/s00023-021-01047-y},
       URL = {https://doi.org/10.1007/s00023-021-01047-y},
}

@article {cederbaumcortiersakovich,
    AUTHOR = {Cederbaum, Carla and Cortier, Julien and Sakovich, Anna},
     TITLE = {On the center of mass of asymptotically hyperbolic initial
              data sets},
   JOURNAL = {Ann. Henri Poincar\'e},
  FJOURNAL = {Annales Henri Poincar\'e. A Journal of Theoretical and
              Mathematical Physics},
    VOLUME = {17},
      YEAR = {2016},
    NUMBER = {6},
     PAGES = {1505--1528},
      ISSN = {1424-0637,1424-0661},
   MRCLASS = {83C05 (53C20)},
  MRNUMBER = {3500223},
MRREVIEWER = {Jos\'e\ Nat\'ario},
       DOI = {10.1007/s00023-015-0438-5},
       URL = {https://doi.org/10.1007/s00023-015-0438-5},
}

@article {marssoria,
    AUTHOR = {Mars, Marc and Soria, Alberto},
     TITLE = {The asymptotic behaviour of the {H}awking energy along null
              asymptotically flat hypersurfaces},
   JOURNAL = {Classical Quantum Gravity},
  FJOURNAL = {Classical and Quantum Gravity},
    VOLUME = {32},
      YEAR = {2015},
    NUMBER = {18},
     PAGES = {185020, 30},
      ISSN = {0264-9381,1361-6382},
   MRCLASS = {83C40 (83C05)},
  MRNUMBER = {3400227},
MRREVIEWER = {Torsten\ Asselmeyer-Maluga},
       DOI = {10.1088/0264-9381/32/18/185020},
       URL = {https://doi.org/10.1088/0264-9381/32/18/185020},
}

@article {klainermanszeftel,
    AUTHOR = {Klainerman, Sergiu and Szeftel, J\'er\'emie},
     TITLE = {Effective results on uniformization and intrinsic {GCM}
              spheres in perturbations of {K}err},
      NOTE = {With an appendix by Camillo De Lellis},
   JOURNAL = {Ann. PDE},
  FJOURNAL = {Annals of PDE. Journal Dedicated to the Analysis of Problems
              from Physical Sciences},
    VOLUME = {8},
      YEAR = {2022},
    NUMBER = {2},
     PAGES = {Paper No. 18, 89},
      ISSN = {2524-5317,2199-2576},
   MRCLASS = {53C21 (35B45 83C57)},
  MRNUMBER = {4462883},
MRREVIEWER = {Man\ Chun\ Leung},
       DOI = {10.1007/s40818-022-00132-7},
       URL = {https://doi.org/10.1007/s40818-022-00132-7},
}

@article {tenan1,
    AUTHOR = {Tenan, Jacopo},
     TITLE = {Volume preserving spacetime mean curvature flow and foliations
              of initial data sets},
   JOURNAL = {J. Funct. Anal.},
  FJOURNAL = {Journal of Functional Analysis},
    VOLUME = {290},
      YEAR = {2026},
    NUMBER = {6},
     PAGES = {Paper No. 111313, 47},
      ISSN = {0022-1236,1096-0783},
   MRCLASS = {53E10 (35B40 35K55 53C80)},
  MRNUMBER = {5002643},
       DOI = {10.1016/j.jfa.2025.111313},
       URL = {https://doi.org/10.1016/j.jfa.2025.111313},
}

@article {brendleeichmair,
    AUTHOR = {Brendle, Simon and Eichmair, Michael},
     TITLE = {Large outlying stable constant mean curvature spheres in
              initial data sets},
   JOURNAL = {Invent. Math.},
  FJOURNAL = {Inventiones Mathematicae},
    VOLUME = {197},
      YEAR = {2014},
    NUMBER = {3},
     PAGES = {663--682},
      ISSN = {0020-9910,1432-1297},
   MRCLASS = {53C42 (53C20 83C99)},
  MRNUMBER = {3251832},
MRREVIEWER = {Andrew\ Bucki},
       DOI = {10.1007/s00222-013-0494-8},
       URL = {https://doi.org/10.1007/s00222-013-0494-8},
}

@article {eichmairkoerber1,
    AUTHOR = {Eichmair, Michael and Koerber, Thomas},
     TITLE = {Foliations of asymptotically flat manifolds by stable constant
              mean curvature spheres},
   JOURNAL = {J. Differential Geom.},
  FJOURNAL = {Journal of Differential Geometry},
    VOLUME = {128},
      YEAR = {2024},
    NUMBER = {3},
     PAGES = {1037--1083},
      ISSN = {0022-040X,1945-743X},
   MRCLASS = {53C20 (53C42)},
  MRNUMBER = {4810218},
MRREVIEWER = {Juan\ A.\ Aledo},
       DOI = {10.4310/jdg/1729092454},
       URL = {https://doi.org/10.4310/jdg/1729092454},
}

@article {huang,
    AUTHOR = {Huang, Lan-Hsuan},
     TITLE = {Foliations by stable spheres with constant mean curvature for
              isolated systems with general asymptotics},
   JOURNAL = {Comm. Math. Phys.},
  FJOURNAL = {Communications in Mathematical Physics},
    VOLUME = {300},
      YEAR = {2010},
    NUMBER = {2},
     PAGES = {331--373},
      ISSN = {0010-3616,1432-0916},
   MRCLASS = {53C12 (53C42 53C80)},
  MRNUMBER = {2728728},
MRREVIEWER = {Jan\ Metzger},
       DOI = {10.1007/s00220-010-1100-1},
       URL = {https://doi.org/10.1007/s00220-010-1100-1},
}

@article {nerz,
    AUTHOR = {Nerz, Christopher},
     TITLE = {Foliations by spheres with constant expansion for isolated
              systems without asymptotic symmetry},
   JOURNAL = {J. Differential Geom.},
  FJOURNAL = {Journal of Differential Geometry},
    VOLUME = {109},
      YEAR = {2018},
    NUMBER = {2},
     PAGES = {257--289},
      ISSN = {0022-040X,1945-743X},
   MRCLASS = {53C12 (53C20 53C42 83C05)},
  MRNUMBER = {3807320},
MRREVIEWER = {A.\ Burtscher},
       DOI = {10.4310/jdg/1527040873},
       URL = {https://doi.org/10.4310/jdg/1527040873},
}

@article {ma,
    AUTHOR = {Ma, Shiguang},
     TITLE = {Uniqueness of the foliation of constant mean curvature spheres
              in asymptotically flat 3-manifolds},
   JOURNAL = {Pacific J. Math.},
  FJOURNAL = {Pacific Journal of Mathematics},
    VOLUME = {252},
      YEAR = {2011},
    NUMBER = {1},
     PAGES = {145--179},
      ISSN = {0030-8730,1945-5844},
   MRCLASS = {53C12 (53C42)},
  MRNUMBER = {2862146},
MRREVIEWER = {Lan-Hsuan\ Huang},
       DOI = {10.2140/pjm.2011.252.145},
       URL = {https://doi.org/10.2140/pjm.2011.252.145},
}

@article {metzger,
    AUTHOR = {Metzger, Jan},
     TITLE = {Foliations of asymptotically flat 3-manifolds by 2-surfaces of
              prescribed mean curvature},
   JOURNAL = {J. Differential Geom.},
  FJOURNAL = {Journal of Differential Geometry},
    VOLUME = {77},
      YEAR = {2007},
    NUMBER = {2},
     PAGES = {201--236},
      ISSN = {0022-040X,1945-743X},
   MRCLASS = {53C12 (53C20)},
  MRNUMBER = {2355784},
MRREVIEWER = {John\ Urbas},
       URL = {http://projecteuclid.org/euclid.jdg/1191860394},
}

@article {chodosheichmair,
    AUTHOR = {Chodosh, Otis and Eichmair, Michael},
     TITLE = {Global uniqueness of large stable {CMC} spheres in
              asymptotically flat {R}iemannian 3-manifolds},
   JOURNAL = {Duke Math. J.},
  FJOURNAL = {Duke Mathematical Journal},
    VOLUME = {171},
      YEAR = {2022},
    NUMBER = {1},
     PAGES = {1--31},
      ISSN = {0012-7094,1547-7398},
   MRCLASS = {53C42 (53C20)},
  MRNUMBER = {4364730},
MRREVIEWER = {Zejun\ Hu},
       DOI = {10.1215/00127094-2021-0043},
       URL = {https://doi.org/10.1215/00127094-2021-0043},
}

@article {qingtian,
    AUTHOR = {Qing, Jie and Tian, Gang},
     TITLE = {On the uniqueness of the foliation of spheres of constant mean
              curvature in asymptotically flat 3-manifolds},
   JOURNAL = {J. Amer. Math. Soc.},
  FJOURNAL = {Journal of the American Mathematical Society},
    VOLUME = {20},
      YEAR = {2007},
    NUMBER = {4},
     PAGES = {1091--1110},
      ISSN = {0894-0347,1088-6834},
   MRCLASS = {53C20 (53C12 53C80)},
  MRNUMBER = {2328717},
MRREVIEWER = {Fei-Tsen\ Liang},
       DOI = {10.1090/S0894-0347-07-00560-7},
       URL = {https://doi.org/10.1090/S0894-0347-07-00560-7},
}

@incollection {ye,
    AUTHOR = {Ye, Rugang},
     TITLE = {Foliation by constant mean curvature spheres on asymptotically
              flat manifolds},
 BOOKTITLE = {Geometric analysis and the calculus of variations},
     PAGES = {369--383},
 PUBLISHER = {Int. Press, Cambridge, MA},
      YEAR = {1996},
      ISBN = {1-57146-037-3},
   MRCLASS = {53C12},
  MRNUMBER = {1449417},
MRREVIEWER = {Vladimir\ G.\ Tkachev},
}

@article {nevestian2,
    AUTHOR = {Neves, Andr\'e{} and Tian, Gang},
     TITLE = {Existence and uniqueness of constant mean curvature foliation
              of asymptotically hyperbolic 3-manifolds. {II}},
   JOURNAL = {J. Reine Angew. Math.},
  FJOURNAL = {Journal f\"ur die Reine und Angewandte Mathematik. [Crelle's
              Journal]},
    VOLUME = {641},
      YEAR = {2010},
     PAGES = {69--93},
      ISSN = {0075-4102,1435-5345},
   MRCLASS = {53C12 (53C21 53C80)},
  MRNUMBER = {2643925},
MRREVIEWER = {Jesse\ Ratzkin},
       DOI = {10.1515/CRELLE.2010.028},
       URL = {https://doi.org/10.1515/CRELLE.2010.028},
}

@article {nevestian1,
    AUTHOR = {Neves, Andr\'e{} and Tian, Gang},
     TITLE = {Existence and uniqueness of constant mean curvature foliation
              of asymptotically hyperbolic 3-manifolds},
   JOURNAL = {Geom. Funct. Anal.},
  FJOURNAL = {Geometric and Functional Analysis},
    VOLUME = {19},
      YEAR = {2009},
    NUMBER = {3},
     PAGES = {910--942},
      ISSN = {1016-443X,1420-8970},
   MRCLASS = {53C12 (53C21 53C80)},
  MRNUMBER = {2563773},
MRREVIEWER = {Jesse\ Ratzkin},
       DOI = {10.1007/s00039-009-0019-1},
       URL = {https://doi.org/10.1007/s00039-009-0019-1},
}

@article {lammmetzgerschulze,
    AUTHOR = {Lamm, Tobias and Metzger, Jan and Schulze, Felix},
     TITLE = {Foliations of asymptotically flat manifolds by surfaces of
              {W}illmore type},
   JOURNAL = {Math. Ann.},
  FJOURNAL = {Mathematische Annalen},
    VOLUME = {350},
      YEAR = {2011},
    NUMBER = {1},
     PAGES = {1--78},
      ISSN = {0025-5831,1432-1807},
   MRCLASS = {53C12 (53C24)},
  MRNUMBER = {2785762},
MRREVIEWER = {Jesse\ Ratzkin},
       DOI = {10.1007/s00208-010-0550-2},
       URL = {https://doi.org/10.1007/s00208-010-0550-2},
}

@misc{tenan2,
      title={Foliations by constant spacetime mean curvature surfaces for asymptotically hyperboloidal initial data sets}, 
      author={Jacopo Tenan},
      year={2026},
      eprint={2607.02244},
      archivePrefix={arXiv},
      primaryClass={math.DG},
      url={https://arxiv.org/abs/2607.02244}, 
	howpublished = {arXiv:2607.02244 (preprint)}
}

@misc{lambertscheuer,
      title={{Foliation of null cones by surfaces of constant spacetime mean curvature near MOTS}}, 
      author={Ben Lambert and Julian Scheuer},
      year={2026},
      eprint={2603.23083},
      archivePrefix={arXiv},
      primaryClass={math.DG},
      url={https://arxiv.org/abs/2603.23083}, 
	howpublished = {arXiv:2603.23083 (preprint)}
}

@article{maedlerwinicour,
   title={Bondi-Sachs Formalism},
   volume={11},
   ISSN={1941-6016},
   url={http://dx.doi.org/10.4249/scholarpedia.33528},
   DOI={10.4249/scholarpedia.33528},
   number={12},
   journal={Scholarpedia},
   publisher={Scholarpedia},
   author={M\"adler, Thomas and Winicour, Jeffrey},
   year={2016},
   pages={33528} }

@article {schoensimonyau,
    AUTHOR = {Schoen, R. and Simon, L. and Yau, S. T.},
     TITLE = {Curvature estimates for minimal hypersurfaces},
   JOURNAL = {Acta Math.},
  FJOURNAL = {Acta Mathematica},
    VOLUME = {134},
      YEAR = {1975},
    NUMBER = {3-4},
     PAGES = {275--288},
      ISSN = {0001-5962,1871-2509},
   MRCLASS = {53C40 (49F10)},
  MRNUMBER = {423263},
MRREVIEWER = {Robert\ Gulliver},
       DOI = {10.1007/BF02392104},
       URL = {https://doi.org/10.1007/BF02392104},
}

@misc{wolff_stability,
      title={A note on the stability of surfaces along null cones under area-preserving variations},
      year={2026},
	author = {Markus Wolff},
      eprint={2607.09325},
      archivePrefix={arXiv},
      primaryClass={math.DG},
      url={https://arxiv.org/abs/2607.09325}, 
	howpublished = {arXiv:2607.09325 (preprint)}
}

@article {changyang,
    AUTHOR = {Chang, Sun-Yung Alice and Yang, Paul C.},
     TITLE = {Prescribing {G}aussian curvature on {$S^2$}},
   JOURNAL = {Acta Math.},
  FJOURNAL = {Acta Mathematica},
    VOLUME = {159},
      YEAR = {1987},
    NUMBER = {3-4},
     PAGES = {215--259},
      ISSN = {0001-5962,1871-2509},
   MRCLASS = {35J60 (53A05 58E11 58G30)},
  MRNUMBER = {908146},
MRREVIEWER = {John\ Urbas},
       DOI = {10.1007/BF02392560},
       URL = {https://doi.org/10.1007/BF02392560},
}

	\nopagebreak
	\bibliographystyle{plain}
	\,\\\\\\
	{\small
		Klaus Kr\"oncke\newline
		KTH Royal Institute of Technology \newline
		Department of Mathematics\newline
		Lindstedtsv\"agen 25\newline
		114 28 Stockholm, Sweden\newline
		{https://orcid.org/0000-0001-7933-0034}\newline
		kroncke@kth.se}\newline
	\,\\
	{\small
		Markus Wolff\newline
		University of Vienna \newline
		Faculty of Mathematics\newline
		Oskar-Morgenstern Platz 1\newline
		1090 Vienna, Austria\newline
		{https://orcid.org/0000-0002-2257-7359}\newline
		markus.wolff@univie.ac.at}
\end{document}